\documentclass[leqno,12pt]{amsart}

\usepackage{amssymb}
\usepackage{amsmath}
\usepackage[usenames, dvipsnames]{color}
\usepackage{enumitem}
\usepackage{amsthm}
\usepackage{mathrsfs}
\usepackage{amsrefs}
\numberwithin{equation}{section}
\usepackage{geometry}

\usepackage{marginnote}
\usepackage{todonotes}
\usepackage{hyperref}
\usepackage{esint}
\usepackage{verbatim}
\usepackage{hyperref}

\usepackage{etoolbox}

\theoremstyle{plain}
\newtheorem{theorem}{Theorem}[section]
\newtheorem{lemma}[theorem]{Lemma}
\newtheorem{corollary}[theorem]{Corollary}
\newtheorem{proposition}[theorem]{Proposition}

\newtheorem{definition}[theorem]{Definition}

\newtheorem{remark}[theorem]{Remark}

\newcommand{\R}{\mathbb{R}}
\newcommand{\N}{\mathbb{N}}

\newcommand{\A}{\mathcal{A}}

\makeatletter
\@namedef{subjclassname@2020}{\textup{2020} Mathematics Subject Classification}
\makeatother

\def\XXint#1#2#3{{\setbox0=\hbox{$#1{#2#3}{\int}$ }
\vcenter{\hbox{$#2#3$ }}\kern-.6\wd0}}

\title[$W^{2,\varepsilon}$-estimates, equations with singular-degenerate coefficients]{On $W^{2,\varepsilon}$-estimates for a class of singular-degenerate parabolic equations}

\author[J. Fang]{Junyuan Fang}
\address[J. Fang]{Department of Mathematics, University of Tennessee, 227 Ayres Hall,
1403 Circle Drive, Knoxville, TN 37996-1320 }
\email{jfang9@vols.utk.edu}

\author[T. Phan]{Tuoc Phan}
\address[T. Phan]{Department of Mathematics, University of Tennessee, 227 Ayres Hall,
1403 Circle Drive, Knoxville, TN 37996-1320}
\email{tphan2@utk.edu}


\begin{document}
\begin{abstract} We study a class of parabolic equations in non-divergence form with measurable coefficients that exhibit singular and/or degenerate behavior governed by weights in the $A_{1+\frac{1}{n}}$-Muckenhoupt class.  Under a smallness assumption on a weighted mean oscillation of the weights, we establish weighted $W^{2,\varepsilon}$-estimates in the spirit of F.-H. Lin. Our results particularly holds for equations whose leading coefficients are of logistic-type singularities, as well as to those with polynomial blow-up or vanishing with sufficiently small exponents. A central component of our approach is the development of local quantitative lower estimates for solutions, which are interpreted as the mean sojourn time of sample paths, a stochastic-geometric perspective that generalizes the seminal work of L. C. Evans. We address the singular-degenerate nature of the operators by employing a class of intrinsic weighted parabolic cylinders, combined with a perturbation argument and parabolic Aleksandrov-Bakelman-Pucci (ABP) estimates. Furthermore, we conduct a rigorous analysis of weight regularization and truncation to ensure that the estimates are independent of the regularization and truncation parameters.  The results extend classical regularity theory to a broad class of second-order parabolic equations and provide a functional analytic foundation for further study of fully nonlinear parabolic equations with singular-degenerate structure. \end{abstract}
\maketitle
\section{Introduction and main results} This paper establishes a weighted parabolic $W^{2,\varepsilon}$-Lin type estimate for solutions to a class of second order linear parabolic equations in non-divergence form whose leading coefficients are merely measurable and do not satisfy the uniformly elliptic condition nor the boundedness condition.  The obtained results naturally extend the classical result proved by F.-H. Lin in \cite{Lin}, and the analysis of the mean sojourn time of sample paths within sets by L. C.~Evans in \cite{Evans}, originally developed for elliptic equations with measurable coefficients satisfying uniform ellipticity and boundedness conditions.  Parabolic analogues under the uniform ellipticity and boundedness condition were later established by N.V. Krylov in \cite{Krylov-2010, Krylov-2012}.  This paper generalizes, for the first time, these results to a broader class of singular-degenerate parabolic equations, highlighting new regularity phenomena in singular and non-uniformly elliptic contexts. Our findings apply, for examples, to second-order parabolic equations whose leading coefficients exhibit logistic-type singularities, as well as to those featuring polynomial blow-up or vanishing with sufficiently small exponents. 

To put our study in perspective, let us point out that equations with singular-degenerate coefficients naturally arise in models where transport or diffusion depends strongly on the state variable, and on heterogeneous structures. In biofilms, diffusion may vanish at low density and blow up near maximal packing, producing degenerate-singular parabolic equations with nonstandard entropy structure and regularity properties, see \cite{DMZ, HM}. In compressible fluid mechanics and astrophysical flows;  for example, compressible Euler equations with physical vacuum and compressible Navier-Stokes equations with density dependent viscosity;  the governing partial differential equations feature singular or degenerate coefficients \cite{DIT, YJ, YZ}. Similarly, semiconductor drift-diffusion and energy-transport models give rise to degenerate or singular coefficients from kinetic principles \cite{MRS}. In addition, equations with singular-degenerate coefficients also arise frequently in other areas of mathematics such as probability, geometric analysis, mathematical finance, free boundary problems, and mathematical biology as discussed and studied in \cite{Dong-Phan, Cho-Fang-Phan, IS, Ji, Kim-Lee, Le, Mooney, Pop} and the references therein. Note also that the numerical analysis of singular-degenerate parabolic PDEs in porous media has been rigorously developed, including stochastic extensions, as in \cite{BGV}, for example. These phenomena motivate rigorous analysis of propagation, regularity, and coercivity.

\smallskip
To set up, let us denote the following class of linear parabolic operators in non-divergence form
\begin{equation}\label{parabolic operator}
\mathcal{L}u = u_t - \omega(x) a_{ij}(x,t) D_{ij} u.
\end{equation}
Here in \eqref{parabolic operator}, the Einstein summation convention is used. The notation $D_{ij}$ denotes the second order partial derivatives with respect to the spatial variables $x_i$ and $x_j$ for $i, j \in \{1, 2, . . . , n\}$ with $n \in \mathbb{N}$, and $u_t$ denotes the partial derivative of $u$ in the time $t$-variable. Moreover, $a_{ij}: \mathbb{R}^n \times \mathbb{R} \rightarrow \mathbb{R}$ is assumed to be measurable, and $\omega: \mathbb{R}^n \rightarrow \mathbb{R}$ is a non-negative measurable function which can be zero or unbounded at certain points in the considered domains.

\smallskip
We assume that the matrix $(a_{ij}(x,t))$ is symmetric, measurable in $(x,t) \in \mathbb{R}^n \times \mathbb{R}$, and it satisfies the following uniformly elliptic and boundedness conditions: there exists a constant $\nu \in (0,1)$ such that
\begin{equation} \label{elliptic}
\nu |\xi|^2 \leq a_{ij}(x,t) \xi_i \xi_j, \quad |a_{ij}(x,t)| \leq \nu^{-1}, \quad \forall \ \xi = (\xi_1, \xi_2,\ldots, \xi_n) \in \mathbb{R}^n,
\end{equation}
for every $(x,t) \in \mathbb{R}^n \times \mathbb{R}$. We also assume that  $\omega$ is  in the $A_{1+\frac{1}{n}}$ Muckenhoupt class, i.e.,
\begin{equation}\label{w-norm}
[\omega]_{A_{1+\frac{1}{n}}}=\sup_{B_r(x_0)\subset \mathbb{R}^n} 
\left(\fint_{B_r(x_0)}\omega(x)\, dx\right)\left(\fint_{B_r(x_0)}\omega(x)^{-n}\, dx\right)^{\frac{1}{n}} <\infty.
\end{equation}
Throughout the paper, $B_r(x_0)$ denotes the ball in $\R^n$ of radius $r>0$ centered at $x_0 \in \R^n$. As $\omega \in A_{1+\frac{1}{n}}$, it can vanish or blow up at some points inside the considered domain. Consequently, the leading coefficients $\omega(x) a_{ij}(x,t)$ in the parabolic operator $\mathcal{L}$ in \eqref{parabolic operator} can be degenerate or singular, or both singular and degenerate. 

\smallskip
To state our main result, let us introduce some notation.  For the given $\omega\in A_{1+\frac{1}{n}}$ in \eqref{parabolic operator}, for any point $Y=(y,s)\in \R^{n}\times\R$ and $r>0$, we define the following weighted parabolic cylinders centered at $Y$ with radius $r$ by
\begin{align*}
C_{r,\omega}(Y)&=B_r(y)\times \left(s-r^{2}\left(\omega^{-n}\right)_{B_r(y)}^{1/n},\, s\right), \ \text{and} \\[4pt]
Q_{r,\omega}(Y)&=B_r(y)\times \left(s-r^{2}\left(\omega^{-n}\right)_{B_r(y)}^{1/n},\, s+r^{2}\left(\omega^{-n}\right)_{B_r(y)}^{1/n}\right),
\end{align*}
where $(\omega^{-n})_{B_r(y)}$ denotes the average value of the function $\omega(x)^{-n}$ over the ball $B_r(y)$, that is,
\[
\left(\omega^{-n}\right)_{B_r(y)}=\frac{1}{|B_r(y)|}\int_{B_r(y)}\omega(x)^{-n}\, dx.
\]
Note that since $\omega \in A_{1+\frac{1}{n}}$, the function $\omega(x)^{-n}, x \in \mathbb{R}^n$ is locally integrable, and thus $(\omega^{-n})_{B_r(y)}$ is well-defined. Throughout the paper, for abbreviation, we denote
\[
B_r=B_r(0), \quad C_{r,\omega}=C_{r,\omega}(0,0), \quad \text{and}\quad Q_{r,\omega}=Q_{r,\omega}(0,0).
\]
In addition to the assumption that $\omega \in A_{1+\frac{1}{n}}$, we also require that its mean oscillation with respect to the weight $\omega$ is sufficiently small.  Specifically, for a nonempty open set $\Omega \subset \R^n$, let us denote the mean oscillation of $\omega$ with respect to $\omega$ in $\Omega$ by
\[
[[\omega]]_{\textup{BMO}(\Omega,\omega)} = \sup_{B_r(y) \subset \Omega} \left(\frac{1}{\omega(B_r(y))}\int_{B_r(y)}|\omega(x)-(\omega)_{B_r(y)}|\,dx \right)
\]
in which
\[
\omega(B_r(y)) = \int_{B_r(y)} \omega(x)\, dx.
\]
See Section \ref{bmo-sub-sec} for more details, and basic analysis results on the class of functions of bounded mean oscillations with a given weight.

\smallskip
For a nonempty open set $U \subset \mathbb{R}^{n+1}$, and a non-negative locally integrable function $\mu : \mathbb{R}^n \rightarrow (0, \infty)$, a measurable function $f: U \rightarrow \mathbb{R}$ is said to be in a weighted Lebesgue space $L^p(U, \mu)$ with some $p \in [1, \infty)$ if  
\begin{equation} \label{weight-Lp-space}
\|f\|_{L^{p}(U,\, \mu)}=\left(\int_{U} |f(x,t)|^p \mu(x)\, dxdt \right)^{1/p} <\infty.
\end{equation}
A locally integrable function $u: U \rightarrow \mathbb{R}$ is said to be in $\mathcal{W}^{2,1}_{n+1}(U, \omega)$ if
\[
D_{ij} u \in L^{n+1}(U,\,\omega), \quad \forall \ i,j \in \{1,2, \ldots, n\}\quad \text{and}\quad
u_t\in L^{n+1}(U,\, \omega^{-n}).
\] 
In addition, we say that $u \in \mathcal{W}^{2,1}_{n+1, \textup{loc}}(U, \omega)$ if $u \in \mathcal{W}^{2,1}_{n+1}(V, \omega)$ for any open bounded set $V \subset U$ such that $\overline{V} \subset U$. See Definition \ref{solution space} for more details on the class of functional space $\mathcal{W}^{2,1}_{n+1}(U, \omega)$. Also, in this paper, $\partial' U$ denotes the parabolic boundary of $U$.  

\smallskip
Now, we state the main result of the paper.
\begin{theorem}[Weighted $W^{2,\varepsilon}$-Lin type estimates]\label{Lin thm} For every $\nu \in (0,1)$ and $K_0\in [1,\infty)$, there exist sufficiently small positive constants $\delta=\delta(n,\nu, K_0)$ and $p_0=p_0(n,\nu, K_0)$ such that the following assertion holds. Suppose that \eqref{elliptic} holds, and that the weight $\omega \in A_{1+\frac{1}{n}}$ satisfies
\begin{equation} \label{omega-smallness-12-25}
[\omega]_{A_{1+\frac{1}{n}}} \leq K_0 \quad \text{and}\quad [[\omega]]_{\textup{BMO}(B_{3},\, \omega)}\leq \delta.
\end{equation}
Then, for every $u\in \mathcal{W}^{2,1}_{n+1, \textup{loc}}(C_{1,\,\omega},\, \omega) \cap C(\overline{C}_{1,\,\omega})$ and every $p\in(0,p_0]$, we have
\begin{equation}\label{Lin esti}
\|D^2u\|_{L^{p}(C_{1,\omega},\, \omega)}\leq N \Big( \sup_{\partial' C_{1,\omega}}|u| + \|\mathcal{L} u\|_{L^{n+1}(C_{1,\omega},\, \omega^{-n})}\Big),
\end{equation}
where $N=N(n,\nu, K_0, p)>0$.
\end{theorem}
\smallskip \noindent
The following remark gives a few examples of the weights $\omega$ satisfying \eqref{omega-smallness-12-25}. For completeness, the proof of the remark is provided in Appendix \ref{proof-weigh-example}.
\begin{remark} \label{remark-example} We provide some simple examples of the class of weights for which Theorem \ref{Lin thm} is applicable.
\begin{itemize}
\item[\textup{(a)}] Let $\omega(x) = |x|^\alpha$ with $x \in \mathbb{R}^n$. It follows from \cite{CMP, Cho-Fang-Phan}   that $\omega \in A_{1+\frac{1}{n}}$ when $\alpha \in (-n,1)$. In addition, there exist $K_0 = K_0(n)\geq 1$ and $N = N(n) >0$ such that
\[
[\omega]_{A_{1+\frac{1}{n}}} \leq K_0, \quad \text{and} \quad [[\omega]]_{\textup{BMO}(B_3,\, \omega)} \leq N|\alpha|, \quad \forall \ |\alpha| \leq 1/2.
\]
Therefore, Theorem \ref{Lin thm} applies if $|\alpha|$ is sufficiently small. 

\item[\textup{(b)}]  Let $\varphi: \mathbb{R}^n\setminus \{0\} \rightarrow \mathbb{R}$ be defined by
\[
\varphi(x) = \left\{
\begin{array}{ll}
-\ln|x| & \quad \textup{if} \quad |x| \leq e^{-1},\\
1& \quad \textup{otherwise}.
\end{array} \right.
\]
It is known that $\varphi \in A_1 \subset A_{1+\frac{1}{n}}$ with $[\varphi]_{A_1} = K_0(n)$. Moreover,  there exists a constant $N = N(n)>0$ such that
\begin{equation} \label{BMO-varphi-example}
[[\varphi]]_{\textup{BMO}(B_{r_0}, \varphi)} \leq \frac{N}{|\ln(4r_0)|} \quad \text{for} \quad  r_0 \in (0, \frac{1}{10e}).
\end{equation}
Hence, $[[\varphi]]_{\textup{BMO}(B_{r_0}, \varphi)} $ is sufficiently small if $r_0$ is sufficiently small. From this and a localization argument, we can apply Theorem \ref{Lin thm} for $\omega = \varphi$.

\item[\textup{(c)}] Given $m$ distinct points $\{x_k\}_{k=1}^m$ in $B_1$, let
\begin{align*}
\omega_1(x)  =  \prod_{k=1}^{m}  \varphi(x-x_k), \quad x \in \mathbb{R}^n,
\end{align*}
where $\varphi$ is defined in \textup{(b)}.  Also, let
\[
\omega_2(x) = \sum_{k=1}^{m} c_k |x-x_k|^{\alpha_k}, \quad \text{and} \quad \omega_3(x) = \prod_{k=1}^{m} |x-x_k|^{\alpha_k} \quad \text{for}\quad x \in \mathbb{R}^n,
\]
with some  $c_k >0$ and $|\alpha_k|$ sufficiently small for $k \in \{1, 2,\ldots, m\}$.  By using localization, dilation, translation, and a covering argument, we can apply Theorem \ref{Lin thm} to derive the estimate \eqref{Lin esti}  for $\omega \in \{\omega_1, \omega_2, \omega_3\}$.
\end{itemize}
\end{remark}
\smallskip
We note that the estimate \eqref{Lin esti} with $\omega \equiv1$ was originally proved by F.-H. Lin in \cite{Lin} for solutions to elliptic equations with uniformly elliptic and bounded measurable coefficients. The stated form of Theorem \ref{Lin thm} with estimate as \eqref{Lin esti} when $\omega \equiv 1$ first appeared in \cite{Dong-Krylov-Li} without a proof, and in \cite{Krylov-2010, Krylov-2012} with complete proofs. See also \cite{Dong-Kitano} for a recent development on this topic. Theorem \ref{Lin thm} provides a generalization of the mentioned results in \cite{Dong-Krylov-Li, Krylov-2010, Krylov-2012} to the class of parabolic equations in which the coefficients do not satisfy the uniformly elliptic nor bounded conditions. As discussed in \cite[Section 9.4, p.~ 180]{Krylov-book}, Theorem \ref{Lin thm} is expected to be one of the cornerstones for $L^p$-theory of fully nonlinear parabolic equations with singular-degenerate coefficients.  Speaking of this, we note that a statement similar to Theorem \ref{Lin thm}  for viscosity solutions to fully nonlinear elliptic equations can be found in \cite{Caff}, and for fully nonlinear parabolic equations can be found in \cite{Wang-1, Wang}. It is important to note that the classes of equations in \cite{Caff, Wang-1, Wang} are assumed to satisfy the uniformly elliptic boundedness conditions, and a vanishing mean oscillation condition. Moreover, the estimates in \cite{Caff, Wang-1, Wang} are slightly different from \eqref{Lin esti} as the domains on their right hand sides are wider than those on the left hand sides; see also \cite{FGS} for similar results. 

\smallskip
To prove Theorem \ref{Lin thm}, we establish the following result on quantitative lower estimates of solutions to the class of parabolic equations with the operator $\mathcal{L}$. This type of estimate was proved by L. C.~Evans in \cite[Theorem 1]{Evans} for elliptic equations with uniformly elliptic and bounded smooth coefficients (i.e., $\omega \equiv 1$), and it is known as the mean sojourn time of sample paths within the set $\{(x,t): \mathcal{L} u(x,t) \geq \omega(x)\}$. Note that in Theorem \ref{M-thm} below and throughout the paper, we denote
\begin{align*}
\omega(E)&=\int_{E}\omega(x)\, dx\ \qquad \text{for every measurable set}\  E \subset \R^n, \  \text{and} \\
\omega(U)&=\int_{U}\omega(x)\, dxdt\ \quad \text{for every measurable set}\ U\subset \R^{n+1}.
\end{align*}

\begin{theorem} \label{M-thm} For every $\nu \in (0,1)$ and $K_0\in [1,\infty)$, there exist constants $\delta_0=\delta_0(n,\nu, K_0)\in (0,1)$ sufficiently small, $\gamma_0=\gamma_0(n,\nu,K_0)>1$, and $N=N(n,\nu, K_0)>0$ such that the following assertion holds. Suppose that \eqref{elliptic} holds, and that the weight $\omega$ is smooth and satisfies
\begin{equation} \label{omega-cond-12-15}
[\omega]_{A_{1+\frac{1}{n}}} \leq K_0, \quad \frac{1}{k} \leq \omega \leq k, \quad \text{and }\quad [[\omega]]_{\textup{BMO}(B_{2r},\, \omega)}\leq \delta_0
\end{equation}
for some $k \in \mathbb{N}$ and $r>0$. Then for every non-negative $u\in \mathcal{W}^{2,1}_{n+1}(Q_{r,\omega}, \omega) \cap C(\overline{Q}_{r,\,\omega})$  satisfying 
\[ 
\mathcal{L}u \geq 0 \quad  \text{in} \quad  Q_{r,\,\omega},
\]
it holds that
\begin{equation}\label{M-result}
\inf_{B_{r/2}}u(\cdot, \bar{t})\geq Nq^{\gamma_0}r^2,  \quad \textup{for} \quad \bar{t} =r^2(\omega^{-n})_{B_r}^{1/n}, 
\end{equation}
where
\[
q=\frac{\omega\big(\{\mathcal{L} u\geq \omega\}\cap C_{r,\,\omega}\big)}{\omega(C_{r,\,\omega})}\in [0,1].
\]
\end{theorem}
\noindent 
It is important to note that in Theorem \ref{M-thm}, the constants $\delta_0, \gamma_0$, and $N$ are independent of the smoothness of $\omega$ and of the boundedness constant $k$ appearing in \eqref{omega-cond-12-15}. This allows us to use the regularization and truncation to remove the smoothness and boundedness conditions of $\omega$ in applications. This is actually one of the key points in the proof of Theorem \ref{Lin thm}. 

\smallskip

Let us now briefly discuss some known regularity results related to equations with singular-degenerate coefficients. Moser Harnack inequalities and H\"{o}lder regularity estimates for linear elliptic and parabolic equations in divergence form with singular-degenerate coefficients have been studied extensively since the 1970s. For example, see \cite{Ser, Fabes, Trud} for some classical work, and \cite{AFV, Bella, Pop, Sir, Sir-1} and the references therein for some recent work and related results. Krylov-Safonov Harnack inequalities and weighted interior H\"{o}lder regularity estimates for solutions $u$ of homogeneous parabolic equations with singular-degenerate coefficients $\mathcal{L}u=0$ are just recently proved in \cite{Cho-Fang-Phan}. See also the work \cite{IS, Le, Mooney} for the closely related results for elliptic equations.  On the other hand, well-posedness and regularity estimates in weighted Sobolev spaces for solutions to linear elliptic and parabolic equations with singular-degenerate coefficients in various settings have been recently studied and developed; see \cite{BDGP, CMP, Dong-Phan, Dong-Phan-Tran, Fang-Phan-div, Ji, Men}, for example. To the best of our knowledge, Theorem~\ref{Lin thm} and Theorem~\ref{M-thm} appear for the first time in the literature. Again, as discussed in \cite[Section 9.4]{Krylov-book}, Theorem \ref{Lin thm} and Theorem \ref{M-thm} provide important ingredients in the study of $L^p$-theory for fully nonlinear parabolic equations with singular-degenerate coefficients.

\smallskip
To prove Theorem \ref{Lin thm} and Theorem \ref{M-thm}, it is essential to control the heat propagation in space and time. For this purpose, barrier functions and the maximum principle are used, and this analysis relies heavily on the ellipticity and boundedness of the coefficients. In our setting, to overcome the unboundedness and degeneracy of the coefficients, we make use of a perturbation technique introduced in \cite{Cho-Fang-Phan} by freezing $\omega$, and then applying the parabolic ABP estimates on barrier functions to control solutions.  The class of weighted cylinders $C_{r, \omega}(Y)$, which is invariant under the scaling and dilation for the class of operators $\mathcal{L}$, is then essential. The analysis for this is carried out in Section \ref{prop-up-section}, in which two important prop-up lemmas are proved. It is important to note that the determinant of the coefficient matrix $\omega(x) (a_{ij}(x,t))$ is of order $\omega(x)^{n}$. Then, as the parabolic ABP estimates are used, we require $\mathcal{L}u$ to be in the weighted Lebesgue space $  L^{n+1}(U, \omega^{-n})$. As such, it is natural that $u \in \mathcal{W}^{2,1}_{n+1}(U, \omega)$; see the discussion after Theorem \ref{ABP} below for more details. We also note that in  \cite{Fang-Phan-div}, similar weighted Sobolev spaces are considered in which the well-posedness and regularity estimates for the same class of equations, but in divergence form, are proved. As such, the well-posedness and regularity estimates in this type of weighted Sobolev spaces for the class of equations $\mathcal{L}u =f$ are planned to be studied in \cite{Fang-Phan-non}.

\smallskip
After the prop-up Lemmas are established, we adjust and follow the method in \cite{Krylov-2010, Krylov-2012} to prove the main results. Due to the unboundedness and degeneracy of $\omega$ and the inhomogeneity of the weighted cylinders, technical difficulties are addressed at various points in the approach. For example, a weighted version of the Krylov-Safonov covering lemma is proved. In many steps, we require classical results on existence, uniqueness, and regularity estimates for solutions to the class of equations $\mathcal{L} u = f$. For such, we truncate $\omega$ and then regularize $\omega$ and $(a_{ij})$. Careful analysis is required to pass the limits when applying the truncation and regularization of the coefficients. Basic results on the stability of Muckenhoupt weights and mean oscillations with respect to weights through the truncation and regularization are established and proved in Section \ref{pre-section}. These foundation results and techniques are of independent interest, and they can be applicable to other problems.

\smallskip
The rest of the paper is organized as follows. In Section \ref{pre-section}, we recall the definition of Muckenhoupt weights and some of their basic properties, and we introduce the class of weighted parabolic cylinders used in the paper. Many basic results on the bounded mean oscillations with weights are proved. In this section, the parabolic ABP theorem and the Krylov-Safonov covering lemma are also stated. In Section \ref{prop-up-section}, we establish and prove two important prop-up lemmas which provide the foundational results for the paper. The proof of Theorem \ref{M-thm} is carried out in Section \ref{pf of Thm1}. Similarly, Section \ref{section of lin's proof} is devoted to the proof of Theorem \ref{Lin thm}. The paper concludes with Appendix \ref{proof-weigh-example} and Appendix \ref{Appendix-A} that provide the proofs of Remark \ref{remark-example} and of the weighted Krylov-Safonov covering lemma, respectively.
\section{Preliminaries} \label{pre-section}
\subsection{Muckenhoupt weights and weighted parabolic cylinders} We recall the definition of the $A_p$-Muckenhoupt class of weights introduced in \cite{Muckenhoupt}, and point out some of its important properties needed in the paper.  Throughout the paper, for a locally integrable function $f$ defined in a neighborhood of $B_r(x_0)$, its mean on $B_r(x_0) \subset \R^n$ is denoted by 
\[
(f)_{B_r(x_0)}=\fint_{B_r(x_0)}f(x)\,dx = \frac{1}{|B_r(x_0)|} \int_{B_r(x_0)} f(x)\, dx,
\]
where $|B_r(x_0)|$ denotes the Lebesgue measure of $B_{r}(x_0)$.

\begin{definition} \label{A-p-def} Let $p \in (1, \infty)$, and $\mu : \R^n \rightarrow \R$ be a non-negative locally integrable function. We say that $\mu$ belongs to the $A_p$ Muckenhoupt class if  $[\mu]_{A_p} <\infty$, where
\[
[\mu]_{A_p}=\sup_{B_r(x_0) \subset \mathbb{R}^{n}} 
\left(\fint_{B_r(x_0)}\mu(x)\, dx\right)\left(\fint_{B_r(x_0)}\mu(x)^{-\frac{1}{p-1}}\, dx\right)^{p-1}.
\]
\end{definition}
\noindent
We observe that if $\mu \in A_p$ then $[\mu]_{A_p} \geq 1$. 
Indeed, as $1<p<\infty$, for each ball $B\subset \mathbb{R}^n$, it follows from H\"{o}lder's inequality that 
\begin{align*}
1= \fint_{B} \mu(x)^{\frac{1}{p}}\mu(x)^{-\frac{1}{p}} \, dx 
\leq \Big(\fint_{B}\mu(x)\, dx \Big)^{\frac{1}{p}} \Big(\fint_{B}\mu(x)^{-\frac{1}{p-1}} \,dx \Big)^{\frac{p-1}{p}}
\leq [\mu]_{A_p}^{\frac{1}{p}}.
\end{align*}
As a result, for the $\omega$ in \eqref{parabolic operator}, if $\omega \in A_{1+\frac{1}{n}}$, then
\begin{equation} \label{A-1-1/n}
1 \leq  (\omega)_{B} (\omega^{-n})_{B}^{\frac{1}{n}}  \leq [\omega]_{A_{1+\frac{1}{n}}}\quad \text{for every ball}\quad B\subset \R^n.
\end{equation}
From this, we see that the function $\omega(x)^{-n}, x \in \mathbb{R}^n$ is a non-negative locally integrable function. In particular, it follows directly from Definition \ref{A-p-def} that $\omega^{-n} \in A_{n+1}$ with
\begin{equation} \label{omega-n-wei-12-22}
[\omega^{-n}]_{A_{n+1}} = [\omega]_{A_{1+\frac{1}{n}}}^n.
\end{equation}
The following classes of weighted parabolic cylinders are used in the paper.
\begin{definition} \label{cylinder-def} Let $\omega \in A_{1+\frac{1}{n}}$. For $r > 0$ and $Y = (y, s) \in \mathbb{R}^{n}\times \R$, we define two weighted non-homogeneous parabolic cylinders centered at $Y$ and of radius $r$ by
\begin{equation}\label{cylinder def-1}
C_{r,\,\omega}(Y)= B_r(y) \times \big(s-r^2(\omega^{-n})_{B_r(y)}^{1/n},\, s\big),
\end{equation}
and 
\begin{equation}\label{cylinder def-2}
Q_{r,\,\omega}(Y)=B_r(y)\times (s-r^{2}(\omega^{-n})_{B_r(y)}^{1/n},\, s+r^{2}(\omega^{-n})_{B_r(y)}^{1/n}).
\end{equation}
\end{definition}
\noindent
From Definition \ref{cylinder-def}, we note that $\omega\big(Q_{r,\, \omega}(Y)\big)=2\omega(C_{r,\, \omega}(Y))$ and
\begin{equation*}
\omega(C_{r,\, \omega}(Y))= \omega(B_r)r^2(\omega^{-n})_{B_r}^{1/n}=\sigma_nr^{n+2}(\omega)_{B_r}(\omega^{-n})_{B_r}^{1/n},\quad \text{where}\quad \sigma_n=|B_1|.
\end{equation*}
Moreover, by \eqref{A-1-1/n}, we have the following property:
\begin{equation}\label{cylinder measure}
\sigma_nr^{n+2}\leq \omega(C_{r,\, \omega}(Y))\leq [\omega]_{A_{1+\frac{1}{n}}}\sigma_nr^{n+2}, \quad \text{for all}\quad r>0,\ Y\in \mathbb{R}^{n} \times \mathbb{R}.
\end{equation}
Besides, the parabolic boundaries of the cylinders $C_{r, \omega}(Y)$ and $Q_{r, \omega}(Y)$ are given by
\[
\partial' C_{r,\, \omega}(Y)= \left(\partial B_r(y) \times \big(s-r^2 (\omega^{-n})_{B_r(y)}^{1/n},\, s\big)\right) \cup \left(\overline{B}_r(y) \times \big\{s-r^2(\omega^{-n})_{B_r(y)}^{1/n}\big\} \right),
\] 
and
\begin{align*}
\partial' Q_{r,\, \omega}(Y)
= &\left(\partial B_r(y) \times \big(s-r^2 (\omega^{-n})_{B_r(y)}^{1/n},\, s+r^2 (\omega^{-n})_{B_r(y)}^{1/n}\big)\right)\\
&\cup \left(\overline{B}_r(y) \times \big\{s-r^2(\omega^{-n})_{B_r(y)}^{1/n}\big\} \right).
\end{align*}

\smallskip
In the rest of the subsection, we introduce two results on Muckenhoupt weights that are needed in the paper. The following proposition asserts that a truncation of an $A_p$ weight is still an $A_p$ weight. For this purpose, we define the upper and lower truncation operators $T^k: \mathbb{R}\rightarrow \mathbb{R}$ and  $T_k: \mathbb{R}\rightarrow \mathbb{R}$ at level $k$ as follows:
\begin{equation}\label{truncation operators}
T^k(s)=\min\{k, s\}\quad \text{and}\quad T_k(s)=\max\{k, s\},\quad  s \in \mathbb{R}.
\end{equation}
Although the proposition seems to be fundamental, we could not locate a proof in the literature. We include the proof of the proposition for completeness; see also \cite[exercise 7.1.8, p.~ 513]{Grafakos} for the question on the upper truncation $T^k\circ \mu$ with a slightly different estimate.
\begin{proposition} \label{cut-omega} 
Let $\mu \in A_{p}$ for some $p \in (1, \infty)$. Then, for each $k>0$, we have
\begin{equation} \label{up-low-weight}
[T^k\circ \mu]_{A_p} \leq \Big([\mu]_{A_{p}}^{\frac{1}{p-1}} + 1\Big)^{p-1} \quad \text{and} \quad [T_k\circ \mu]_{A_p} \leq [\mu]_{A_{p}}+ 1.
\end{equation}
Moreover, for each $0 < s < \tau < \infty$, define the truncated weight $\bar \mu$ by
\[
\bar{\mu}(x) = \left\{\begin{array}{ll}
\mu(x) & \quad \text{if} \quad s \leq \mu(x) \leq \tau,\\ \smallskip
s & \quad \text{if} \quad \mu(x) < s, \\ \smallskip
\tau & \quad \text{if} \quad \mu(x) > \tau,
\end{array} \qquad x\in\mathbb{R}^n.\right.
\]
Then 
\begin{equation}\label{truncate-A-p-estimate}
[\bar{\mu}]_{A_p} \leq 2^{\max\{p-2,\, 0\}}\left([\mu]_{A_{p}}+1\right)+1.
\end{equation}
Particularly, when $p=1+\frac{1}{n}$, we have
\begin{equation} \label{truncate-A-p-est}
[\bar{\mu}]_{A_{1+\frac{1}{n}}} \leq [\mu]_{A_{1+\frac{1}{n}}}+2.
\end{equation}
\end{proposition}
\begin{proof} Let $B = B_r(x_0)$ for any $r>0$ and $x_0 \in \mathbb{R}^n$. We consider the weight $T^k\circ \mu$.  Let us denote
$
E= \left\{x \in B: \mu(x) >k \right\}.
$
Then,
\begin{align*}
\fint_{B} [T^k\circ \mu(x)]^{-\frac{1}{p-1}}\, dx 
&=k^{-\frac{1}{p-1}} \frac{|E|}{|B|}+\frac{1}{|B|} \int_{B\setminus E} \mu(x)^{-\frac{1}{p-1}}\, dx \\
&\leq k^{-\frac{1}{p-1}}+\fint_{B}\mu(x)^{-\frac{1}{p-1}}\, dx.
\end{align*}
By this, and the facts that $T^k\circ \mu(x)\leq k$ and $T^k\circ \mu(x)\leq \mu(x)$, we get
\begin{align*}
&\left(\fint_{B} T^k\circ \mu(x)\, dx \right)^{\frac{1}{p-1}}\left(\fint_{B} [T^k\circ \mu(x)]^{-\frac{1}{p-1}}\, dx\right) \\
& \leq \left(\fint_{B} T^k\circ \mu(x)\, dx \right)^{\frac{1}{p-1}} \left(k^{-\frac{1}{p-1}}+\fint_{B}\mu(x)^{-\frac{1}{p-1}}\, dx \right)\\
&\leq k^{\frac{1}{p-1}}k^{-\frac{1}{p-1}}+\left(\fint_{B} \mu (x)\, dx \right)^{\frac{1}{p-1}} \left(\int_{B} \mu(x)^{-\frac{1}{p-1}}\, dx\right)
\leq  [\mu]_{A_{p}}^{\frac{1}{p-1}} + 1.
\end{align*}
As $B=B_r(x_0)$ is arbitrary, it then follows that
\[ 
[T^k\circ \mu]_{A_p} \leq \Big([\mu]_{A_{p}}^{\frac{1}{p-1}} + 1\Big)^{p-1},
\]
and the first assertion in \eqref{up-low-weight} is proved.

\smallskip
To prove $T_k\circ \mu \in A_p$,  we consider the set
$
A= \{x \in B: \mu(x) < k\}.
$
Then,
\[
\fint_{B} T_k\circ \mu(x)\, dx = k \frac{|A|}{|B|}+\frac{1}{|B|} \int_{B \setminus A} \mu(x)\,dx \leq k+\fint_B \mu(x)\, dx.
\]
Using this, and the facts that $T_k\circ \mu(x)\geq k$ and $T_k\circ \mu(x)\geq \mu(x)$, we see that
\begin{align*}
&\left(\fint_{B} T_k\circ \mu(x)\, dx\right) \left( \fint_{B} [T_k\circ \mu(x)]^{-\frac{1}{p-1}}\, dx \right)^{p-1}\\
&\leq \left( k+\fint_B \mu(x)\, dx\right) \left( \fint_{B} [T_k\circ \mu(x)]^{-\frac{1}{p-1}}\, dx \right)^{p-1} \\
&\leq k \left( k^{-\frac{1}{p-1}} \right)^{p-1}+\left(\fint_{B } \mu(x)\, dx \right)  \left( \fint_{B} \mu(x)^{-\frac{1}{p-1}}\, dx \right)^{p-1}
\leq  [\mu]_{A_p} + 1.
\end{align*} 
Similarly, the second assertion in \eqref{up-low-weight} is then proved.

\smallskip
To prove \eqref{truncate-A-p-estimate}, observe that
\[
\bar{\mu}(x) = \max\big\{\min\{ \mu(x), \tau\},\, s\big\}=T_{s}\circ [T^{\tau}\circ \mu(x)].
\]
It then follows from \eqref{up-low-weight} that
\begin{align*}
[\bar \mu]_{A_p}
&\leq \Big([\mu]_{A_p}^{\frac{1}{p-1}}+1\Big)^{p-1}+1 \\
&\leq 2^{\max\{p-2,\, 0\}}\left([\mu]_{A_{p}}+1\right)+1,
\end{align*}
where in the last step we used the elementary inequality $(a+b)^q\leq 2^{\max\{q-1,\, 0\}}(a^q+b^q)$ for all $a,\, b\geq 0$ and $q>0$. 
This proves \eqref{truncate-A-p-estimate}, and \eqref{truncate-A-p-est} follows directly from \eqref{truncate-A-p-estimate}. The proof of the lemma is completed.
\end{proof}

We conclude this subsection with the following proposition on the stability of  Muckenhoupt weights under regularization by convolution with mollifiers. We remark that the case for $A_1$-weights is proved in \cite[Lemma 2.1]{Kiku}. However, the general case for $A_p$-weights with $p \in (1, \infty)$ does not seem to be written in the literature. We include the statement and its proof for completeness.
\begin{proposition} \label{A-p-regularization} 
Let $\phi \in C_c^\infty(B_1)$ satisfy $0 \leq \phi \leq 1$ and $\int_{\mathbb{R}^n} \phi (x)\, dx= 1$.  For $\varepsilon>0$ and a given $\mu  \in A_{p}$ with $p \in (1,\infty)$, define
\[
\mu_\varepsilon = \mu * \phi_\varepsilon, \quad \text{where}\quad
\phi_\varepsilon(x) = \varepsilon^{-n}\phi(x/\varepsilon), \quad x \in \mathbb{R}^n. 
\]
Then $\mu_\varepsilon \in A_p$, and
\begin{equation}\label{mu-epsi-ap}
[\mu_\varepsilon]_{A_p} \leq 2^{np}\, [\mu]_{A_p}, \qquad \forall\, \varepsilon>0.
\end{equation}
Particularly, when $p=1+\frac{1}{n}$, we have
\begin{equation}\label{mu-epsi-1ap}
[\mu_\varepsilon]_{A_{1+\frac{1}{n}}} \leq 2^{n+1}\, [\mu]_{A_{1+\frac{1}{n}}}, \qquad \forall\, \varepsilon>0.
\end{equation}
\end{proposition}

\begin{proof} We fix $\varepsilon>0$. By the definition of convolution, we have
\begin{equation}\label{mu-epsi}
\mu_{\varepsilon}(x)=\int_{\mathbb{R}^n}\mu(x-y)\phi_{\varepsilon}(y)\, dy,\qquad \forall\, x\in \mathbb{R}^n.
\end{equation}
Since the function $s \mapsto s^{-\frac{1}{p-1}}$ is a convex function for $s\in (0,\infty)$, and $\phi_{\varepsilon}(y)dy$ is a probability measure, it follows from Jensen's inequality that for all $x\in \R^n$,
\begin{equation}\label{mu-epsi-1}
   \mu_{\varepsilon}(x)^{-\frac{1}{p-1}}
   =\left(\int_{\R^n}\mu(x-y)\phi_{\varepsilon}(y)\, dy\right)^{-\frac{1}{p-1}}
   \leq \int_{\R^n}\mu(x-y)^{-\frac{1}{p-1}}\phi_{\varepsilon}(y)\, dy.
\end{equation}

\smallskip
To prove \eqref{mu-epsi-ap}, we fix a ball $B_r(x_0)\subset \R^n$ with $r>0$, and consider the followings. 

\smallskip
\noindent
\textbf{Case 1: $r\geq \varepsilon$.}
From this, \eqref{mu-epsi}, and Fubini's theorem, it follows that
\begin{align*}
(\mu_{\varepsilon})_{B_r(x_0)}
&=\frac{1}{|B_{r}(x_0)|}\int_{\R^n}\phi_{\varepsilon}(y)\left(\int_{B_r(x_0)}\mu(x-y)\, dx\right)\, dy\\
&\leq \frac{1}{|B_{r}(x_0)|}\int_{B_{r+\varepsilon}(x_0)}\mu(x)\, dx\\
&=\left(\frac{r+\varepsilon}{r}\right)^n(\mu)_{B_{r+\varepsilon}(x_0)}\leq 2^n\, (\mu)_{B_{r+\varepsilon}(x_0)}.
\end{align*}
Similarly, by \eqref{mu-epsi-1}, Fubini's theorem, and the assumption that $r\geq \varepsilon$, 
\begin{align*}
\big(\mu_{\varepsilon}^{-\frac{1}{p-1}}\big)_{B_r(x_0)}
=2^n\,\big(\mu^{-\frac{1}{p-1}}\big)_{B_{r+\varepsilon}(x_0)}.
\end{align*}
Combining the last two estimates yields
\begin{equation}\label{r-geq-epsi}
\big(\mu_{\varepsilon}\big)_{B_r(x_0)} \big(\mu_{\varepsilon}^{-\frac{1}{p-1}}\big)_{B_r(x_0)}^{p-1} \le 2^{np}\, \big(\mu\big)_{B_{r+\varepsilon}(x_0)}\big(\mu^{-\frac{1}{p-1}}\big)_{B_{r+\varepsilon}(x_0)}^{p-1}
\le 2^{np}\, [\mu]_{A_p}.
\end{equation}

\smallskip
\noindent
\textbf{Case 2: $r<\varepsilon$.} From this, \eqref{mu-epsi-1}, and the definition of $\phi_{\varepsilon}$, we obtain
\begin{align*}
\mu_{\varepsilon}(x) &\le \varepsilon^{-n}\int_{B_{\varepsilon}}\mu(x-y)\, dy\\ 
&\le \varepsilon^{-n}\int_{B_{2\varepsilon}(x_0)}\mu(y)\, dy
=2^n \, (\mu)_{B_{2\varepsilon}(x_0)}, \qquad \forall \, x\in B_r(x_0).
\end{align*}
Hence,
\[
(\mu_{\varepsilon})_{B_r(x_0)}\leq 2^n \, (\mu)_{B_{2\varepsilon}(x_0)}.
\]
Similarly, by \eqref{mu-epsi-1}, the definition of $\phi_{\varepsilon}$, and the assumption that $r<\varepsilon$, it follows that
\[
\big(\mu_{\varepsilon}^{-\frac{1}{p-1}}\big)_{B_r(x_0)}\leq 2^n\, \big(\mu^{-\frac{1}{p-1}}\big)_{B_{2\varepsilon}(x_0)}.
\]
Thus, the previous two estimates imply that
\begin{equation}\label{r-leq-epsi}
\big(\mu_{\varepsilon}\big)_{B_r(x_0)} \big(\mu_{\varepsilon}^{-\frac{1}{p-1}}\big)_{B_r(x_0)}^{p-1} \le 2^{np}\, \big(\mu\big)_{B_{2\varepsilon}(x_0)}\big(\mu^{-\frac{1}{p-1}}\big)_{B_{2\varepsilon}(x_0)}^{p-1}
\le 2^{np}\, [\mu]_{A_p}.
\end{equation}

\smallskip
Therefore, from \eqref{r-geq-epsi} and \eqref{r-leq-epsi}, and since $B_r(x_0)$ is arbitrary, we obtain
\[
[\mu_{\varepsilon}]_{A_p}\leq 2^{np}\, [\mu]_{A_p}.
\]
Since $\varepsilon>0$ is also arbitrary, then \eqref{mu-epsi-ap} is proved and \eqref{mu-epsi-1ap} follows directly from \eqref{mu-epsi-ap}. The proof of the lemma is completed.
\end{proof}

\subsection{Bounded mean oscillations with respect to weights and their stability}  \label{bmo-sub-sec}  To prove Theorem \ref{Lin thm}, in some steps, we truncate $\omega$ to avoid its degeneracy and singularity, and then regularize $\omega$ and $(a_{ij})$ so that we can use the classical results such as existence and uniqueness of solutions for parabolic equations in non-divergence form, and comparison principles. Therefore, we need to control the mean oscillations with respect to weights through the truncation and regularization. These basic results are proved in this subsection.

\smallskip
For the reader's convenience, let us recall the following definition on the class of functions that have bounded mean oscillation with respect to a weight, introduced in \cite{MW}.
\begin{definition} Let $\mu \in L^1_{\textup{loc}}(\mathbb{R}^n)$ be non-negative and $f \in L^1_{\textup{loc}}(\Omega)$ with some nonempty open set $\Omega \subset \mathbb{R}^n$. We say that $f$ is in $\textup{BMO}(\Omega, \mu)$ if
\[
[[f]]_{\textup{BMO}(\Omega, \mu)} = \sup_{B_r(x_0) \subset \Omega} \frac{1}{\mu(B_r(x_0))} \int_{B_r(x_0)} |f(x) - (f)_{B_r(x_0)}|\, dx <\infty.
\]
\end{definition}
\noindent
For a given number $q \in [1, \infty)$, we denote
\[
[[f]]_{\textup{BMO}_q(\Omega, \mu)} = \left( \sup_{B_r(x_0) \subset \Omega} \frac{1}{\mu(B_r(x_0))} \int_{B_r(x_0)} |f(x) - (f)_{B_r(x_0)}|^q \mu(x)^{1-q}\, dx \right)^{\frac{1}{q}}.
\]
The following John-Nirenberg type lemma is proved in \cite[Theorem 4]{MW}.
\begin{lemma} \label{bmo-lemma} Let $K_0 \geq 1$, $p \in (1, \infty)$, and $q \in [1, \frac{p}{p-1}]$. There are positive constants $\bar N_1 = \bar N_1(n, p, q, K_0)$ and $\bar N_2 =\bar N_2(n, p, q, K_0)$ such that the following assertion holds. If $\mu \in A_{p}$ with $[\mu]_{A_p} \leq K_0$, then
\[
\bar N_1 [[f]]_{\textup{BMO}_q(\Omega, \mu)}   \leq [[f]]_{\textup{BMO}(\Omega, \mu)}  \leq \bar N_2 [[f]]_{\textup{BMO}_q(\Omega, \mu)},
\]
for all $f \in \textup{BMO}(\Omega, \mu)$.
\end{lemma}
\noindent
It is important to note that, due to the assumption $\omega \in A_{1+\frac{1}{n}}$ and Lemma \ref{bmo-lemma}, we use the power $q=n+1$ as in \eqref{BMO on ball} instead of $q=1$. This is because the exponent $q=n+1$ appears naturally in the calculations. However, for convenient manipulation, in the remaining part of this subsection, we use $q =1$.

\smallskip
The following proposition gives control of the weighted bounded mean oscillations of the regularization of a given weight.
The proposition is needed in the proof of Theorem \ref{Lin thm}.

\begin{proposition} \label{weighted-BMO-regularization}
Let $\phi \in C_c^\infty(B_1)$ satisfy $0 \leq \phi \leq 1$ and $\int_{\mathbb{R}^n} \phi(x) dx = 1$.  For $\varepsilon>0$ and a given $\mu  \in A_{p}$ with $p \in (1,\infty)$, define
\[
\mu_\varepsilon = \mu * \phi_\varepsilon, \quad \text{where}\quad
\phi_\varepsilon(x) = \varepsilon^{-n}\phi(x/\varepsilon). 
\]
Then, for each $R_0>0$,
\[ 
[[\mu_\varepsilon]]_{\mathrm{BMO}(B_{R_0},\, \mu_\varepsilon)} \le  [[\mu]]_{\mathrm{BMO}(B_{R_0+1},\, \mu)}, \quad \forall \ \varepsilon \in (0,1).
\]
\end{proposition}
\begin{proof}  Fix a ball $B = B_r(x_0)$ with some $r >0$ and $x_0 \in \mathbb{R}^n$ such that $B\subset B_{R_0}$. We need to estimate the mean oscillation of $\mu_\varepsilon$ on $B$:
\[
\frac{1}{\mu_\varepsilon(B)}\int_B |\mu_\varepsilon(x) - (\mu_\varepsilon)_B|\,dx.
\]
Let us note that
\[
\frac{1}{|B|}\int_B \mu(z-y)\,dz = \frac{1}{|B-y|}\int_{B-y} \mu(x)\,dx= (\mu)_{B-y},
\]
where $B-y = \{x - y: x \in B\} = B_r(x_0 -y)$.
Therefore, it follows from this and Fubini's theorem that
\begin{equation} \label{eps-measure}
\mu_\varepsilon(B) = \int_{B_{\varepsilon}} \phi_\varepsilon(y) \mu (B-y)\, dy \quad \text{and} \quad (\mu_\varepsilon)_B = \int_{B_{\varepsilon}} \phi_\varepsilon(y) (\mu)_{B-y}\, dy.
\end{equation}
Then, we see that
\begin{align*}
|\mu_\varepsilon(x) - (\mu_\varepsilon)_B|
    &=\left| \int_{B_{\varepsilon}} \big( \mu(x-y) - (\mu)_{B-y} \big)\phi_\varepsilon(y)\,dy\right|\\
    &\le \int_{B_{\varepsilon}}|\mu(x-y) - (\mu)_{B-y}|\phi_\varepsilon(y)\,dy.
\end{align*}
From this and by using Fubini's theorem again, we obtain
\[
\frac{1}{\mu_\varepsilon(B)}\int_B |\mu_\varepsilon(x) - (\mu_\varepsilon)_B|\,dx
\le \int_{B_\varepsilon} \phi_\varepsilon(y)
    \left( \frac{1}{\mu_\varepsilon(B)}\int_B |\mu(x-y) - (\mu)_{B-y}|\,dx \right) dy. 
\]
Now, we control the inner integral in the right hand side of the last estimate as follows:
\begin{align*}
\int_B |\mu(x-y) - (\mu)_{B-y}|\,dx  & = \int_{B-y} |\mu(z) - (\mu)_{B-y}|\,dz \\
& \leq \mu(B-y) [\mu]_{\mathrm{BMO}(B_{R_0+1}, \mu)}
\end{align*}
for all $y \in B_\varepsilon$ and for all $\varepsilon \in (0,1)$. Hence, 
\[
\frac{1}{\mu_\varepsilon(B)}\int_B |\mu_\varepsilon(x) - (\mu_\varepsilon)_B|\,dx \leq 
[\mu]_{\mathrm{BMO}(B_{R_0+1}, \mu)}  \frac{1}{\mu_\varepsilon(B)}  \int_{B_\epsilon} \mu(B-y)\phi_\varepsilon(y)\,dy.
\]
This last estimate and the first assertion in \eqref{eps-measure} yield
\[
\frac{1}{\mu_\varepsilon(B)}\int_B |\mu_\varepsilon(x) - (\mu_\varepsilon)_B|\,dx \leq  [\mu]_{\mathrm{BMO}(B_{R_0+1}, \mu)}.
\]
As $B= B_r(x_0) \subset B_{R_0}$ is arbitrary, we conclude that
\[ 
[[\mu_\varepsilon]]_{\mathrm{BMO}(B_{R_0}, \mu_\varepsilon)} \le  [[\mu]]_{\mathrm{BMO}(B_{R_0+1}, \mu)}, \quad \forall \ \varepsilon \in (0,1).
\]
The proof of the proposition is completed.
\end{proof}
Next, we provide several results on the weighted mean oscillations of truncated weights. The following lemma is on the lower truncation of an $A_p$ weight.
\begin{lemma} \label{bmo-lower-truncated-weight} Let $\mu \in A_{p}$ with some $p \in (1, \infty)$, and let $T_k$ be the lower truncation operator defined in \eqref{truncation operators} with some $k\in (0,\infty)$.
Then, for each $R_0 >0$, we have
\[
[[T_k\circ \mu]]_{\textup{BMO}(B_{R_0}, T_k\circ \mu)} \leq 2 [[\mu]]_{\textup{BMO}(B_{R_0}, \mu)}.
\]
\end{lemma}
\begin{proof} 
We observe that 
\[
T_k\circ \mu(x) = \max\{k, \mu(x) \}, \quad x \in \mathbb{R}^n.
\]
Let $B = B_r(x_0) \subset B_{R_0}$ with some $r>0$ and $x_0 \in \mathbb{R}^n$. Also, as $T_k$ is convex, it follows from Jensen's inequality that 
\[
T_k \left(\frac{1}{|B|} \int_{B} \mu (x)\, dx \right) \leq \frac{1}{|B|} \int_{B} T_k\circ \mu(x)\, dx.
\]
Hence,
\begin{equation} \label{measure-phi}
(T_k \circ \mu) (B) \geq  T_k \left(\frac{1}{|B|} \int_{B} \mu (x)\, dx \right)|B| \geq  \mu(B).
\end{equation}
Now, by using the triangle inequality and the fact that $T_k$ is Lipschitz with unit Lipschitz constant, we see that
\begin{align*}
\int_{B} |T_k\circ \mu(x) - (T_k\circ \mu)_{B}|\, dx  & \leq 2 \int_{B} |T_k\circ \mu(x) - T_k((\mu)_{B}) |\, dx \\
& \leq 2 \int_{B} |\mu(x) -(\mu)_{B} |\, dx.
\end{align*}
Then, it follows from the last estimate and \eqref{measure-phi} that
\begin{align*}
\frac{1}{(T_k\circ\mu)(B)}\int_{B} |T_k\circ \mu(x) - (T_k\circ \mu)_{B}|\, dx  &\leq   \frac{2}{\mu(B)} \int_{B} |\mu(x) -(\mu)_{B} |\, dx\\
& \leq 2 [\mu]_{\mathrm{BMO}(B_{R_0}, \mu)}.
\end{align*}
The assertion of the lemma then follows.
\end{proof}

To control the weighted mean oscillations of the upper truncation of a given weight, we need the following auxiliary result, which is of independent interest.
\begin{lemma} \label{osc-inver-wei} 
For $K_0 \geq 1$, there exists $N = N(n, K_0) >0$ such that for every $\mu \in A_2$ with $[\mu]_{A_2} \leq K_0$,  it holds that
\[
[[\beta]]_{\mathrm{BMO}(B_{R_0}, \beta)} \leq N  [[\mu]]_{{\mathrm{BMO}(B_{R_0}, \mu)}},\quad \forall\, R_0>0,
\]
where $\beta(x) = \mu(x)^{-1}$ for $x \in \mathbb{R}^n$.
\end{lemma}

\begin{proof} It follows from Lemma \ref{bmo-lemma} that there exists $N = N(n,  K_0)>0$ such that
\begin{equation} \label{equiv-BMO}
\left( \frac{1}{\mu(B)} \int_{B}|\mu(x) - (\mu)_{B}|^{2} \mu(x)^{-1}\, dx \right)^{1/2} \leq N [[\mu]]_{{\mathrm{BMO}(B_{R_0}, \mu)}},
\end{equation}
for all ball $B \subset B_{R_0}$. We now fix a ball $B = B_r(x_0) \subset B_{R_0}$ with some $r>0$ and $x_0 \in B_{R_0}$. We recall that
\[
(\mu)_B = \fint_{B} \mu(x)\, dx.
\]
Then, we have 
\begin{align} \notag
& \fint_{B} |\beta(x) - (\mu)^{-1}_{B} |\, dx = (\mu)^{-1}_{B} \fint_{B} |\mu(x) - (\mu)_{B}| \mu(x)^{-1}\, dx \\ \notag
& \leq  (\mu)^{-1}_{B} \left( \fint_{B}| \mu(x) - (\mu)_{B}|^2 \mu(x)^{-1}\, dx\right)^{1/2} \left(\fint_{B} \mu(x)^{-1}\, dx \right)^{1/2} \\ \notag
& = (\mu)_{B}^{-1/2}  \left(\frac{1}{\mu(B)} \int_{B}| \mu(x) - (\mu)_{B}|^2 \mu(x)^{-1}\, dx \right)^{1/2} \left(\fint_{B} \beta(x)\, dx \right)^{1/2} \\ \label{beta-bmo-mu-12-31}
&\leq  N [[\mu]]_{{\mathrm{BMO}(B_{R_0}, \mu)}}  \left(\fint_{B}  \mu(x)\, dx \right)^{-1/2} \left(\fint_{B} \beta(x)\, dx \right)^{1/2},
\end{align}
where we used \eqref{equiv-BMO} in the last step. By H\"{o}lder's inequality, we note that
\[
1 = \fint_{B}\, dx \leq  \left(\fint_{B}  \mu(x) dx \right)^{1/2} \left(\fint_{B} \mu(x)^{-1} dx \right)^{1/2},
\]
and therefore
\[
 \left(\fint_{B}  \mu(x) dx \right)^{-1/2} \leq \left(\fint_{B} \mu(x)^{-1} dx \right)^{1/2} = \left(\fint_{B} \beta(x) dx \right)^{1/2}.
\]
From this and \eqref{beta-bmo-mu-12-31}, we obtain
\[
\fint_{B} |\beta(x) -  (\mu)_B^{-1}|\, dx \leq  N [[\mu]]_{{\mathrm{BMO}(B_{R_0}, \mu)}}   \fint_{B} \beta(x)\, dx.
\]
Then, it follows from the triangle inequality and the last estimate that
\begin{align*}
\fint_{B} |\beta(x) - (\beta)_B|\, dx
&\leq 2\fint_{B} |\beta(x) - (\mu)_B^{-1}|\, dx\\
&\leq N [[\mu]]_{{\mathrm{BMO}(B_{R_0}, \mu)}} \left(\fint_{B} \beta(x)\, dx \right).
\end{align*}
Consequently,
\[
\frac{1}{\beta(B)} \int_{B} |\beta(x) - (\beta)_B|\, dx \leq N [[\mu]]_{{\mathrm{BMO}(B_{R_0}, \mu)}}.
\]
As $B \subset B_{R_0}$ is arbitrary, the assertion of the lemma follows.
\end{proof}

The following lemma establishes the stability of the weighted mean oscillation of a weight under upper truncations.
\begin{lemma}  \label{mean-upp-truncate-wei} 
For $K_0 \geq 1$, there exists $N = N(n, K_0) >0$ such that for every $\mu \in A_2$ with $[\mu]_{A_2} \leq K_0$, it holds that
\[
[[T^k\circ \mu]]_{\mathrm{BMO}(B_{R_0},\, T^k\circ \mu)} \leq N  [\mu]_{\mathrm{BMO}(B_{R_0}, \mu)},
\]
for all $k>0$ and $R_0>0$, where $T^k$ be the upper truncation operator defined in \eqref{truncation operators}. 
\end{lemma} 
\begin{proof} Let $\beta(x) = \mu(x)^{-1}$ for $x \in \mathbb{R}^n$. Since $\mu\in A_2$, it follows directly from Definition \ref{A-p-def} that $\beta \in A_2$ with $[\beta]_{A_2} = [\mu]_{A_2}$. Let $T_{k^{-1}}$ be the lower truncation operator defined in \eqref{truncation operators}, corresponding to the level $k^{-1}$. From Proposition \ref{cut-omega}, we see that $T_{k^{-1}}\circ \beta \in A_2$, and
\[
 [T_{k^{-1}}\circ \beta]_{A_2} \leq [\beta]_{A_2}  + 1 \leq 2K_0.
\]
As a result, it follows from Lemma \ref{bmo-lower-truncated-weight} and Lemma \ref{osc-inver-wei} that
\[
[[T_{k^{-1}}\circ \beta]]_{\mathrm{BMO}(B_{R_0}, T_{k^{-1}}\circ \beta)} \leq  2[[\beta]]_{\mathrm{BMO}(B_{R_0}, \beta)} \leq N  [[\mu]]_{\mathrm{BMO}(B_{R_0}, \mu)},
\]
for $N = N(n, K_0)>0$. Now, note that $T^k\circ \mu(x) = [T_{k^{-1}}\circ\beta (x)]^{-1}$. Then, by applying Lemma \ref{osc-inver-wei}, and using the last estimate, we infer that
\[
[[T^k\circ \mu]]_{\mathrm{BMO}(B_{R_0}, T^{k}\circ \mu)} \leq N [[T_{k^{-1}}\circ\beta]]_{\mathrm{BMO}(B_{R_0}, T_{k^{-1}}\circ\beta)} \leq  N  [[\mu]]_{\mathrm{BMO}(B_{R_0}, \mu)},
\]
for $N = N(n, K_0)>0$. The lemma is proved.
\end{proof}

As a direct corollary of Proposition \ref{cut-omega}, Lemma \ref{bmo-lower-truncated-weight}, and Lemma \ref{mean-upp-truncate-wei}, we obtain the following result, which will be used in the proof of Theorem \ref{Lin thm}.

\begin{proposition} \label{stability-cut-omega} For every $K_0 \geq 1$, there exists $\bar{N}_0 = \bar{N}_0 (n, K_0)>0$ such that the following assertion holds. For each $\mu \in A_{2}$ satisfying $[\mu]_{A_2} \leq K_0$, let $0< s < \tau <\infty$, and define
\[
\bar \mu(x) = \left\{\begin{array}{ll}
\mu(x) & \quad \text{if} \quad s \leq \mu(x) \leq \tau,\\ \smallskip
s & \quad \text{if} \quad \mu(x) < s, \\ \smallskip
\tau & \quad \text{if} \quad \mu(x) > \tau,
\end{array} \right.
\]
for $x \in \mathbb{R}^n$. Then, for each $R_0>0$,
\[
[[\bar \mu]]_{\textup{BMO}(B_{R_0}, \bar \mu)} \leq  \bar{N}_0 [[\mu]]_{\textup{BMO}(B_{R_0}, \mu)}.
\]
\end{proposition}

\subsection{Parabolic Aleksandrov-Bakelman-Pucci estimates} 
Recall that the weighted Lebesgue space $L^p(U, \mu)$ with $p \in [1, \infty)$ is defined in \eqref{weight-Lp-space}. The following notation for the parabolic weighted Sobolev space is also used in the paper.
\begin{definition}\label{solution space} Let $U\subset \mathbb{R}^{n+1}$, and $\omega\in A_{1+\frac{1}{n}}$. The weighted Sobolev space $\mathcal{W}^{2,1}_{n+1}(U,\, \omega)$ is defined by
\[
\mathcal{W}^{2,1}_{n+1}(U,\, \omega)
= \Big\{ u\in L^1_{\textup{loc}}(U):\ 
D^2 u \in L^{n+1}(U,\, \omega),\
u_t\in L^{n+1}(U,\, \omega^{-n})
\Big\}.
\]
\end{definition}

The parabolic Aleksandrov-Bakelman-Pucci maximum principle for strong solutions of linear parabolic equations was first proved by N. V. Krylov in \cite{Krylov-1976}, and independently by Tso \cite{Tso}. The following weighted version of the parabolic Aleksandrov-Bakelman-Pucci maximum principle for solutions to degenerate equations is a modification of the classical one in \cite{Tso}. For this purpose, we recall that the \emph{upper contact set} in $U$ of a function $u \in C(U)$ is defined by
\begin{align} \notag 
\Gamma^+(u)=\Big \{(x,t)\in U:\ &  
\exists \ l\in \R^n \textup{ such that }\ u(y,s)\leq u(x,t)+\langle l,\, y-x\rangle, \ \\ \label{upper contact}
    &\text{for all }\ (y,s) \in U, \ s \leq t \Big\}.
\end{align}
Also, for a given function $f$ defined on a domain $U \subset \mathbb{R}^{n+1}$, we denote $f^+(x,t) = \max\{f(x,t), 0\}$ for $(x,t) \in U$.
\begin{theorem}[Parabolic ABP estimate]\label{ABP}  For every $\nu\in (0,1)$, there exists a constant $N_0=N_0(n,\nu)>0$ such that the following assertion holds. Let $U \subset \mathbb{R}^{n+1}$ be a nonempty, open, and bounded set with parabolic boundary $\partial'U$. Suppose that there are $r>0$ and $y\in \mathbb{R}^n$ such that 
\[
U \subset B_r(y)\times \mathbb{R}.
\]
Suppose that $\omega\in{A_{1+\frac{1}{n}}}$, and that \eqref{elliptic} holds in $U$.
Then, for  every $u\in \mathcal{W}^{2,1}_{n+1}(U,\, \omega)\cap C(\overline{U})$, it holds that
\begin{equation} \label{ABP-est-11-16}
\sup_{U}u^+\leq \sup_{\partial' U}u^++N_0r^{\frac{n}{n+1}}\|(\mathcal{L}u)^+\|_{L^{n+1}(\Gamma^+(u),\, \omega^{-n})}.
\end{equation}
\end{theorem}
\noindent
We observe that  
\[
\text{det}(\omega(x) a_{ij}(x))\approx \omega(x)^n.
\]
In addition, due to the fact that $u\in \mathcal{W}^{2,1}_{n+1}(U, \omega)$ and \eqref{elliptic}, it follows that 
\[ 
\|\mathcal{L} u\|_{L^{n+1}(U, \omega^{-n})} \leq \| u_t\|_{L^{n+1}(U, \omega^{-n})} + N(n, \nu) \|D^2 u\|_{L^{n+1}(U, \omega)} < \infty.
\] 
Hence, the right hand side in \eqref{ABP-est-11-16} is well-defined. Therefore, Theorem \ref{ABP} follows from the parabolic Aleksandrov-Bakelman-Pucci estimate proved in \cite{Krylov-1976, Tso}. 

\subsection{Krylov-Safonov covering lemma} Given a measurable set $\Gamma\subset C_{r,\, \omega}$ and $q \in (0,1)$, let
\begin{equation}\label{cylinders set}
\mathcal{A} =\Big\{C = C_{\rho,\, \omega}(X)\subset C_{r,\, \omega}:\ \omega(C \cap \Gamma) \geq q \omega(C) \Big\}.
\end{equation}
In addition, for given constants $\eta \in (0,1)$ and $l>1$,  corresponding to each $C = C_{\rho,\, \omega}(X) \in \mathcal{A}$ with some $X = (x, s) \in \mathbb{R}^n \times \mathbb{R}$ and $\rho>0$, we denote
\begin{equation}\label{hat cylinder}
\begin{aligned}
\tilde {C} &= \tilde{C}_{\rho,\, \omega}(X) =B_{\eta \rho}(x)\times \big(s-{\rho}^2 (\omega^{-n})_{B_{\rho}(x)}^{1/n},\, s\big), \quad \text{and}\\
\hat{C} &= \hat{C}_{\rho,\, \omega}(X) =B_{\eta\rho}(x)\times \big(s +{\rho}^2 (\omega^{-n})_{B_{\rho}(x)}^{1/n},\, s + l {\rho}^2 (\omega^{-n})_{B_{\rho}(x)}^{1/n}\big).
\end{aligned}
\end{equation}
We then define
\begin{equation}\label{set E}
\tilde{E}(q, \eta) =\bigcup_{C \in \A } \tilde{C},\quad\text{and}\quad
\hat{E}(q, \eta, l) =\bigcup_{C \in \A } \hat{C}.
\end{equation}
We note that $\mathcal{A}$ depends on $\Gamma$, $q$, and $\omega$. However, for brevity, we do not explicitly write these factors. Similarly,  $\tilde E$ and $\hat E$ also depend on $\Gamma$, $q$, and $\omega$. Because each $C\in \A$ is open,  and its corresponding $\tilde C$ and $\hat C$ are open, the sets $\tilde E$ and $\hat E$ are open and measurable.

\smallskip
We now state the following covering lemma, whose unweighted case (i.e., $\omega=1$) was proved by N. V. Krylov and M. V. Safonov \cite{KrySa-1, KrySa-2}, see also \cite[Lemmas 5.1--5.2]{FeSa} and \cite[Lemma 9.2.6, p.~173]{Krylov-book}. 

\begin{lemma}[Weighted Krylov-Safonov covering lemma]\label{covering-1}  
Let $K_0\in[1,\infty)$, and let $\omega\in A_{1+\frac{1}{n}}$ satisfy $[\omega]_{A_{1+\frac{1}{n}}}\leq K_0$. Let $\Gamma\subset C_{r,\, \omega}$, and $\tilde E$ and $\hat E$ be defined as in \eqref{set E} with some $q\in (0,1)$, $\eta\in (0,1)$, and $l\in (1,\infty)$. Then, it holds that
\begin{equation}\label{conclusion covering}
\omega(\Gamma\setminus \tilde E) =0,\quad  \omega(\tilde E) \geq \xi_0 \omega(\Gamma),\quad \text{and}\quad \omega(\hat{E})\geq \xi_1 \omega(\tilde E),
\end{equation}
where 
\begin{equation}\label{hat-E-E}
\begin{split}
& \xi_0=\xi_0(n, K_0, q, \eta)=1+\frac{1-q-N(1-\eta^n)^{\lambda_0}}{3^{n+2}K_0}, \ \text{and}\\
& \xi_1=\xi_1(l) = \frac{l-1}{l+1}>0,
\end{split}
\end{equation}
in which $N=N(n,K_0)$ and $\lambda_0=\lambda_0(n, K_0)$ are positive constants.
\end{lemma} 
\noindent
The proof of this lemma is presented in Appendix \ref{Appendix-A}.

\section{Prop-up lemmas} \label{prop-up-section}
In this section, we establish two prop-up lemmas that are needed in the proofs of the main results. We begin with a prop-up lemma on slant cylinders.  For this purpose, let us introduce several notation. For $\rho>0$, $X_0=(x_0, t_0)$, $Y_0=(y_0, s_0)\in \mathbb{R}^{n}\times \mathbb{R}$ with $s_0>t_0$, we define the parabolic slant cylinder $V_{\rho}(X_0, Y_0)$ as
\begin{equation}\label{def-slant}
V_{\rho}(X_0, Y_0) =\left\{(x, t) \in \mathbb{R}^{n} \times \mathbb{R}: \big| x- \tfrac{y_0-x_0}{s_0-t_0}(t-t_0)-x_0\big| < \rho, \ t_0 < t< s_0 \right\}.
\end{equation}
Then, for a given number $K\in [1,\infty)$, we say $V_{\rho}(X_0, Y_0)$ satisfies a \emph{$K$--slant condition} if there exist $r>0$ and $y \in \mathbb{R}^n$ so that
\begin{equation}  \label{slant-cyl}
\left\{
\begin{split}
& \frac{r}{2K}\leq \rho\leq r, \quad V_{\rho}(X_0, Y_0)\subset B_r(y)\times \R, \quad \text{and}\\
& K^{-1} \rho(\omega^{-n})_{B_r(y)}^{1/n}|y_0-x_0| \leq s_0-t_0 \leq K {\rho}^2(\omega^{-n})_{B_r(y)}^{1/n}.
\end{split} \right.
\end{equation}
Here in \eqref{slant-cyl}, for the reader's convenience, we recall that
\[
(\omega^{-n})_{B_r(y)} = \fint_{B_r(y)} \omega(x)^{-n}dx.
\]
In addition, we write the parabolic boundary of $V_{\rho}(X_0,Y_0)$ as
\begin{equation}\label{bdy of slant}
\partial' V_{\rho}(X_0, Y_0) = \big(\overline{B}_{\rho}(x_0) \times \{t_0\}\big) \cup S
\end{equation}
with \[
S = \big\{(x, t): \big| x- \tfrac{y_0-x_0}{s_0-t_0}(t-t_0)-x_0\big| =\rho, \ t_0 < t < s_0\big\}.
\]
In addition, for the sake of brevity, with a given $\nu \in (0,1)$ and $K_0 \geq 1$, let us denote the set containing parabolic operators $\mathcal{L}$ defined in \eqref{parabolic operator} as
\begin{align} \notag 
\mathbb{L}_{\nu} (K_0)  =\Big\{\mathcal{L}: \ \mathcal{L} u = & u_t -\omega(x) a_{ij}(x,t) D_{ij}u, \ \text{where}  \  \omega \in A_{1+\frac{1}{n}},\  [\omega]_{A_{1+\frac{1}{n}}} \leq K_0,\\ \label{L-non-smooth}
&\text{and} \ (a_{ij}) \ \text{is symmetric and it satisfies} \ \eqref{elliptic} \Big\}.
\end{align}
Due to Lemma \ref{bmo-lemma}, and for the sake of brevity, with $r>0$ and $y \in \mathbb{R}^n$, we denote
\begin{equation}\label{BMO on ball}
[[\omega]]_{B_r(y), \omega}= \left(\frac{1}{\omega(B_r(y))}\int_{B_r(y)}|\omega(x)-(\omega)_{B_r(y)}|^{n+1} \omega(x)^{-n}\,dx \right)^{\frac{1}{n+1}}.
\end{equation}

Now, we give a lemma for non-negative supersolutions of $\mathcal{L}u\geq 0$ on slant cylinders, which is a modification of \cite[Lemma 4.3]{Cho-Fang-Phan}.

\begin{lemma}[Prop-up lemma on slant cylinders]\label{baby prop-up}  Let $\nu\in (0,1)$, $K_0\in [1,\infty)$, and $K\in [1,\infty)$. There exist constants $\delta_1=\delta_1(n,\nu, K_0, K)\in (0,1)$  and $\beta=\beta(n,\nu, K_0, K)\in (0,1)$ such that the following statement holds. Suppose that $V_{\rho}(X_0, Y_0)\subset B_r(y)\times \R$ with $X_0=(x_0, t_0),\, Y_0=(y_0, s_0)\in \R^{n}\times \R$ satisfies condition \eqref{slant-cyl}. Suppose also that $\mathcal{L} \in \mathbb{L}_{\nu}(K_0)$ with corresponding weight $\omega$ satisfying
\[
 [[\omega]]_{B_r(y),\, \omega}\leq \delta_1.
\]
Then, for all non-negative $u\in \mathcal{W}_{n+1}^{2,1}(V_{\rho}(X_0, Y_0),\, \omega)\cap C(\overline {V}_{\rho}(X_0, Y_0))$ satisfying
\[
\mathcal{L}u\geq 0 \quad \text{in }\ V_{\rho}(X_0, Y_0),
\]
it holds that 
\begin{equation}\label{slant conclusion}
\inf_{B_{\rho/2}(y_0)}u(\cdot, s_0)\geq \beta \inf_{B_{\rho}(x_0)}u(\cdot, t_0).
\end{equation}
\end{lemma}

\begin{proof} We construct a suitable barrier function and then apply Theorem \ref{ABP} to control the solution. After taking a linear translation, we may assume without loss of generality that
\[
X_0=(x_0, t_0)=(0,0).
\]
Define
\begin{equation}\label{m-def}
m=\inf_{B_{\rho}}u(\cdot, 0).
\end{equation}
We may further assume that $m>0$; otherwise, the conclusion \eqref{slant conclusion} is trivial.

\smallskip
Let
\begin{equation}\label{v-def}
v(x, t) = e^{-\lambda t}\rho^{-4} \Phi^2(x, t) \quad\text{for }\ (x, t) \in \overline{V}_{\rho}(X_0, Y_0),
\end{equation}
where $\lambda>0$ is some constant to be determined later, and
\[
\Phi(x,t)=\rho^2-|x-t\ell|^2\quad \text{with}\ \ell=\frac{y_0}{s_0}\in \mathbb{R}^n.
\]
We note that
\[
D_i\Phi = -2(x_i-t\ell_i) \quad \text{and} \quad D_{ij} \Phi = -2\delta_{ij},
\quad \text{where}\quad i, j= 1, 2,\ldots, n.
\]
Moreover,
\[
v_t(x, t) = e^{-\lambda t}\rho^{-4}\left\{-\lambda \Phi^2 +  4 \ell \cdot (x-t\ell) \Phi \right\},
\]
and
\[
D_i v  =2 e^{-\lambda t}\rho^{-4} \Phi D_i \Phi, \quad D_{ij} v =2 e^{-\lambda t}\rho^{-4}\big(D_i \Phi D_j \Phi + \Phi D_{ij}\Phi \big).
\]
Therefore, by \eqref{elliptic},
\begin{align*}
\mathcal {L}v
&=v_t - \omega(x) a_{ij} D_{ij}v\\
&= e^{-\lambda t}\rho^{-4}\big[-\lambda \Phi^2 +  4 \ell \cdot (x-t\ell) \Phi - 2\omega \big(a_{ij}D_i\Phi D_j \Phi + \Phi a_{ij}D_{ij}\Phi \big)\big] \\
& \leq  e^{-\lambda t}\rho^{-4}\big[-\lambda \Phi^2 +  4 \ell \cdot (x-t\ell) \Phi -2 \omega \nu|D\Phi|^2 + 4\omega \text{tr}(a_{ij})\Phi\big],
\end{align*}
As $|D\Phi|^2 = 4|x-t\ell|^2 = 4(\rho^2-\Phi)$, we have
\begin{align} \notag
 \mathcal {L}v  &\leq e^{-\lambda t}\rho^{-4}\left\{-\lambda \Phi^2 +  \big[ 4 \ell \cdot (x-t\ell) + 4\omega \text{tr}(a_{ij}) + 8\omega \nu \big]\Phi -8 \nu  \rho^2   \omega \right\} \\ \label{Lv}
 &= e^{-\lambda t}\rho^{-4}\big\{-\lambda \Phi^2 + g_1(x,t)\Phi -8 \nu  \rho^2  (\omega)_{B_r(y)}+ g_2(x, t)\big\},
\end{align}
where  
\begin{align*}
g_1(x,t)&=4 \ell \cdot (x-t\ell) + 4(\omega)_{B_r(y)}  \text{tr}(a_{ij}) + 8(\omega)_{B_r(y)} \nu,\\  
g_2(x,t)&= \big[ 4\big(\omega - (\omega)_{B_r(y)}  \big) \text{tr}(a_{ij}) + 8\big( \omega - (\omega)_{B_r(y)} \big) \nu \big] \Phi -8 \nu  \rho^2\big(\omega-  (\omega)_{B_r(y)}\big).
\end{align*}

We now control the right hand side of \eqref{Lv}. By the second formula of \eqref{slant-cyl} and the fact that $(\omega^{-n})_{B_r(y)}^{-1/n}\leq (\omega)_{B_r(y)}$ from \eqref{A-1-1/n}, we first note that
\[
| \ell \cdot (x-t\ell)| \leq |\ell| |x-t \ell| \leq K (\omega^{-n})_{B_r(y)}^{-1/n}  \leq K  (\omega)_{B_r(y)} \quad\text{for }\ (x, t) \in \overline{V}_{\rho}(X_0, Y_0).
\]
From this and use \eqref{elliptic} and $K\geq 1$, there exists $\tilde N = \tilde N(n, \nu)>0$ such that
\begin{equation}\label{g-1}
|g_1(x,t)|\Phi\leq \tilde NK(\omega)_{B_r(y)}\Phi \quad \text{in} \quad \overline{V}_{\rho}(X_0, Y_0).
\end{equation}
By Cauchy-Schwarz's inequality, we have
\[
\tilde NK(\omega)_{B_r(y)} \Phi\leq \frac{\tilde N^2K^2(\omega)_{B_r(y)}\Phi^2}{32\nu \rho^2}+8\nu \rho^2(\omega)_{B_r(y)}.
\]
Then, by taking 
\begin{equation} \label{lambda-def}
\lambda = N_1K^2{\rho}^{-2}(\omega)_{B_r(y)} \quad \text{with }\ N_1(n,\nu)=\frac{{\tilde N}^2}{16\nu},
\end{equation}
we see that
\begin{equation}\label{g-1 negative}
-\lambda \Phi^2 +  \tilde N K (\omega)_{B_r(y)} \Phi  - 8\nu \rho^2  (\omega)_{B_r(y)} \leq 0 \quad \text{in } \ \overline{V}_{\rho}(X_0, Y_0).
\end{equation}
Therefore, it follows from \eqref{Lv}, \eqref{g-1}, and \eqref{g-1 negative} that
\begin{align*}
 \mathcal{L}v\leq e^{-\lambda t}\rho^{-4}g_2(x, t) \quad \text{in } \ \overline{V}_{\rho}(X_0, Y_0).
\end{align*}
Moreover, using \eqref{elliptic} and the fact that $|\Phi| \leq \rho^2$ in $\overline{V}_{\rho}(X_0, Y_0)$, there exists $N = N(n, \nu)>0$ such that
\begin{equation} \label{est-g}
|g_2(x,t)| \leq N \rho^{2}\big| \omega(x) - (\omega)_{B_r(y)} \big|.
\end{equation}
Thus, the last two estimates imply that
\[ 
\mathcal{L} v\leq  Ne^{-\lambda t}\rho^{-2}\big| \omega(x) - (\omega)_{B_r(y)} \big| \quad \text{in } \ \overline{V}_{\rho}(X_0, Y_0). 
\]

Next, we consider the function 
\[ \phi(x,t) = mv(x,t)-u(x,t) \quad (x,t) \in V_{\rho}(X_0, Y_0),
\] 
with $m>0$ defined in \eqref{m-def}. As $\mathcal{L}u\geq 0$, it follows that
\begin{equation}\label{v-u}
\mathcal{L}\phi  \leq N(n,\nu)me^{-\lambda t}\rho^{-2}\big| \omega(x) - (\omega)_{B_r(y)} \big|\quad \text{in }\ \overline{V}_{\rho}(X_0, Y_0).
\end{equation}
As $u(\cdot, 0)\geq m$ on $B_\rho$, which is due to the definition of $m$ in \eqref{m-def}, we have
\[
\phi(x,0)\leq\frac{(\rho^2-|x|^2)^2}{\rho^{4}}m-m\leq 0\quad \text{for all}\quad x\in \overline{B}_\rho.
\]
Furthermore, by the definition of $v(x,t)$ and the fact that $u\geq 0$, it follows that
\[
\phi(x,t)=0-u(x,t)\leq 0\quad \text{for all}\quad (x,t)\in S,
\]
where $S$ is defined in \eqref{bdy of slant}. Thus, we have
\begin{equation}\label{bdy of v-u}
\phi(x,t)\leq 0\quad \text{on}\quad \partial'{V_{\rho}(X_0, Y_0)}.
\end{equation}
By \eqref{v-u}, \eqref{bdy of v-u}, and noting that $V_{\rho}(X_0, Y_0)\subset B_r(y)\times \R$, we apply Theorem \ref{ABP}  to \eqref{v-u} to obtain  
\[
\sup_{V_{\rho}(X_0, Y_0)} \phi \leq N(n,\nu)m\rho^{-2}r^{\frac{n}{n+1}}\bigl\|e^{-\lambda t}|\omega(x)-(\omega)_{B_r(y)}|\bigr\|_{L^{n+1}(V_{\rho}(X_0, Y_0),\, \omega^{-n})}.
\]
Note that
\begin{align*}
&\bigl\|e^{-\lambda t}|\omega(x)-(\omega)_{B_r(y)}|\bigr\|_{L^{n+1}(V_{\rho}(X_0, Y_0),\, \omega^{-n})}\\
&=\Big(\int_{V_{\rho}(X_0, Y_0)}e^{-\lambda(n+1)t}|\omega(x)-(\omega)_{B_r(y)}|^{n+1}\omega(x)^{-n}\, dxdt\Big)^{\frac{1}{n+1}}\\
&\leq\Big(\frac{1}{\lambda}\int_{B_r(y)}|\omega(x)-(\omega)_{B_r(y)}|^{n+1}\omega(x)^{-n}\,dxdt\Big)^{\frac{1}{n+1}}\\
&= \lambda^{-\frac{1}{n+1}}[\omega(B_r(y))]^{\frac{1}{n+1}} \Big(\frac{1}{\omega(B_r(y))}\int_{B_r(y)}|\omega(x)-(\omega)_{B_r(y)}|^{n+1}\omega(x)^{-n}\,dxdt\Big)^{\frac{1}{n+1}}\\
&\leq N(n,\nu)K^{-\frac{2}{n+1}}r^{\frac{n+2}{n+1}}[[\omega]]_{B_r(y),\, \omega},
\end{align*}
where in the last step we used the choice of $\lambda$ in \eqref{lambda-def}, $\omega(B_r(y))=N(n)(\omega)_{B_r(y)}r^n$, and $\rho\leq r$, which is due to \eqref{slant-cyl}.
Hence, combining the last two estimates and the first condition in \eqref{slant-cyl}, we obtain
\[
\sup_{V_{\rho}(X_0, Y_0)} \phi \leq N_2m K^{2-\frac{2}{n+1}}[[\omega]]_{B_r(y),\, \omega},\quad \text{where}\quad N_2=N_2(n,\nu)>0.
\]
On the other hand, from \eqref{lambda-def}, the second condition in \eqref{slant-cyl}, and the fact that $(\omega)_{B_r(y)}(\omega^{-n})_{B_r(y)}^{1/n}\leq K_0$ from \eqref{A-1-1/n}, we infer that
\begin{equation*}
\lambda s_0\leq  N_1(n,\nu)K^3K_0.
\end{equation*}
Recall the definition of $v$ in \eqref{v-def}, it follows by the last two estimates that
\begin{equation}\label{fun-u}
\begin{aligned}
u(x,s_0)
&=mv(x,s_0)-\phi(x,s_0)\\
&\geq \tfrac{9}{16}me^{-N_1K^3K_0}-N_2m  K^{2-\frac{2}{n+1}}[[\omega]]_{B_r(y),\, \omega} \quad \text{on} \quad \overline{B}_{\rho/2}(y_0).
\end{aligned}
\end{equation}
We now define
\begin{equation*}
\begin{aligned}
\delta_1 &= \delta_1(n,\nu,K_0,K) 
   = \frac{9}{32N_2  K^{2-\frac{2}{n+1}}} e^{-N_1K^3K_0}, \\
\beta &= \beta(n,\nu,K_0,K) 
   = \frac{9}{32} e^{-N_1K^3K_0} \in (0,1),
\end{aligned}
\end{equation*}
where $N_1 = N_1(n,\nu) > 0$ and $N_2 = N_2(n,\nu) > 0$. Then, from the definition of $m$ in \eqref{m-def}, we see that \eqref{fun-u} implies \eqref{slant conclusion}, and the lemma is proved.
\end{proof}
  
The next result is known as the prop-up lemma on cylinder $Q_{r,\omega}(Y)$, which is similar to \cite[Corollary 4.6]{Cho-Fang-Phan}. See also \cite[Lemma 4.3]{FeSa} for the un-weighted case. Note that the proof of Lemma \ref{prop-up} below is slightly different from that of \cite[Corollary 4.6]{Cho-Fang-Phan}.
\begin{lemma}[Prop-up Lemma]\label{prop-up}  Let $\nu\in (0,1)$ and $K_0\in [1,\infty)$. There exist constants $\delta_2=\delta_2(n,\nu, K_0)\in (0,1)$ and $\gamma=\gamma(n,\nu, K_0)>0$ such that the following statement holds. Suppose that $\mathcal{L} \in \mathbb{L}_{\nu}(K_0)$ with corresponding weight $\omega \in A_{1+\frac{1}{n}}$ satisfying
\[
[[\omega]]_{\textup{BMO}(B_{2r}(y), \omega)}\leq \delta_2
\]
for some $r>0$ and $Y= (y,s) \in \mathbb{R}^{n} \times \mathbb{R}$. Suppose also that $u\in \mathcal{W}_{n+1}^{2,1}(Q_{r,\,\omega}(Y),\,  \omega)\cap C(\overline{Q}_{r,\, \omega}(Y))$ is non-negative and satisfies
\[
\mathcal{L}u\geq 0 \quad \text{in }\ Q_{r,\, \omega}(Y).
\]
Then, for all $B_{\rho}(x_0)\subset B_r(y)$ and $t_0\in \big(s-r^2(\omega^{-n})_{B_r(y)}^{1/n},\, s\big)$, we have
\begin{equation}\label{prop-up conclu}
\inf_{B_{r/2}(y)}u\big(\cdot,\, s+r^2(\omega^{-n})_{B_r(y)}^{1/n}\big)\geq \left(\frac{\rho}{2r}\right)^{\gamma}\inf_{B_{\rho}(x_0)}u(\cdot, t_0).
\end{equation}
\end{lemma}

\begin{proof}
To begin with, we may assume without loss of generality that  
\[
\inf_{B_{\rho}(x_0)} u(\cdot, t_0) > 0,
\]
since otherwise \eqref{prop-up conclu} holds trivially. By applying the linear translation \((x,t) \mapsto (x-x_0, t-t_0)\), we may further assume that \((x_0,t_0) = (0,0)\).  
Due to this, we have  
\[
s \in \bigl(0,\, r^2(\omega^{-n})_{B_r(y)}^{1/n}\bigr),
\]  
and then choose $h \in [1, 2]$ such that  
\begin{equation} \label{h-def-1-1-26}
s + r^2(\omega^{-n})_{B_r(y)}^{1/n} = h\,r^2(\omega^{-n})_{B_r(y)}^{1/n}.
\end{equation}

\smallskip
Now, we follow the approach used in \cite[Lemma 4.3]{FeSa}.  The main idea is to build a stair with blocks of cylinders of certain heights, which allows us to climb from the base at $\{t = 0\}$ to the top level at $\{t = s + r^2(\omega^{-n})_{B_r(y)}^{1/n}\}$ by applying Lemma \ref{baby prop-up} to each block iteratively. To this end, we introduce some notation. First, let
\[
y^*=\left\{
\begin{aligned}
&\tfrac{\rho}{\rho-r}y\quad &&\text{if}\quad \rho\in (0,r),\\[4pt]
&\quad 0\quad &&\text{if}\quad \rho=r,
\end{aligned}\right.
\]
and define 
\[
y(\tau)=(1-\tau)y^*+ \tau y, \quad  r(\tau)=r\tau \quad \text{for }\ \tau \in [0,1]. 
\]
As such, we have
\[
B_{r(1)}(y(1))=B_r(y), \quad \text{and}\quad B_{r(\tau_0)}(y(\tau_0))=B_{\rho}(0)\quad \text{when}\quad \tau_0=\frac{\rho}{r}.
\]
Moreover, since $|y-y^*|\leq r$ due to $B_{\rho}\subset B_r(y)$, it follows that
\begin{equation}\label{inclusion-1}
B_{r(\tau_1)}(y(\tau_1))\subsetneq B_{r(\tau_2)}(y(\tau_2)), \quad  0<\tau_1<\tau_2\leq 1.
\end{equation}
Now, for $m = 0, 1, 2,\ldots$, let
\[
r_m=r(2^{-m})=\frac{r}{2^m}, \quad y_m=y(2^{-m}),\quad \text{and}\quad s_m=hr_m^2(\omega^{-n})_{B_{r_m}(y_m)}^{1/n}.
\]
Next, let $m_0\in \mathbb{N}$ be the unique integer such that
\begin{equation}\label{r-rho}
r_{m_0+1}=\frac{r}{2^{m_0+1}}\leq \rho<\frac{r}{2^{m_0}}=r_{m_0}.
\end{equation}
It then follows from \eqref{inclusion-1} and \eqref{r-rho} that
\begin{equation}\label{two chains}
B_{r_{m_0+1}}(y_{m_0+1})\subseteq B_{\rho}\subsetneq B_{r_{m_0}}(y_{m_0})\subsetneq\ldots\subsetneq B_{r_1}(y_1)\subsetneq B_{r_0}(y_0)=B_r(y).
\end{equation}

Now, let
\begin{equation}\label{choices}
\delta_2 =\delta_1(n,\nu, K_0, 32) \bar N_1\in (0,1),\quad \text{and}  \quad 
 \beta =\beta(n,\nu, K_0, 32)\in (0,1)
\end{equation}
where $\delta_1$ and $\beta$ are the constants defined in Lemma \ref{baby prop-up} with $K=32$, and $\bar N_1 =\bar N_1(n, 1+\frac{1}{n}, n+1, K_0)$ is the constant defined in Lemma \ref{bmo-lemma}. With this choice of $\delta_2$, we prove \eqref{prop-up conclu} by applying Lemma \ref{baby prop-up}  iteratively.

\smallskip
\noindent  
\textbf{Step 1}. Let $Z_{m_0 +1}=(y_{m_0+1},0)$ and $X_{m_0}=(x, s_{m_0})$, where $x\in B_{r_{m_0+1}}(y_{m_0})$. We construct the slant cylinder $V_{r_{m_0+2}}(Z_{m_0 +1}, X_{m_0})$ as defined in \eqref{def-slant}, and observe that
\[
V_{r_{m_0+2}}(Z_{m_0 +1}, X_{m_0})\subset B_{r_{m_0}}(y_{m_0})\times \mathbb{R}.
\]
Moreover, the condition \eqref{slant-cyl} is satisfied for $K=32$, that is,
\[
\frac{r_{m_0}}{2K}\leq r_{m_0+2}\leq r_{m_0},
\] 
and
\[
K^{-1}r_{m_0+2}(\omega^{-n})_{B_{r_{m_0}}(y_{m_0})}^{1/n} |x-y_{m_0+1}|\leq s_{m_0}\leq Kr_{m_0+2}^2(\omega^{-n})_{B_{r_{m_0}}(y_{m_0})}^{1/n}.
\]
Therefore, it follows from  Lemma \ref{baby prop-up}, \eqref{two chains}, and the choice of $\delta_2$ in \eqref{choices} that
\begin{equation*}
u(x, s_{m_0})\geq \beta\inf_{B_{r_{m_0+2}}(y_{m_0+1})}u(\cdot, 0)\geq \beta \inf_{B_{\rho}} u(\cdot, 0) > 0.
\end{equation*}
Because $x \in B_{r_{m_0+1}}(y_{m_0})$ is arbitrary, we infer that
\begin{equation}\label{step-1}
\inf_{B_{r_{m_0+1}}(y_{m_0})}u(\cdot, s_{m_0})\geq\beta \inf_{B_{\rho}} u(\cdot, 0).
\end{equation}

\noindent  
\textbf{Step 2}. Let $Z_{m_0}=(y_{m_0}, s_{m_0})$ and $X_{m_0-1}=(x, s_{m_0-1})$ with $x\in B_{r_{m_0}}(y_{m_0-1})$. Then, 
\[
V_{r_{m_0+1}}(Z_{m_0}, X_{m_0-1})\subset B_{r_{m_0-1}}(y_{m_0-1})\times \mathbb{R}.
\] 
Moreover, $\frac{r_{m_0-1}}{2K}\leq r_{m_0+1}\leq r_{m_0-1}$ and
\[
K^{-1}r_{m_0+1}(\omega^{-n})_{B_{r_{m_0-1}}(y_{m_0-1})}^{1/n} |x-y_{m_0}|\leq s_{m_0-1}-s_{m_0}\leq Kr_{m_0+1}^2(\omega^{-n})_{B_{r_{m_0-1}}(y_{m_0-1})}^{1/n}
\]
also hold for $K=32$. As in {\bf Step 1}, we apply Lemma \ref{baby prop-up}, and use \eqref{two chains} and \eqref{step-1} to conclude that
\begin{equation}\label{step-2}
\inf_{B_{r_{m_0}}(y_{m_0-1})}u(\cdot, s_{m_0-1})\geq\beta\inf_{B_{r_{m_0+1}}(y_{m_0})}u(\cdot, s_{m_0})\geq\beta^2 \inf_{B_{\rho}} u(\cdot, 0).
\end{equation}

\noindent  
\textbf{Step 3}. Let $Z_m =(y_{m}, s_m)$ for each $m = m_0-1, m_0-2, \ldots, 1$. We construct the slant cylinders $V_{r_{m+1}}(Z_{m}, X_{m-1})\subset B_{r_{m-1}}(y_{m-1})\times \mathbb{R}$ as in the previous steps, where $X_{m-1}=(x, s_{m-1})$ with $x\in B_{r_m}(y_{m-1})$. Then, apply Lemma \ref{baby prop-up} iteratively on each slant cylinder $V_{r_{m+1}}(Z_{m}, X_{m-1})$. Based on \eqref{step-1}, \eqref{step-2}, and by induction, we have
\begin{equation}\label{step-3}
\inf_{B_{r/2}(y)}u(\cdot, s+ r^2(\omega^{-n})_{B_r(y)}^{1/n})\geq\beta^{m_0+1} \inf_{B_{\rho}} u(\cdot, 0).
\end{equation}

We now observe from \eqref{r-rho} that
\[
m_0+1<\log_2\tfrac{2r}{\rho}\leq m_0+2.
\]
Note that $\beta\in (0,1)$, the previous estimate and \eqref{step-3} imply that
\begin{equation*}\label{claim-conclu}
\inf_{B_{r/2}(y)}u(\cdot, s+ r^2(\omega^{-n})_{B_r(y)}^{1/n})\geq \left(\frac{\rho}{2r}\right)^{\gamma} \inf_{B_{\rho}} u(\cdot, 0)
\end{equation*}
with
\[
 \gamma=\gamma(n,\nu, K_0)=-\log_2{\beta}.
\]
Hence, \eqref{prop-up conclu} is proved, and the proof is completed.
\end{proof}

\section{Evans type estimates and Proof of Theorem \ref{M-thm}} \label{pf of Thm1} \label{Evans-sec}
The main goal in this section is to prove Theorem \ref{M-thm}, which is a main ingredient to prove Theorem \ref{Lin thm}. As mentioned in the introduction, in several steps, we need classical results on the existence and uniqueness of solutions to parabolic equations in non-divergence form.  We truncate $\omega$, and then regularize the coefficients $(a_{ij})$ and $\omega$. Due to this, and for the sake of brevity, given $\nu \in (0,1)$ and $K_0 \geq 1$, and for $k \in \mathbb{N}$, we introduce the following set of operators $\mathcal{L}$  defined in \eqref{parabolic operator} with uniformly elliptic and bounded smooth coefficients:
\begin{align} \notag
\mathbb{L}_{\nu, k}^\infty(K_0) & =\Big\{\mathcal{L}:  \ \mathcal{L} u = u_t -\omega(x) a_{ij}(x,t) D_{ij}u ,\ \text{where} \ \omega\  \text{and}\ (a_{ij}) \ \text{are smooth}, \\ \label{L-smooth}
& \frac{1}{k} \leq \omega \leq k, \ [\omega]_{A_{1+\frac{1}{n}}} \leq K_0, \ \text{and}\ (a_{ij}) \ \text{is symmetric and it satisfies} \ \eqref{elliptic} \Big\}.
\end{align}
One can consider the set of operators $\mathbb{L}_{\nu, k}^\infty(K_0)$ are from $\mathbb{L}_{\nu}(K_0)$ defined in \eqref{L-non-smooth}, in which the coefficients $\omega, (a_{ij})$ are regularized, and $\omega$ is also truncated. Obviously, 
\[ 
\mathbb{L}_{\nu, k}^\infty(K_0) \subset \mathbb{L}_{\nu}(K_0). 
\]
We also emphasize that results in Section \ref{prop-up-section} hold for $\mathcal{L} \in \mathbb{L}_{\nu}(K_0)$.

\subsection{Auxiliary results} We provide several lemmas needed for the proof of Theorem \ref{M-thm}. We begin with the following lemma.
\begin{lemma}\label{large density} Let $\nu\in (0,1)$ and $K_0\in [1,\infty)$. There exist $\bar \delta=\bar \delta(n,\nu, K_0)>0$ sufficiently small, $q_0=q_0(n,\nu, K_0)\in (0,1)$ sufficiently close to $1$, and $\kappa=\kappa(n,\nu,K_0)>0$ such that the following assertion holds.  Suppose that  $\mathcal{L} \in  \mathbb{L}_{\nu, k}^\infty(K_0)$ for some $k\in \mathbb{N}$, and its corresponding weight $\omega \in A_{1 + \frac{1}{n}}$ satisfies
\begin{equation} \label{omega-BMO-existence}
[[\omega]]_{\textup{BMO}(B_{2r},\, \omega)}\leq \bar \delta
\end{equation}
for some $r>0$. Suppose also that $u\in \mathcal{W}^{2,1}_{n+1}(C_{r,\,\omega},\, \omega)\cap C(\overline{C}_{r,\,\omega})$ is a non-negative function satisfying 
\[
\mathcal{L} u \geq 0  \  \text{in}\ C_{r,\, \omega}, \quad \text{and}  \quad \omega\big(\{(x,t)\in C_{r,\,\omega}: \mathcal{L}u(x,t)\geq \omega(x)\}\big)\geq q_0\, \omega(C_{r,\, \omega}).
\]  
Then 
\begin{equation}\label{conclu-1}
\inf_{B_{r/2}}u(\cdot,0)\geq \kappa r^2.
\end{equation}
\end{lemma}

\begin{proof} For the given $\nu \in (0,1)$ and $K_0 \geq 1$, let $\delta_2 = \delta_2(n, \nu, K_0)$ and $\gamma =\gamma(n, \nu, K_0)$ be the constants defined in Lemma \ref{prop-up}. Let 
\begin{equation}\label{choice of k_0}
\kappa_0=4^{-\gamma}, \quad \text{and}  \quad \bar \delta=\min\left\{\delta_2, \, \frac{\bar N_1\kappa_0}{4N_0  K_0^{1/(n+1)}}\right\},
\end{equation}
where $N_0 = N_0(n, \nu)$ is the constant defined in Theorem \ref{ABP}, and $\bar N_1 =\bar N_1(n, 1+\frac{1}{n}, n+1, K_0)$ is the constant defined in Lemma \ref{bmo-lemma}.  We note that $\kappa_0$ and $\bar \delta$ depend only on $n,\nu$ and $K_0$.  

\smallskip
Next, let us denote 
\[
U=\{\mathcal{L}u\leq \omega(x)\}\cap C_{r,\, \omega} \quad \text{and}\quad q=\frac{\omega(\{\mathcal{L}u\geq \omega(x)\}\cap C_{r,\, \omega})}{\omega(C_{r,\, \omega})}\in [0,1].
\]
Our goal is to search for $q_0 \in (0,1)$ such that \eqref{conclu-1} holds if $q \in [q_0, 1]$. We note that
\begin{equation}\label{density-q}
\omega(U)= (1-q)\omega(C_{r,\, \omega}).
\end{equation}
Since $\mathcal{L} \in \mathbb{L}_{\nu, k}^\infty(K_0)$, it follows from the classical theory for parabolic equations with smooth, uniformly elliptic and bounded coefficients that  there exist unique  solutions $u_0,\, u_1,\, u_2\in \mathcal{W}^{2,1}_{n+1}(C_{r,\omega}, \omega) \cap C(\overline{C}_{r,\omega})$ solving the equations
\begin{equation*}
\left\{
\begin{aligned}
\mathcal{L} u_0 &= (\omega)_{B_r} &&\text{in } C_{r,\, \omega},\\
u_0 &= 0 &&\text{on } \partial' C_{r,\, \omega},
\end{aligned}\right.
\end{equation*}
and
\begin{equation*}
\left\{
\begin{aligned}
\mathcal{L} u_1 &= (\omega)_{B_r}-\omega(x) &&\text{in } C_{r,\, \omega},\\
u_1 &= 0 &&\text{on } \partial' C_{r,\, \omega},
\end{aligned}\right.
\qquad
\left\{
\begin{aligned}
\mathcal{L} u_2 &= \omega(x)\chi_{U} &&\text{in } C_{r,\, \omega},\\
u_2 &= 0 &&\text{on } \partial' C_{r,\, \omega}.
\end{aligned}\right.
\end{equation*}
Here, we note that since $\frac{1}{k} \leq \omega \leq k$ and $\omega$ is smooth, we actually seek the solutions $u_0, u_1, u_2$ in the standard un-weighted Sobolev space $W^{2,1}_p(C_{r,\, \omega})$ for $p \in (1, \infty)$. As we can take $p \geq n+1$ as we need, we get $u_0, u_1, u_2 \in C(\overline{C}_{r,\, \omega})$ by the Sobolev embedding theorem.

\smallskip
We observe that  $\mathcal{L}u\geq \mathcal{L}(u_0-u_1-u_2)$ in $C_{r,\, \omega}$ and $u \geq u_0-u_1-u_2=0$ on $\partial'C_{r,\, \omega}$. It follows from the comparison principle (or Theorem \ref{ABP}) that
\begin{equation}\label{u-comparison}
u\geq u_0-u_1-u_2 \quad \text{in}\quad C_{r,\, \omega}.
\end{equation}

We estimate each term on the right hand side of \eqref{u-comparison}. Firstly, we control $u_0$ from below. Let us denote
\[
\varphi(x,t)=\frac{(\omega)_{B_r}}{r^2}(r^2-|x|^2)\in C^\infty(\overline{C}_{r,\, \omega}).
\]
Also, let $\psi\in C^\infty(\overline{C}_{r,\, \omega})$  solving
\begin{equation*}
\left\{
\begin{aligned}
\mathcal{L}\psi&=0 \quad &&\text{in}\quad C_{r,\, \omega},\\
\psi&=\varphi \quad &&\text{on}\quad \partial'C_{r,\, \omega}.
\end{aligned}\right.
\end{equation*}
It then follows from the maximal principle that
\begin{equation*}\label{varphi-psi}
0\leq \psi\leq (\omega)_{B_r} \quad \text{in}\quad \overline{C}_{r,\omega}.
\end{equation*}
On the other hand, note that
\[
\psi \geq \tfrac{3}{4}(\omega)_{B_r} \quad \text{on}\quad B_{r/2}\times \{-r^2(\omega^{-n})_{B_r}^{1/n}\}.
\]
Then, it follows from the choice of $\bar \delta$ in \eqref{choice of k_0} and Lemma \ref{prop-up} that
\begin{equation}\label{top lower bdd}
\inf_{B_{r/2}}\psi(\cdot, 0)\geq \left(\tfrac{1}{4}\right)^{\gamma}(\omega)_{B_r}=\kappa_0 (\omega)_{B_r},
\end{equation}
where the last equality follows from the definition of $\kappa_0$ in \eqref{choice of k_0}. Now, we define
\[
\Psi(x,t)=\big(t+r^2(\omega^{-n})_{B_r}^{\frac{1}{n}}\big)\psi(x,t)\quad \text{in}\quad \overline{C}_{r,\, \omega}.
\]
We see that
\begin{equation*}
\left\{
\begin{aligned}
\mathcal{L}\Psi&=\psi+\big(t+r^2(\omega^{-n})_{B_r}^{\frac{1}{n}}\big)\mathcal{L}\psi\leq (\omega)_{B_r}\quad &&\text{in} \quad C_{r,\, \omega},\\
\Psi&=0\quad &&\text{on}\quad \partial'C_{r,\, \omega}.
\end{aligned}\right.
\end{equation*}
Then, it follows from the comparison principle that
\[
u_0\geq \Psi \quad \text{in} \quad \overline{C}_{r,\, \omega}.
\]
Hence, by \eqref{top lower bdd} and the fact that $(\omega)_{B_r}(\omega^{-n})_{B_r}^{1/n}\geq 1$ from \eqref{A-1-1/n}, we obtain
\begin{equation}\label{bdd of u_0}
\inf_{B_{r/2}}u_0(\cdot, 0)\geq \inf_{B_{r/2}}\Psi(\cdot, 0)=\kappa_0(\omega)_{B_r}(\omega^{-n})_{B_r}^{1/n}r^2
\geq \kappa_0 r^2.
\end{equation}

Secondly, we control $u_1$. It follows from Theorem \ref{ABP} and \eqref{A-1-1/n} that
\begin{align*}
u_1
&\leq N_0 r^{\frac{n}{n+1}}\left(\int_{\Gamma^+(u_1)}\frac{|\omega(x)-(\omega)_{B_r}|^{n+1}}{\omega(x)^{n}}\, dxdt\right)^{\frac{1}{n+1}}\\
&\leq N_0 r^{\frac{n}{n+1}}\left[r^{n+2}(\omega)_{B_r}(\omega^{-n})_{B_r}^{\frac{1}{n}}\right]^{\frac{1}{n+1}}\left(\frac{1}{\omega(B_r)}\int_{B_r}\frac{|\omega(x)-(\omega)_{B_r}|^{n+1}}{\omega(x)^{n}}\, dx\right)^{\frac{1}{n+1}}\\
&\leq N_0 K_0^{\frac{1}{n+1}}r^2[[\omega]]_{B_r,\, \omega} \leq N_0 \bar N_1^{-1} K_0^{\frac{1}{n+1}}r^2[[\omega]]_{\textup{BMO}(B_{2r},\, \omega)},
\end{align*}
where $N_0=N_0(n, \nu)>0$ is defined in Theorem \ref{ABP}, and $\bar N_1 =\bar N_1(n, 1+\frac{1}{n}, n+1, K_0)>0$ is defined in Lemma \ref{bmo-lemma}. Thus, by the condition of $[[\omega]]_{\textup{BMO}(B_{2r},\, \omega)}\leq \bar \delta$ and the choice of $\bar \delta$ in \eqref{choice of k_0}, it follows that
\begin{equation}\label{bdd of u_1}
u_1\leq \frac{1}{4}\kappa_0r^2\quad \text{in}\quad C_{r,\, \omega}.
\end{equation}

Lastly, to control $u_2$, we apply Theorem \ref{ABP} and   \eqref{density-q} to obtain 
\begin{align} \notag
u_2 &\leq N_0 r^{\frac{n}{n+1}}(1-q)^{\frac{1}{n+1}}\big[\omega(C_{r,\, \omega})\big]^{\frac{1}{n+1}}\qquad  \\ \label{bdd of u_2}
&\leq N_0 (1-q)^{\frac{1}{n+1}}K_0^{\frac{1}{n+1}}r^2,
\end{align}
where we used \eqref{cylinder measure} in the last step.

\smallskip
Combining \eqref{bdd of u_0}, \eqref{bdd of u_1}, and \eqref{bdd of u_2} with \eqref{u-comparison}, we infer that 
\[
\inf_{B_{r/2}}u(\cdot,0)\geq \kappa_0r^2-\frac{1}{4}\kappa_0r^2-N_0(1-q)^{\frac{1}{n+1}}K_0^{\frac{1}{n+1}}r^2.
\]
Now, we choose $q_0=q_0(n, \nu,K_0)\in (0,1)$ sufficiently close to $1$ so that  
\[
N_0(1-q)^{\frac{1}{n+1}}K_0^{\frac{1}{n+1}}\leq \frac{\kappa_0}{4}, \qquad \forall\ q\in [q_0,1].
\] 
Then, by setting $\kappa=\frac{\kappa_0}{2}$, we see that
\[
\inf_{B_{r/2}}u(\cdot, 0)\geq \kappa r^2.  
\]
The proof of the lemma is completed.
\end{proof}

Based on Lemmas \ref{prop-up} and \ref{large density}, we have the following corollary, in which the equation is defined on $Q_{r,\, \omega}$ instead of $C_{r,\, \omega}$. The corollary will be used in the proof of Corollary \ref{I(q) bdd} below.
\begin{corollary}\label{M bdd} Let $\nu,\, K_0,\, \bar \delta$, and $q_0$ be as in Lemma \ref{large density}. Assume that $\mathcal{L} \in \mathbb{L}_{\nu,k}^\infty(K_0)$ for some $k \in \mathbb{N}$, and that $\omega$ satisfies \eqref{omega-BMO-existence}. Suppose that $u\in \mathcal{W}^{2,1}_{n+1}(Q_{r,\, \omega},\, \omega)\cap C(\overline {Q_{r,\,\omega}})$ is a non-negative function satisfying
\[
\mathcal{L} u  \geq 0 \ \text{in}\  Q_{r,\, \omega}, \quad \text{and} \quad  \omega\big(\big\{(x,t)\in C_{r,\, \omega}:\mathcal{L}u(x,t)\geq \omega\big(x)\}\big)\geq q_0\omega(C_{r,\, \omega}).
\]
Then 
\begin{equation}\label{conclu-2}
\inf_{B_{r/2}}u(\cdot, r^2(\omega^{-n})_{B_r}^{1/n})\geq \kappa_1 r^2,
\end{equation}
where $\kappa_1=\kappa_1(n, \nu, K_0)>0$.
\end{corollary}  
 
\begin{proof}
First, by Lemma \ref{large density}, there exists a constant $\kappa=\kappa(n,\nu, K_0)>0$ such that
\[
\inf_{B_{r/2}}u(\cdot,0)\geq \kappa r^2.
\]
Also, under the choice of the $\bar \delta$ in \eqref{choice of k_0} in the proof of Lemma \ref{large density}, we can apply Lemma \ref{prop-up} in which $Y=(0,0)$ and $\rho=\tfrac{1}{2}r$, to get
\[
\inf_{B_{r/2}}u(\cdot,\,  r^2(\omega^{-n})_{B_r}^{1/n})\geq 4^{-\gamma}\kappa r^2,
\]
for $\gamma=\gamma(n,\nu, K_0)$ defined in Lemma \ref{prop-up}.  Thus, by setting 
\begin{equation}\label{kappa-1}
\kappa_1=\kappa_1(n,\nu,K_0)=4^{-\gamma}\kappa,
\end{equation}
we obtain the assertion in the corollary.    
\end{proof}

We also need the following result, which asserts that ``if you win locally in time, then you win globally".
\begin{lemma}\label{local-global} Let $\nu \in (0,1)$ and $K_0 \geq 1$, and let $\mathcal{L} \in \mathbb{L}_{\nu, k}^\infty(K_0)$ for some $k \in \mathbb{N}$ with its corresponding $\omega$. Let $E\subset Q_{r,\omega}$ be an open nonempty set, and let $f,\ g\in L^{n+1}(Q_{r,\, \omega},\, \omega^{-n})$ be non-negative satisfying $g = g \chi_{E}$. Suppose that  $u$, $v \in \mathcal{W}^{2,1}_{n+1}(Q_{r,\, \omega}, \omega)\cap C(\overline{Q}_{r,\, \omega})$ solve the equations
\begin{equation}\label{eqns in lemma}
\left\{
\begin{aligned}
\mathcal{L}u&=f\quad &&\text{in}\quad Q_{r,\, \omega},\\
u&=0\quad &&\text{on}\quad \partial'Q_{r,\, \omega},
\end{aligned}\right.
\quad \text{and}\quad
\left\{
\begin{aligned}
\mathcal{L}v&=g \quad &&\text{in}\quad Q_{r,\, \omega},\\
v&=0\quad &&\text{on}\quad \partial'Q_{r,\, \omega}.
\end{aligned}\right.
\end{equation}
Assume that for each point $X_0=(x_0,t_0)\in E$, there exists an open cylinder $U\subset Q_{r,\, \omega}$ with $X_0\in U$ such that if $u_1$, $v_1\in \mathcal{W}^{2,1}_{n+1}(U,\, \omega)\cap C(\overline U)$ solve
\begin{equation}\label{local eqns}
\left\{
\begin{aligned}
\mathcal{L}u_1&=f\ &&\text{in}\  U,\\
u_1&=0\ &&\text{on}\ \partial'U,
\end{aligned}\right.
\quad\text{and}\quad
\left\{
\begin{aligned}
\mathcal{L}v_1&=g\ &&\text{in}\ U,\\
v_1&=0\ &&\text{on}\ \partial'U,
\end{aligned}\right.
\end{equation}
we have 
\begin{equation} \label{u-1-v-1-1-28}
u_1(x_0,t_0)\geq v_1(x_0,t_0). 
\end{equation}
Then,
\[
u\geq v\quad \text{in}\quad Q_{r,\, \omega}.
\]
\end{lemma}

\begin{proof} We modify the approach used in \cite[Lemma 4, p. 171]{Krylov-book}. By replacing $u$ by $u + \tau_0[t + r^2(\omega^{-n})_{B_r}^{1/n}]$ and then sending $\tau_0 \rightarrow 0^+$, it is sufficient to prove the lemma under the assumption that $f \geq \tau_0>0$ with some $\tau_0>0$. Note also that by letting a sequence of closed increasing subsets $\{E_{k}\}$ of $E$ that converges to $E$ and let $v_k \in  \mathcal{W}^{2,1}_{n+1}(Q_{r,\omega}, \omega) \cap C(\overline{Q}_{r,\omega})$ be solutions to the equation
\[
\left\{
\begin{aligned}
\mathcal{L}v_k&=g \chi_{E_k}\quad &&\text{in}\quad Q_{r,\, \omega},\\
v_k&=0\quad &&\text{on}\quad \partial'Q_{r,\, \omega}
\end{aligned} \right.
\]
we see from Theorem \ref{ABP} that $v_k \leq v$ and $v_k \rightarrow v$ uniformly in $\overline{Q}_{r,\omega}$ as $k \rightarrow \infty$. We also note that for the given $\mathcal{L}, E, f$, if $g$ is a function such that the conclusion \eqref{u-1-v-1-1-28} holds then the same also holds with $g \chi_{E_k}$ due to the maximum principle. Hence, instead of working on $v_k$, $E_k$ and $g\chi_{E_k}$, it is sufficient to prove the theorem under the assumption that $E$ is closed and 
\begin{equation}\label{equiv claim}
(1+\sigma)u>v \quad \text{in}\quad Q_{r,\, \omega},\quad \text{for all}\quad \sigma>0.
\end{equation}

To prove \eqref{equiv claim}, we use a contradiction argument. Suppose that \eqref{equiv claim} is not true; then there exists $\sigma_0>0$ such that 
\[
M = \max_{\overline{Q}_{r,\, \omega}}\, [v-(1+\sigma_0)u]>0.
\]
Let us denote by $X_0= (x_0, t_0) \in \overline{Q}_{r,\, \omega}$ a point such that 
\begin{equation}\label{max value}
[v-(1+\sigma_0)u](X_0)= M >0.
\end{equation}
As $u=v=0$ on $\partial'Q_{r,\, \omega}$,  we see that $X_0 \notin \partial' Q_{r,\omega}$. On the other hand, as
\[
\mathcal{L}[v-(1+\sigma_0)u]=-(1+\sigma_0)f \leq -(1+\sigma_0) \tau_0  <0 \quad \text{in } Q_{r,\, \omega}\setminus E,
\]
and $Q_{r,\, \omega}\setminus E$ is open as $E$ is closed, it follows from the proof of the maximum principle that $X_0 \notin  Q_{r,\, \omega}\setminus E$. Thus, $X_0=(x_0, t_0)\in E$.

\smallskip
By the assumption in the lemma, there exists a cylinder $U\subset Q_{r,\, \omega}$ with $X_0\in U$
such that 
\begin{equation}\label{inequality 1}
u_{1}(X_0)\geq v_1(X_0),
\end{equation}
where $u_1,\, v_1\in \mathcal{W}^{2,1}_{n+1}(U,\, \omega)\cap C(\overline{U})$ are the solutions to equations \eqref{local eqns}.  Observe also as $f \geq \tau_0>0$, it follows from the strong maximum principle that
\begin{equation} \label{u-1-X-zero-1-2-26}
u_1(X_0) >0.
\end{equation} 

Now, let $w = M + (1+\sigma_0)(u-u_1)$. From \eqref{eqns in lemma} and \eqref{local eqns}, we see that
\[
\mathcal{L} w = 0 \quad \text{and} \quad \mathcal{L}(v-v_1)=0 \quad \text{in }  \ U,
\]
and
\[
v - v_1 = v \leq  w \quad \text{on} \quad \partial' U.
\]
Hence, it follows from the maximum principle that
\[
v - v_1 \leq w \quad \text{in} \quad U.
\]
Therefore,
\[
v(X_0) \leq v_1(X_0) + w(X_0)  = v_1(X_0) + M + (1+\sigma_0) u(X_0)  - (1+\sigma_0) u_1(X_0).
\]
From this, \eqref{inequality 1}, and \eqref{u-1-X-zero-1-2-26}, we infer that
\begin{align*}
M = v(X_0) - (1+\sigma_0) u(X_0) & \leq v_1(X_0)  - (1+\sigma_0) u_1(X_0) + M \\
& < v_1(X_0) - u_1(X_0) + M \leq M.
\end{align*}
Hence, we obtain a contradiction. The proof is completed.
\end{proof}

\subsection{An auxiliary function} \smallskip  For $r>0$, we recall that 
\begin{align*}
&C_{r,\, \omega}=B_r\times \big(-r^{2}(\omega^{-n})_{B_r}^{1/n},\, 0\big),\qquad\text{and}\\[2pt]
&Q_{r,\, \omega}=B_r\times \big(-r^{2}(\omega^{-n})_{B_r}^{1/n},\, r^{2}(\omega^{-n})_{B_r}^{1/n}\big).
\end{align*}

\begin{definition}\label{auxiliary function} Let $\nu \in (0,1)$ and $K_0 \in [1, \infty)$, and  let $\mathcal{L} \in \mathbb{L}_{\nu, k}^\infty(K_0)$ for some $k\in \mathbb{N}$ with its corresponding weight $\omega \in A_{1+\frac{1}{n}}$. For each $r>0$, define the function $I : [0,1] \rightarrow [0, \infty)$ by
\begin{align*}
I(q) =\inf \Big\{ & \inf_{x\in B_{r/2}} u\big(x,\, r^{2}(\omega^{-n})_{B_r}^{1/n}\big):\ u\in \mathcal{W}^{2,1}_{n+1}(Q_{r,\, \omega},\,\omega) \cap\, C(\overline{Q}_{r,\, \omega}), \ u\geq 0,  \\  
& \qquad \qquad \mathcal{L}u\ge 0 \ \text{in}\ Q_{r,\, \omega},\ \text{and} \quad \omega\!\left(\big\{ C_{r,\,\omega}:
\mathcal{L}u\ge \omega \big\}\right)
\ge q\,\omega(C_{r,\, \omega}) \Big\},
\end{align*}
for $q\in [0,1]$.
\end{definition}
\noindent
We note that $I(q)$ also depends on $\mathcal{L}$ and $r$. However, we suppress those dependent variables for notational simplicity. Obviously, $I(q)$ is increasing in $q\in [0,1]$. The main goal of this subsection is to derive a lower bound of $I(q)$ uniformly with respect to $\mathcal{L}$, which is stated and proved in Corollary \ref{I(q) bdd} below.

\smallskip
For settings, with the given numbers $\nu \in (0,1)$ and $K_0 \in [1, \infty)$, let 
\begin{equation}\label{def q_0}
q_0=q_0(n,\nu, K_0)\in (0,1)
\end{equation}
be the constant defined in Lemma \ref{large density}. We can then choose a constant 
\begin{equation}\label{def eta_0}
\eta_0 = \eta_0(n,\nu, K_0)\in (0,1)
\end{equation}
sufficiently close to $1$ so that the constant $\xi_0$, defined in \eqref{hat-E-E} with $q_0$ in place of $q$ and $\eta_0$ in place of $\eta$, satisfies
\begin{equation}\label{def xi_0}
\xi_0  = \xi_0 (n, \nu, K_0) =1+\frac{(1-q_0)}{2}3^{-n-2}K_0^{-1}\in (1,2).
\end{equation}
We further define
\begin{equation}\label{definition l_0}
l_0 = l_0 (n, \nu, K_0) =\frac{3\xi_0+1}{\xi_0-1} >1.
\end{equation}
Note that
\begin{equation}\label{def l_0}
\xi_0\xi_1=\frac{1+\xi_0}{2}, \quad \text{where}\quad \xi_1(l_0)=\frac{l_0-1}{l_0+1}.
\end{equation}
Given a measurable set $\Gamma\subset C_{r,\, \omega}$, using the definition \eqref{set E} with $q= q_0$, $\eta= \eta_0$, and $l = l_0$, we obtain the following sets
\begin{equation}\label{def E_0}
\tilde E_0=\tilde E(q_0, \eta_0), \quad \text{and}\quad \hat E_0=\hat E(q_0, \eta_0, l_0).
\end{equation}
Moreover, by Lemma \ref{covering-1} and \eqref{def l_0}, it follows that
\begin{equation}\label{mea hat E}
\omega(\hat E_0)\geq \xi_0\xi_1\omega(\Gamma)=\frac{1+\xi_0}{2}\omega(\Gamma).
\end{equation}

We begin with the following lemma, providing an initial step in estimating $I(q)$.

\begin{lemma}\label{alternative lemma}
Let $\nu \in (0,1)$ and $K_0\in [1,\infty)$. There exist constants $\delta_0=\delta_0(n,\nu, K_0)\in (0,1)$, $\xi=\xi(n,\nu, K_0)\in (0,1)$, $N_1 = N_1(n, \nu, K_0)>0$, and $N_2 = N_2(n, \nu, K_0)>1$ such that the following assertions hold. Suppose that $\mathcal{L} \in \mathbb{L}_{\nu, k}^\infty(K_0)$ for some $k \in \mathbb{N}$, and its corresponding weight $\omega \in A_{1+\frac{1}{n}}$ satisfies
\[
[[\omega]]_{\textup{BMO}(B_{2r}, \omega)}\leq \delta_0.
\]
Suppose also that $u\in \mathcal{W}^{2,1}_{n+1}(Q_{r,\, \omega},\, \omega)\cap C(\overline{Q_{r,\, \omega}})$ is non-negative and satisfies $\mathcal{L}u\geq 0$ in $Q_{r,\, \omega}$, and that
\begin{equation}\label{q density}
\omega(\Gamma)\geq q\omega(C_{r,\, \omega})\quad \text{with}\quad \Gamma=\{\mathcal{L}u\geq \omega\}\cap C_{r,\, \omega},
\end{equation}
for some $q \in (0,1)$. Corresponding to $\Gamma$, let $\hat E_0 = \hat{E}_0(q_0, \eta_0, l_0)$ be the set defined in \eqref{def E_0}. Then, we have the following alternatives.
\begin{itemize}
\item[\textup{(i)}] If $\omega(\hat E_0\setminus C_{r,\, \omega})> \xi q\omega(C_{r,\, \omega})$, then
\begin{equation} \label{big mea}
\inf_{B_{r/2}}u\big(\ \cdot,\ r^2(\omega^{-n})_{B_r}^{1/n}\big)\geq N_1\xi^{\gamma+2} q^{\gamma+2}r^2,
\end{equation}
where $\gamma=\gamma(n,\nu, K_0)>0$ is the constant defined in Lemma \ref{prop-up}.

\item[\textup{(ii)}]  If $\omega(\hat E_0\setminus C_{r,\, \omega})\leq \xi q\omega(C_{r,\, \omega})$, then $(1+\xi) q \in (0, 1]$ and
\begin{equation} \label{small mea}
\inf_{B_{r/2}}u\big(\ \cdot,\ r^2(\omega^{-n})_{B_r}^{1/n}\big)\geq I\left((1+\xi) q\right)/N_2.
\end{equation}
\end{itemize}
\end{lemma}    

\begin{proof}
For the given $\nu \in (0,1)$  and $K_0\in [1,\infty)$,  let  $q_0,\, \eta_0,\, \xi_0$, $l_0$ be the numbers defined in \eqref{def q_0}, \eqref{def eta_0}, \eqref{def xi_0}, and \eqref{definition l_0}, respectively. Then, let $\delta_0$ and $\xi$ be defined by
\begin{equation}\label{def xi}
\delta_0 =\min\left\{\bar N_1\delta_1,\, \delta_2,\, \bar \delta\right\}\in (0,1), \quad\text{and} \quad \xi=\frac{\xi_0-1}{4}\in (0,1),
\end{equation}
where $\bar N_1=\bar N_1(n, 1+\frac{1}{n}, n+1, K_0)$, $\delta_1= \delta_1(n, \nu, K_0, \frac{4l_0}{(1-\eta_0)^2})$, $\delta_2= \delta_2(n, \nu, K_0)$, and $\bar \delta=\bar \delta(n, \nu, K_0)$ are the constants defined in Lemma \ref{bmo-lemma}, Lemma \ref{baby prop-up}, Lemma \ref{prop-up}, and Lemma \ref{large density}, respectively. It is worth mentioning that $\delta_0$ and $\xi$ depend only on $n,\, \nu$, and $ K_0$. We now prove the lemma with the choices of $\delta_0$ and $\xi$ as in \eqref{def xi}.

\smallskip
\noindent
\smallskip
\textbf{Proof of (i)}. Let us denote
\[
r_0=\frac{\xi q}{l_0}r\in (0,r).
\]
We claim that there exist 
\begin{equation}\label{claim cylinder}
\rho_1\geq r_0 \quad  \text{and}\quad \tilde X=(\tilde x, \tilde t )\in C_{r,\, \omega} \quad \text{such that}\quad C_{\rho_1,\, \omega}(\tilde X)\in \mathcal{A}.
\end{equation}
Suppose that this is not true; then for all $C_{\rho,\, \omega}(X)\in \mathcal{A}$ with $X=(x,t)$, we have
$\rho< r_0$. By the definitions  of $\mathcal{A}$ in \eqref{cylinders set} and $\hat C$ in \eqref{hat cylinder}, we find $C_{\rho,\, \omega}(X)\subset C_{r,\, \omega}$ and
\[
\hat C_{\rho,\, \omega}(X)=B_{\eta_0 \rho}(x)\times\big(t+\rho^2(\omega^{-n})_{B_{\rho}(x)}^{1/n},\, t+l_0\rho^2(\omega^{-n})_{B_{\rho}(x)}^{1/n}\big).
\]
Since $t\leq 0$ and $B_{\rho}(x)\subset B_r$, by a simple calculation, we have
\[
t+l_0\rho^2(\omega^{-n})_{B_{\rho}(x)}^{1/n}\leq l_0\rho r(\omega^{-n})_{B_r}^{1/n}\leq\xi qr^2(\omega^{-n})_{B_r}^{1/n}.
\]
This, together with the fact that $t>-r^2(\omega^{-n})_{B_r}^{1/n}$, implies that
\[
\hat E=\bigcup_{C\in \A}\hat C\subset B_{r}\times \left(-r^2(\omega^{-n})_{B_r}^{1/n},\ \xi qr^2(\omega^{-n})_{B_r}^{1/n}\right).
\]
Hence, 
\[
\omega\big(\hat E\setminus C_{r,\, \omega}\big)\leq \xi qr^2(\omega^{-n})_{B_r}^{1/n}\omega(B_r)=\xi q\omega(C_{r,\, \omega}),
\]
which contradicts the assumption of \eqref{big mea}. The claim \eqref{claim cylinder} is proved.

\smallskip
Next, for this $C_{\rho_1,\, \omega}(\tilde X)$, by taking a linear translation and applying Lemma \ref{large density} to $u$, we have
\[
\inf_{B_{{\rho_1}/2}(\tilde x)}u(\cdot, \tilde t)\geq \kappa \rho_1^2,\quad \text{where}\quad \kappa=\kappa(n,\nu, K_0)>0.
\] 
Recall that $\rho_1\geq r_0=\frac{\xi q}{l_0}r$, it follows from the previous estimate and Lemma \ref{prop-up} that
\begin{align*}
\inf_{B_{r/2}}u(\ \cdot,\ r^2(\omega^{-n})_{B_r}^{1/n})\geq \big(\frac{\rho_1}{4r}\big)^{\gamma}\kappa \rho_1^2
&\geq (4^{-\gamma}l_0^{-\gamma-2}\kappa) (\xi q)^{\gamma+2}r^2\\
&=N_1(n,\nu, K_0)(\xi q)^{\gamma+2}r^2.
\end{align*}

\smallskip
\noindent
\textbf{Proof of (ii)}.  We denote 
\[
\Gamma_0=\hat E_0\cap C_{r,\, \omega},
\]
and it then follows from \eqref{mea hat E}, \eqref{q density}, the assumption that $\omega(\hat E_0\setminus C_{r,\, \omega})\leq \xi q\omega(C_{r,\, \omega})$ and the definition of $\xi$ in \eqref{def xi} that
\begin{equation}\label{mea Gamma_0}
\begin{aligned}
\omega(\Gamma_0)=\omega(\hat E_0)-\omega(\hat E_0\setminus C_{r,\, \omega})
&\geq \frac{1+\xi_0}{2}q\omega(C_{r,\, \omega})-\frac{\xi_0-1}{4}q\omega(C_{r,\, \omega})\\
&=(1+\xi)q\omega(C_{r,\, \omega}).
\end{aligned}
\end{equation}
This implies that $(1+\xi)q \in (0, 1]$, which particularly confirms that the right hand side of \eqref{small mea} is well-defined. Note that as $\mathcal{L} \in \mathbb{L}^{\infty}_{\nu, k}(K_0)$, there are solutions $\bar u,\ \bar v \in \mathcal{W}^{2,1}_{n+1}(Q_{r,\omega}, \omega) \cap C(\overline{Q}_{r,\, \omega})$ solving the equations 
\begin{equation}\label{two eqns}
\left\{
\begin{aligned}
\mathcal{L}\bar u&=\omega \chi_{\Gamma} &&\text{in}\quad Q_{r,\, \omega},\\
\bar u&=0 &&\text{on}\quad \partial'Q_{r,\, \omega},
\end{aligned}\right.
\quad\text{and}\quad
\left\{
\begin{aligned}
\mathcal{L}\bar v&=\omega \chi_{\Gamma_0} &&\text{in}\quad Q_{r,\, \omega},\\
\bar v&=0 &&\text{on}\quad \partial'Q_{r,\, \omega}.
\end{aligned}\right.
\end{equation}
By the comparison principle, the definition of function $I$ and \eqref{mea Gamma_0}, we have
\begin{equation}\label{comp}
u\geq \bar u\quad \text{in } Q_{r,\, \omega},\quad   \text{and}\quad \inf_{B_{r/2}}\bar v\big(\,\cdot, r^2(\omega^{-n})_{B_r}^{1/n}\big)\geq I\big((1+\xi) q\big).
\end{equation}

Now, we apply Lemma \ref{local-global} to compare $\bar u$ with $\bar v$. To this end, for each $X_0=(x_0,t_0)\in \Gamma_0$, by the construction of $\Gamma_0$, there exists $C_{\rho_0,\, \omega}(\bar{X})\in \A$ with $\rho_0>0$ and $\bar{X}=(\bar {x}, \bar t)$ such that $X_0\in \hat C_{\rho_0,\, \omega}(\bar{X})$.  Next, as $\mathcal{L} \in \mathbb{L}_{\nu, k}^\infty(K_0)$, there are solutions $u_1,\, v_1\in \mathcal{W}^{2,1}_{n+1}(U,\, \omega)\cap C(\overline{U})$  to the equations
\[
\left\{
\begin{aligned}
\mathcal{L} u_1&=N_2\omega\chi_{\Gamma} &&\text{in } U,\\
u_1&=0 &&\text{on } \partial' U,
\end{aligned}\right.
\quad \text{and}\quad
\left\{
\begin{aligned}
\mathcal{L}v_1&=\omega\chi_{\Gamma_0} &&\text{in } U,\\
v_1&=0 &&\text{on }\partial' U,
\end{aligned}\right.
\]
with $U=B_{\rho_0}(\bar x)\times (\bar t-\rho_0^2(\omega^{-n})_{B_{\rho_0}(\bar x)}^{1/n},\, \bar t+l_0\rho_0^2(\omega^{-n})_{B_{\rho_0}(\bar x)}^{1/n})$, where $N_2>0$ is some constant to be determined.

\smallskip
For the equation of $u_1$, since $C_{\rho_0,\, \omega}(\bar{X})\subset U$ and $C_{\rho_0,\, \omega}(\bar{X})\in \A$, we see that 
\[
\omega(\Gamma\cap C_{\rho_0,\, \omega}(\bar{X}))\geq q_0\omega(C_{\rho_0,\,\omega}(\bar{X})).
\]
Then, by rescaling the equation and applying Lemma \ref{large density}, we obtain
\begin{equation}\label{inf u_1}
\inf_{B_{\rho_0/2}(\bar{x})}u_1(\ \cdot \ , \bar t)\geq N_2\kappa \rho_0^2, \quad 
\text{where}\quad  \kappa=\kappa(n, \nu, K_0)>0.
\end{equation}
Next, as $X_0\in \hat C_{\rho_0,\omega}(\bar{X})$, it follows from the construction of $\hat C_{\rho_0,\omega}(\bar{X})$ in \eqref{hat cylinder} that 
\[
|x_0-\bar{x}|\leq \eta_0 \rho_0 \quad \text{and}\quad \rho_0^2(\omega^{-n})_{B_{\rho_0}(\bar{x})}^{1/n}\leq |t_0-\bar t|\leq l_0\rho_0^2(\omega^{-n})_{B_{\rho_0}(\bar{x})}^{1/n}.
\]
Thus, $V_{\frac{(1-\eta_0)\rho_0}{2}}(\bar X, X_0)\subset B_{\rho_0}(\bar x)\times \mathbb{R}$, and it satisfies condition \eqref{slant-cyl} under $K=\frac{4l_0}{(1-\eta_0)^2}$. Due to the choice of $\delta_0$ in \eqref{def xi}, we can apply Lemma \ref{baby prop-up} to $u_1$ on $V_{\frac{(1-\eta_0)\rho_0}{2}}(\bar X, X_0)$, and use the fact that $\eta_0\in (0,1)$ and \eqref{inf u_1}, to obtain
\begin{equation}\label{hat u}
u_1(x_0,t_0)\geq \beta \inf_{B_{(1-\eta_0)\rho_0/2}(\bar x)}u_1(\ \cdot,\, \bar t)\geq \beta \inf_{B_{\rho_0/2}(\bar{x})}u_1(\ \cdot \ , \bar t) \geq N_2\beta\kappa \rho_0^2,
\end{equation}
where $\beta=\beta(n,\nu, K_0)\in (0,1)$.

\smallskip
On the other hand, we control $v_1$ from above. Let
\[
\phi(x,t)=\frac{1}{2n\nu}\left(\rho_0^2-|x-\bar{x}|^2\right).
\]
By a direct computation using \eqref{elliptic}, we see that
\[
\left\{
\begin{aligned}
\mathcal{L}\phi&\geq \omega &&\text{in } C_{\rho_0,\, \omega}(\bar{X}),\\
\phi&\geq 0 &&\text{on }\partial'C_{\rho_0,\, \omega}(\bar{X}).
\end{aligned}\right.
\]
Then, by the comparison principle and the fact that $|x_0-\bar{x}|\leq \eta_0\rho_0$, it follows that
\begin{equation}\label{hat v}
v_1(x_0, t_0)\leq \phi(x_0, t_0)\leq \frac{1}{2n\nu}(1-\eta_0^2)\rho_0^2.
\end{equation}

Combining \eqref{hat u} and \eqref{hat v}, we infer that
\[
u_1(x_0, t_0)\geq \frac{2n\nu N_2\beta\kappa}{1-\eta_0^2}v_1(x_0,t_0).
\]
Then, we can choose the constant
\begin{equation}\label{N_2 choice}
N_2=N_2(n,\nu, K_0)\geq\frac{1-\eta_0^2}{2n\nu\beta\kappa}\quad  \text{so that}\quad u_1(x_0, t_0)\geq v_1(x_0,t_0).
\end{equation}
Hence, by scaling $\bar u\mapsto N_2\bar u$ for the first equation in \eqref{two eqns}, formula \eqref{N_2 choice} implies that the conditions in Lemma \ref{local-global} are satisfied. By Lemma \ref{local-global}, it follows that
\begin{equation}\label{comp-2}
N_2\bar u\geq  \bar v \quad \text{in}\quad C_{r,\, \omega}.
\end{equation}
Therefore, by this and \eqref{comp}, we obtain
\begin{equation*}
\begin{aligned}
\inf_{B_{r/2}}u(\ \cdot,\,  r^2(\omega^{-n})_{B_r}^{1/n})
\geq \inf_{B_{r/2}}\bar u(\ \cdot,\, r^2(\omega^{-n})_{B_r}^{1/n})
&\geq \frac{1}{N_2}\inf_{B_{r/2}}\bar v(\ \cdot,\,  r^2(\omega^{-n})_{B_r}^{1/n})\\
&\geq \frac{1}{N_2}I\big((1+\xi)q\big).
    \end{aligned}
\end{equation*}
The proof of the lemma is completed.
\end{proof}

From Lemma \ref{alternative lemma}, we derive the following important corollary on the lower estimate of the function $I(q)$ defined in Definition \ref{auxiliary function}. The corollary, in fact, provides the proof of Theorem \ref{M-thm} when the coefficients $(a_{ij})$ are smooth.

\begin{corollary}\label{I(q) bdd}  For every $\nu \in (0,1)$ and $K_0 \geq 1$, there exist $N=N(n,\nu,K_0)>0$ and $\gamma_0=\gamma_0(n,\nu, K_0)>2$ such that the following assertion holds. Let $\mathcal{L}\in \mathbb{L}_{\nu, k}^\infty(K_0)$ for some $k \in \mathbb{N}$, and assume that the corresponding $\omega$ of $\mathcal{L}$ satisfies $[[\omega]]_{\textup{BMO}(B_{2r},\, \omega)}\leq \delta_0$, where $\delta_0=\delta_0(n, \nu, K_0)$ is the constant defined in Lemma \ref{alternative lemma}, and $r>0$. Then 
\begin{equation}\label{I-q conclu}
 I(q)\geq Nq^{\gamma_0}r^2,\qquad \forall\, q\in [0,1],
\end{equation}
where $I(q)$ is the function defined in Definition \ref{auxiliary function}.
\end{corollary}

\begin{proof}
Firstly, by the definition of $I(q)$ in Definition \ref{auxiliary function}, and replacing $q$ by $(1+\xi)^{-1}q$ in \eqref{big mea} and \eqref{small mea} of Lemma \ref{alternative lemma}, we have 
\begin{equation}\label{min I}
 I\Big(\frac{q}{1+\xi}\Big)\geq \min\left\{N_1\left(\frac{\xi}{1+\xi}\right)^{\gamma+2}q^{\gamma+2}r^2,\, \frac{1}{N_2}I(q)\right\}, \quad q\in [0,1],
\end{equation}
where $\xi\in (0,1),\, \gamma>0,\, N_1>0$, and $N_2>1$ are the constants defined in Lemma \ref{alternative lemma}, and they only depend on $n,\, \nu$, and $K_0$. Moreover, it follows from the definition of $I(q)$, the choice of $\delta_0$ in \eqref{def xi}, and Corollary \ref{M bdd} that 
\begin{equation*}
I(q)\geq \kappa_1r^2, \quad \forall \, q\in [q_0, 1],
\end{equation*}
where $\kappa_1=\kappa_1(n,\nu, K_0)>0$, and $q_0=q_0(n,\nu, K_0)\in (0,1)$. Particularly,
\begin{equation}\label{large I}
I(1)\geq \kappa_1r^2.
\end{equation}
Therefore, by \eqref{min I} with $q=1$ and \eqref{large I}, it follows that
\[
I\Big(\frac{1}{1+\xi}\Big)\geq \min\left\{\alpha ,\ \frac{\kappa_1}{N_2}\right\}r^2, \quad \text{where}\quad \alpha=N_1\left(\frac{\xi}{1+\xi}\right)^{\gamma+2}.
\]
Due to this and using \eqref{min I} with $q=(1+\xi)^{-1}$, we get
\begin{align*}
I\left(\frac{1}{(1+\xi)^2}\right)
&\geq \min\left\{\frac{\alpha}{(1+\xi)^{\gamma+2}}, \ \frac{\alpha}{N_2},\ \frac{\kappa_1}{N_2^2}\right\}r^2\\
&\geq \min\left\{\alpha,\ \frac{\kappa_1}{N_2}\right\}\min\left\{\frac{1}{(1+\xi)^{\gamma+2}},\ \frac{1}{N_2}\right\}r^2,
\end{align*}
and iteratively,
\begin{align*}
I\left(\frac{1}{(1+\xi)^i}\right)
&\geq \min\left\{\frac{\alpha}{(1+\xi)^{(i-1)(\gamma+2)}}r^2,\ \frac{1}{N_2}I\left(\frac{1}{(1+\xi)^{i-1}}\right)\right\}\\
&\geq \min\left\{\alpha,\ \frac{\kappa_1}{N_2}\right\}\left(\min\left\{\frac{1}{(1+\xi)^{\gamma+2}},\ \frac{1}{N_2}\right\}\right)^{i-1}r^2, \quad \forall\, i=3, 4, \dots.
\end{align*}
Thus, we conclude that
\begin{equation}\label{discrete esti}
I\left((1+\xi)^{-i}\right)\geq a_0 b_0^{i-1}r^2, \quad \forall\, i\in \mathbb{N},
\end{equation}
with
\[
a_0=\min\left\{\alpha,\, \frac{\kappa_1}{N_2}\right\},\quad \text{and}\quad b_0=\min\left\{\frac{1}{(1+\xi)^{\gamma+2}},\, \frac{1}{N_2}\right\}.
\]

Lastly, since $I(q)$ is non-decreasing in $q\in [0,1]$ and the sequence $\{(1+\xi)^{-i}\}_{i\in \mathbb{N}}$ is strictly decreasing, it follows from \eqref{discrete esti} that 
\[
I(q)\geq a_0q_0^{\log_{1+\xi}b_0}q^{-\log_{1+\xi}b_0}r^2,\quad \forall\, q\in [0,1].
\]
By setting $N=N(n,\nu, K_0)=a_0q_0^{\log_{1+\xi}b_0}$ and $\gamma_0=\gamma_0(n,\nu, K_0)=-\log_{1+\xi}b_0$, and noting that $\gamma_0\geq \gamma+2>2$ from the definition of $b_0$, we achieve \eqref{I-q conclu}. Therefore, the proof of the corollary is completed.
\end{proof}

\subsection{Proof of Theorem \ref{M-thm}.} We are now ready to show the proof of Theorem \ref{M-thm}. 
\begin{proof} Observe that if $\mathcal{L} \in \mathbb{L}_{\nu, k}^\infty(K_0)$, then Theorem \ref{M-thm} follows directly from Corollary \ref{I(q) bdd} and the definition of $I(q)$ in Definition \ref{auxiliary function}. It remains to remove the regularity assumption on the coefficient matrix $(a_{ij})$. 

\smallskip
Assume that $\omega$ is smooth and satisfies \eqref{omega-cond-12-15} for some $k \in \mathbb{N}$. For the given measurable coefficient matrix $(a_{ij})$ satisfying \eqref{elliptic}, by taking the convolution of $a_{ij}$ with the standard mollifiers, we find a sequence $\{a_{ij}^{m}\}_m \in C^{\infty}$ such that 
\begin{equation} \label{a-k-pointwise}
a_{ij}^{m}(x,t) \rightarrow a_{ij}(x,t) \quad \text{for a.e.} \quad (x,t) \in Q_{r,\omega} \quad \text{as} \quad m \rightarrow \infty,
\end{equation}
for every $i, j \in \{1, 2,\ldots, n\}$. Moreover, the matrix $(a_{ij}^m)$ satisfies \eqref{elliptic} for all $m \in \mathbb{N}$. Let us denote
\[
\mathcal{L}^m \phi (x) =  \phi_t -  \omega(x) a_{ij}^m(x,t) D_{ij} \phi.
\]
We note that $\mathcal{L}^m \in \mathbb{L}_{\nu, k}^\infty(K_0)$ for $m \in \mathbb{N}$. Because $u \in \mathcal{W}^{2,1}_{n+1}(Q_{r,\omega}, \omega) \cap C(\overline{Q}_{r,\omega})$, by the Lebesgue dominated convergence theorem, we see that
\begin{equation} \label{convergence-L-regularize}
\mathcal{L}^m u \rightarrow \mathcal{L} u \quad \text{in} \quad L^{n+1}(Q_{r,\omega}, \omega^{-n}) \quad \text{as} \quad m \rightarrow \infty.
\end{equation}

Next, we note that as $\frac{1}{k} \leq \omega  \leq k$, $ L^{n+1}(Q_{r,\omega}) =  L^{n+1}(Q_{r,\omega}, \omega^{-n})$ and therefore $\mathcal{L} u,\, \mathcal{L}^m u \in L^{n+1}(Q_{r,\omega})$. From this and as $\mathcal{L}^m \in \mathbb{L}_{\nu, k}^\infty(K_0)$, it follows from the classical theory that there exists $v_m \in \mathcal{W}^{2,1}_{n+1}(Q_{r,\omega}, \omega) \cap C(\overline{Q}_{r,\omega})$ solving the equation
\begin{equation} \label{v-k-eqn-12-16}
\left\{
\begin{aligned}
\mathcal{L}^m v_m & =  \mathcal{L} u  \quad &&\text{in} \quad Q_{r,\omega},\\
v_m & =  u \quad &&\text{on} \quad \partial' Q_{r,\omega}.
\end{aligned}
\right.
\end{equation}
Due to the PDE in \eqref{v-k-eqn-12-16}, we note that
\begin{align} \notag
q  & = \frac{\omega \big(\big\{(x,t) \in C_{r,\omega}: \mathcal{L} u(x,t) \geq \omega(x) \big\}\big)}{\omega(C_{r,\omega})}\\ \label{q-v-m-v-compare}
& = \frac{\omega \big(\big\{(x,t) \in C_{r,\omega}: \mathcal{L}^m v_m(x,t) \geq \omega(x) \big\}\big)}{\omega(C_{r,\omega})}, \quad \forall \ m \in \mathbb{N}.
\end{align}
Also, by using  Theorem \ref{ABP} (or the comparison principle), and the assumptions that $u \geq 0$ and $\mathcal{L} u \geq 0$, we infer that
\begin{align}  \label{v-k-boundedness-12-16}
v_m \geq 0  \quad \text{in} \quad Q_{r,\omega}, \quad \forall\, m\in \mathbb{N}.
\end{align}
In addition, we also have $\mathcal{L}^m v_m = \mathcal{L} u\geq 0$ in $Q_{r,\omega}$.  From this, \eqref{q-v-m-v-compare}, and \eqref{v-k-boundedness-12-16}, we can apply the result we just proved for the operator $\mathcal{L}^m$  to obtain 
\[
 v_{m}(\cdot, r^2(\omega^{-n})_{B_r}^{1/n}) \geq N q^{\gamma_0} r^2 \quad \text{on} \quad B_{r/2}, \quad \forall \ m \in \mathbb{N}.
\]
On the other hand, let $w_m =  u - v_m$. We note that $w_m \in \mathcal{W}^{2,1}_{n+1}(Q_{r,\omega}, \omega) \cap C(\overline{Q}_{r,\omega})$ solving
\[
\left\{
\begin{aligned}
\mathcal{L}^m w_m & = \mathcal{L}^m u - \mathcal{L} u  \quad &&\text{in} \quad Q_{r,\omega}, \\
w_m & =  0  \quad &&\text{on} \quad \partial' Q_{r,\omega}.
\end{aligned}
\right.
\]
It then follows from Theorem \ref{ABP} and \eqref{convergence-L-regularize} that
\begin{equation} \label{w-k-con-12-16}
\|w_m\|_{L^\infty(Q_{r,\omega})} \leq N_0r^{\frac{n}{n+1}} \|\mathcal{L}^m u - \mathcal{L} u \|_{L^{n+1}(Q_{r,\omega}, \omega^{-n})} \rightarrow 0 \quad \text{as} \quad m \rightarrow \infty.
\end{equation}
Next, observe that
\begin{align*}
u(x, r^2(\omega^{-n})_{B_r}^{1/n}) & = v_{m}(x, r^2(\omega^{-n})_{B_r}^{1/n}) + w_{m}(x, r^2(\omega^{-n})_{B_r}^{1/n}) \\
& \geq N q^{\gamma_0} r^2 - \|w_m\|_{L^\infty(Q_{r,\omega})}, \quad \forall\ x \in B_{r/2}, \quad \forall \ m \in \mathbb{N}.
\end{align*}
Letting $m\rightarrow \infty$ and using  \eqref{w-k-con-12-16}, we conclude that
\[
u(x, r^2(\omega^{-n})_{B_r}^{1/n})  \geq N q^{\gamma_0} r^2, \quad \text{for} \quad x \in B_{r/2}.
\]
The proof of the theorem is completed.
\end{proof}
\section{Parabolic weighted $W^{2,\varepsilon}$-Lin type estimates and Proof of  Theorem~\ref{Lin thm}} \label{section of lin's proof}
To prove Theorem \ref{Lin thm}, we need several lemmas. Let us recall the set $\mathbb{L}_{\nu, k}^\infty(K_0)$ of operators with truncated and smooth coefficients defined in \eqref{L-smooth}. We begin with the following lemma, known as weak Harnack inequalities.
\begin{lemma}\label{posti lemma}
Let $\nu \in (0,1)$ and $K_0 \in [1, \infty)$, and let $\delta_0 =\delta_0(n, \nu, K_0) \in (0,1)$, $\gamma_0=\gamma_0(n,\nu, K_0)>1$ be the constants defined in Theorem \ref{M-thm}. Suppose that $\mathcal{L} \in \mathbb{L}_{\nu, k}^\infty(K_0)$ for some $k \in \mathbb{N}$, with its $\omega \in A_{1+\frac{1}{n}}$ satisfying $[[\omega]]_{\textup{BMO}(B_{2r},\, \omega)}\leq \delta_0$. Suppose also that $u\in \mathcal{W}^{2,1}_{n+1}(Q_{r,\, \omega},\, \omega) \cap C(\overline{Q}_{r,\omega})$ satisfies $\mathcal{L}u\geq 0$ in $Q_{r,\, \omega}$ and $u\geq 0$ on $\partial'Q_{r,\, \omega}$. Then, for every $p\in (0,\frac{1}{2\gamma_0}]$, we have
\begin{equation}\label{Lp est}
\left(\frac{1}{\omega(C_{r,\, \omega})}\int _{C_{r,\, \omega}}|\mathcal{L}u(x,t)|^p\omega(x)^{1-p}\, dxdt \right)^{\frac{1}{p}}
\leq Nr^{-2}u(0, \bar t),
\end{equation}
where $\bar t= r^2(\omega^{-n})_{B_r}^{1/n}$,  and $N=N(n,\nu, K_0)>0$.
\end{lemma} 

\begin{proof} By the maximal principle, we see that $u\geq 0$ in $Q_{r,\, \omega}$. Then, it follows from Theorem \ref{M-thm} that there exists $N=N(n,\nu, K_0)>0$ such that
\[
u(0, \bar t)\geq Nq^{\gamma_0}r^2,
\]
where
\[
q=\frac{\omega(C_{r,\, \omega}\cap \{\mathcal{L}u\geq \omega\})}{\omega(C_{r,\, \omega})}\in [0, 1]. 
\]
Hence,
\[
\omega\left(C_{r,\, \omega}\cap \big\{\mathcal{L}u\geq \omega\big\}\right)\leq (Nr^2)^{-\frac{1}{\gamma_0}}\omega(C_{r,\, \omega})[u(0, \bar t)]^{\frac{1}{\gamma_0}}.
\]
Applying this estimate to $u/\lambda$ for $\lambda>0$, we obtain
\begin{equation}\label{measure est}
\omega(C_{r,\, \omega}\cap \big\{\mathcal{L}u\geq \lambda\omega\big\})
\leq \big(Nr^2\big)^{-\frac{1}{\gamma_0}}\omega(C_{r,\, \omega})[u(0, \bar t)]^{\frac{1}{\gamma_0}}\lambda^{-\frac{1}{\gamma_0}},
\end{equation}
where $N=N(n,\nu, K_0)>0$. Besides, we notice that
\begin{align*}
&\int_{C_{r,\, \omega}}|(\mathcal{L}u)/\omega|^p\, \omega(x) dx dt
=p\int_0^{\infty}\omega(C_{r,\, \omega}\cap \{\mathcal{L}u\geq \lambda\omega\})\lambda^{p-1}\, d\lambda\\
&= p \int_0^s \omega(C_{r,\, \omega}\cap \{\mathcal{L}u\geq \lambda\omega\})\lambda^{p-1}\, d\lambda + p\int_s^{\infty} \omega(C_{r,\, \omega}\cap \{\mathcal{L}u\geq \lambda\})\lambda^{p-1}\, d\lambda \\
&=I_1+I_2,
\end{align*}
where $s>0$ is some number to be determined.

\smallskip
Now, we estimate the two terms $I_1$ and $I_2$. For $I_1$, since 
\[
\omega(C_{r,\, \omega}\cap \{\mathcal{L}u\geq \lambda\omega\})\leq \omega(C_{r,\, \omega}),
\]
we have
\begin{equation}\label{I-1}
I_1\leq p\int_0^s\omega(C_{r,\, \omega})\lambda^{p-1}\, d\lambda
=\omega(C_{r,\, \omega})s^p.
\end{equation}
For $I_2$, let $p\in (0,\frac{1}{2\gamma_0}]$, by \eqref{measure est}, we get
\begin{equation}\label{I-2}
\begin{aligned}
I_2
&\leq p\big(Nr^2\big)^{-\frac{1}{\gamma_0}}\omega(C_{r,\, \omega})[u(0, \bar t)]^{\frac{1}{\gamma_0}} \int_s^{\infty}\lambda^{-\frac{1}{\gamma_0}+p-1}\, d\lambda\\
&\leq \big(Nr^2\big)^{-\frac{1}{\gamma_0}}\omega(C_{r,\, \omega})[u(0, \bar t)]^{\frac{1}{\gamma_0}} s^{-\frac{1}{\gamma_0}+p}.
\end{aligned}
\end{equation}
Next, we set 
\begin{equation}\label{s-choice}
s=\big(Nr^2\big)^{-1}u(0, \bar t).
\end{equation}
Then, for all $p\in (0,\frac{1}{2\gamma_0}]$, it follows from \eqref{I-1} and \eqref{I-2} that
\[
\int_{C_{r,\, \omega}}|(\mathcal{L}u)/\omega|^p\omega(x)\, dxdt\leq \big(Nr^2\big)^{-p} \omega(C_{r,\, \omega})[u(0, \bar t)]^p,
\]
where $N=N(n,\nu, K_0)>0$. Therefore, the lemma is proved.
\end{proof}

The following lemma is an improvement of Lemma \ref{posti lemma}.
\begin{lemma}\label{Lemma 2}
Let $\nu,\, K_0, \, \delta_0,\, \gamma_0$ be as in Lemma \ref{posti lemma}. Let $f,\, g\in L^{n+1}(Q_{r,\, \omega},\, \omega^{-n})$ with $g\geq 0$, and suppose that $\mathcal{L} \in \mathbb{L}_{\nu, k}^\infty(K_0)$ for some $k \in \mathbb{N}$, with its $\omega \in A_{1+\frac{1}{n}}$ satisfying $[[\omega]]_{\textup{BMO}(B_{2r},\, \omega)}\leq \delta_0$. Suppose also that $u\in \mathcal{W}^{2,1}_{n+1}(Q_{r,\, \omega},\, \omega) \cap C(\overline{Q}_{r,\omega})$ satisfies 
\[ \mathcal{L}u=g+f \quad \text{in} \quad  Q_{r,\, \omega}.
\]
Then, for every $p\in (0,\frac{1}{2\gamma_0}]$, we have
\begin{equation}\label{general f}
\begin{aligned}
& \left( \frac{1}{\omega(C_{r,\, \omega})}\int_{C_{r,\, \omega}}|g/\omega|^p\omega\, dxdt \right)^{1/p} \\
&\leq Nr^{-2}\Big\{u^{+}(0, \bar t)+\sup_{\partial'Q_{r,\, \omega}}u^-+r^{\frac{n}{n+1}}\lVert f^-\lVert_{L^{n+1}(Q_{r,\, \omega},\, \omega^{-n})}\Big\},
\end{aligned}
\end{equation}
where $N=N(n,\nu, K_0)>0$, and $\bar t= r^2(\omega^{-n})_{B_r}^{1/n}$.
\end{lemma}

\begin{proof} Since $\mathcal{L}\in \mathbb{L}_{\nu, k}^\infty(K_0)$, by the classical theory for parabolic equations in non-divergence form with smooth and uniformly elliptic and bounded coefficients, there are $v,\, w\in \mathcal{W}^{2,1}_{n+1}(Q_{r,\, \omega},\, \omega) \cap C(\overline{Q}_{r,\omega})$ solving
\begin{equation*}
\left\{
\begin{aligned}
\mathcal{L} v &=g  &&\text{in}\quad Q_{r,\, \omega},\\
v&=0  &&\text{on}\quad \partial'Q_{r,\, \omega},
\end{aligned}\right.
\quad \text{and}\quad
\left\{
\begin{aligned}
\mathcal{L}w &=-f  &&\text{in}\quad Q_{r,\, \omega},\\
w&=-u &&\text{on}\quad \partial'Q_{r,\, \omega}.
\end{aligned}\right.
\end{equation*}
Hence, we have $v=u+w$ in $Q_{r,\, \omega}$. For the equation of $v$, due to $g\geq 0$, it then follows from Lemma \ref{posti lemma} that
\begin{equation}\label{eqn 1}
\left(\frac{1}{\omega(C_{r,\, \omega})}\int_{C_{r,\, \omega}}|g/\omega|^p\omega\, dxdt\right)^{\frac{1}{p}}\leq N(n, \nu, K_0) r^{-2}v(0, \bar t),
\end{equation}
for all $p\in (0,\frac{1}{2\gamma_0}]$. On the other hand, for $w$, by Theorem \ref{ABP}, we have
\begin{align*}
w^+(0, \bar t)\leq \sup_{Q_{r,\, \omega}}w^+
&\leq \sup_{\partial'Q_{r,\, \omega}}(-u)^+ +N_0(n,\nu)r^{\frac{n}{n+1}}\lVert (-f)^+\lVert_{L^{n+1}(Q_{r,\, \omega},\, \omega^{-n})}\\
&=\sup_{\partial'Q_{r,\, \omega}}u^-+N_0(n,\nu)r^{\frac{n}{n+1}}\lVert f^-\lVert_{L^{n+1}(Q_{r,\, \omega},\, \omega^{-n})}.
\end{align*}
As $\frac{1}{2\gamma_0}<1$, then for all $p\in (0,\frac{1}{2\gamma_0}]$,
\begin{align*}
v(0, \bar t)
&\leq u^+(0, \bar t)+w^+(0, \bar t)\\
&\leq u^+(0, \bar t)+\sup_{\partial'Q_{r,\, \omega}}u^-+N_0(n,\nu)r^{\frac{n}{n+1}}\lVert f^-\lVert_{L^{n+1}(Q_{r,\, \omega},\, \omega^{-n})}.
\end{align*}
This, together with \eqref{eqn 1}, implies \eqref{general f}. Therefore, the lemma is proved.
\end{proof}

Based on Lemma \ref{Lemma 2}, we state and prove the following theorem, which is a weaker form of Theorem \ref{Lin thm}.
\begin{theorem}\label{Lin thm-2}
Let $\nu,\, K_0,\, \delta_0,\, \gamma_0$ be as in Lemma \ref{posti lemma}. There exists a sufficiently small constant $p_0=p_0(n,\nu, K_0)>0$ such that the following assertions hold. Suppose that $u\in \mathcal{W}^{2,1}_{n+1}(Q_{r,\, \omega},\, \omega)\cap C(\overline{Q}_{r,\, \omega})$, and that $\mathcal{L} \in \mathbb{L}_{\nu, k}^\infty(K_0)$ for some $k \in \mathbb{N}$, with its $\omega \in A_{1+\frac{1}{n}}$ satisfying $[[\omega]]_{\textup{BMO}(B_{2r},\, \omega)}\leq \delta_0$. Then, for every $p \in (0,p_0]$, we have
\begin{equation}\label{Lin esti-2}
\lVert D^2u\lVert_{L^{p}(C_{r,\, \omega},\, \omega)}\leq Nr^{-2+\frac{n+2}{p}}\Big(\sup_{\partial' Q_{r,\, \omega}}|u|+r^{\frac{n}{n+1}}\lVert \mathcal{L} u\lVert_{L^{n+1}(Q_{r,\, \omega},\, \omega^{-n})}\Big),
\end{equation}
where $N=N(n, \nu, K_0, p)>0$.
\end{theorem}

\begin{proof} For the given $u \in \mathcal{W}^{2,1}_{n+1}(Q_{r,\, \omega},\, \omega) \cap C(\overline{Q_{r,\, \omega}})$, we claim that we can choose a smooth matrix $(a^0_{ij}(x,t))$ satisfying \eqref{elliptic} on $Q_{r,\, \omega}$ with ellipticity constant $\frac{\nu}{2}$, and
\begin{equation}\label{difference esti}
\frac{\nu}{2}\omega(x)|D^2u|(x,t)\leq (\mathcal{L}_0-\mathcal{L})u(x,t),\quad \, \text{for a.e.} \ (x,t)\in Q_{r,\, \omega},
\end{equation}
where $\mathcal{L}_0$ is the operator defined by
\begin{equation} \label{L-0.def}
\mathcal{L}_0 v =v_t-\omega(x)a^0_{ij}(x,t)D_{ij}v(x,t).
\end{equation}
We follow the idea in \cite[section~2]{Lin} to prove the claim. For a.e. $(x,t)\in Q_{r,\, \omega}$, there exists an orthogonal matrix $B(x,t)\in \mathbb{R}^{n\times n}$ such that
\begin{equation}\label{eigenvalues}
B(x,t)^*D^2u(x,t)B(x,t)=\textup{diag}\{\lambda_1,\dots, \lambda_{k_0}, -\lambda_{k_0+1},\dots, -\lambda_n\}(x,t),
\end{equation}
where $\lambda_i(x,t)\geq 0$ for each $i\in \{1,2, \dots,n\}$, $k_0 =k_0(x,t)\in \{1,2, \dots,n\}$, and $B^*$ denotes the transpose of $B$. Based on the matrix $B$ and $k_0$, we define a matrix $A_0(x,t)=(a^0_{ij}(x,t))$ by
\[
(a^0_{ij}(x,t))=B(x,t)\textup{diag}\left\{\tfrac{\nu}{2},\dots,\tfrac{\nu}{2}, \tfrac{2}{\nu},\dots, \tfrac{2}{\nu}\right\}B^*(x,t),
\]
where $\text{diag}\{ \ldots \}$ denotes the $n \times n$ diagonal matrix, and $\frac{\nu}{2}$ appears $k_0$ times. It is straightforward to see that the matrix $(a^0_{ij}(x,t))$ satisfies \eqref{elliptic} with ellipticity constant $\frac{\nu}{2}\in (0,1)$. 
By a direct computation and using \eqref{eigenvalues}, we have
\begin{equation}\label{part-1}
\begin{aligned}
a^0_{ij}D_{ij}u=
&=\textup{Tr}(B\,\textup{diag}\left\{\tfrac{\nu}{2},\dots,\tfrac{\nu}{2}, \tfrac{2}{\nu},\dots, \tfrac{2}{\nu}\right\}B^*D^2u)\\
&=\textup{Tr}(\textup{diag}\left\{\tfrac{\nu}{2},\dots,\tfrac{\nu}{2}, \tfrac{2}{\nu},\dots, \tfrac{2}{\nu}\right\}B^* D^2u\, B)\\
&=\frac{\nu}{2}\displaystyle{\sum_{i=1}^{k_0}}\lambda_i
  -\frac{2}{\nu}\displaystyle{\sum_{i=k_0+1}^{n}}\lambda_i.
\end{aligned}
\end{equation}
Similarly, since $\mathcal{L}$ satisfies \eqref{elliptic}, there exist constants $\alpha,\, \beta\in (\nu,\, \nu^{-1})$ such that
\begin{equation}\label{part-2}
a_{ij}D_{ij}u=\alpha\sum_{i=1}^{k_0}\lambda_i
  -\beta\sum_{i=k_0+1}^n\lambda_i. 
\end{equation}
Moreover, observing that
\[
|D^2u|^2=\textup{Tr}(D^2uD^2u)=\textup{Tr}(\textup{diag}\{\lambda_1^2,\lambda_2^2,\dots, \lambda_n^2\})=\sum_{i=1}^{n}\lambda_i^2\leq \Bigl(\sum_{i=1}^n\lambda_i\Bigr)^2,
\]
we infer
\begin{equation}\label{part-3}
|D^2u|\leq \sum_{i=1}^n\lambda_i.
\end{equation}
Now, from the definition of $\mathcal{L}_0$ in \eqref{L-0.def}, and by combining \eqref{part-1}, \eqref{part-2}, and \eqref{part-3}, we conclude that
\begin{align*}
(\mathcal{L}_0-\mathcal{L})u(x,t)
&=\omega(x)\Bigl[(\alpha-\tfrac{\nu}{2})\sum_{i=1}^{k_0}\lambda_i-(\beta-\tfrac{2}{\nu})\sum_{i=k_0+1}^{n}\lambda_i\Bigr]\\
&\geq \omega(x)\frac{\nu}{2}\sum_{i=1}^{n}\lambda_i\geq\frac{\nu}{2}\omega(x)|D^2u|(x,t),\quad \forall\, (x,t)\in Q_{r,\, \omega}.
\end{align*}
Therefore, the claim \eqref{difference esti} is proved.

\smallskip
Next, by Theorem \ref{ABP}, we find that
\begin{equation}\label{u esti}
u^+(0, r^2(\omega^{-n})_{B_r}^{1/n})\leq \sup_{\partial'Q_{r,\, \omega}}u^++N_0(n,\nu)r^{\frac{n}{n+1}}\lVert f^+\lVert_{L^{n+1}(Q_{r,\, \omega},\, \omega^{-n})}.
\end{equation}
Also, we write
\[
g=(\mathcal{L}_0-\mathcal{L})u\geq 0,
\]
where $\mathcal{L}_0$ is defined  in \eqref{L-0.def}.  It follows that
\begin{equation}\label{l_0 eqn}
\mathcal{L}_0u=(\mathcal{L}_0-\mathcal{L})u+\mathcal{L}u=g+f \quad \text{in}\quad Q_{r,\, \omega}.
\end{equation}
Applying Lemma \ref{Lemma 2} to \eqref{l_0 eqn}, and by \eqref{u esti}, we obtain
\begin{equation}\label{diff integral}
\frac{1}{\omega(C_{r,\, \omega})}\int_{C_{r,\, \omega}} |g/\omega|^s\omega\, dxdt\leq N_1^sr^{-2s}\big\{\sup_{\partial'Q_{r,\, \omega}}|u|^s+r^{\frac{ns}{n+1}}\lVert f\lVert_{L^{n+1}(Q_{r,\, \omega},\, \omega^{-n})}^s\big\}
\end{equation}
for all $s\in (0,\frac{1}{2\gamma_0(n,\, \nu/2,\, K_0)}]$, where $N_1=N_1(n, \nu, K_0)>0$. On the other hand, by H\"{o}lder's inequality and \eqref{difference esti}, we infer that for all $q\in [1,\infty)$,
\begin{equation}\label{Hessian-1}
\begin{aligned}
\int_{C_{r,\, \omega}}|D^2u|^{\frac{s}{q}}\omega\, dxdt
&\leq \left(\int_{C_{r,\, \omega}}|D^2u|^s\omega\, dxdt\right)^{\frac{1}{q}}\left(\int_{C_{r,\, \omega}}\omega(x)\, dxdt\right)^{1-\frac{1}{q}}\\
&\leq \left(\frac{2}{\nu}\right)^{\frac{s}{q}}\left(\frac{1}{\omega(C_{r,\, \omega})}\int_{C_{r,\, \omega}}|g/\omega|^s\omega\, dxdt\right)^{\frac{1}{q}}\omega(C_{r,\, \omega}).
\end{aligned}
\end{equation}

Therefore, combining \eqref{diff integral}, \eqref{Hessian-1} with the fact that $\omega(C_{r,\, \omega})\leq N(n) K_0r^{n+2}$ from \eqref{cylinder measure}, we infer that
\begin{align*}
\int_{C_{r,\, \omega}}|D^2u|^{\frac{s}{q}}\omega\, dxdt\leq (2N_1\nu^{-1})^{\frac{s}{q}}N(n)K_0r^{n+2-\frac{2s}{q}}\big\{\sup_{\partial'Q_{r,\, \omega}}|u|^{\frac{s}{q}}+r^{\frac{ns}{(n+1)q}}\lVert f\lVert_{L^{n+1}(Q_{r,\, \omega},\, \omega^{-n})}^{\frac{s}{q}}\big\}.
\end{align*}
Denoting $p=\frac{s}{q}$, due to $s\in (0,\frac{1}{2\gamma_0}]$ and $q\in [1,\infty)$, we have
\[
p\in (0, p_0] \quad \text{with}\quad p_0=1/2\gamma_0.
\]
The theorem then follows from the last two formulas.
\end{proof}

\smallskip
Now, we present the proof of Theorem \ref{Lin thm}.
\begin{proof}[Proof of Theorem \ref{Lin thm}.] Let $\delta_0 = \delta_0(n, \nu, K_0) \in (0,1)$ and $p_0 = p_0(n, \nu, K_0)>0$ be defined as in Theorem \ref{Lin thm-2}, and let 
\[ 
\delta = \min\{\delta_0(n, \nu, K_0),\, \delta_0(n, \nu, K_0+2)/\bar{N}_0, \,  \delta_0(n, \nu, 2^{n+1}K_0)\},
\]
where $\bar{N}_0 = \bar{N_0}(n, K_0)>0$ is the constant defined in  Proposition \ref{stability-cut-omega}. We prove the theorem with these choices of $\delta$ and $p_0$. We split the proof into four steps.

\smallskip \noindent
{\bf Step 1}.  We assume $\mathcal{L} \in \mathbb{L}_{\nu, k}^\infty(K_0)$ for some $k \in \mathbb{N}$. In this case, we have  $\omega, a_{ij} \in C^\infty(\overline{Q}_{1,\omega})$ and $\frac{1}{k} \leq \omega \leq k$. Let us denote $f = \mathcal{L} u \in L^{n+1}(C_{1,\omega}, \omega^{-n})$. We note that $f \in L^{n+1}(C_{1,\omega})$ due to the boundedness of $\omega$. Define 
\[
\tilde f=\left\{
\begin{array}{ll}
f & \text{in } C_{1,\,\omega},\\
0 &\text{in } Q_{1,\,\omega}\setminus C_{1,\,\omega},
\end{array} \right.
\]
Also, let us define
\[
 \tilde u(x,t)=\left\{
\begin{aligned}
&u(x,t)  &&\text{in } \overline C_{1,\,\omega},\\
&u(x,-t) &&\text{in }  \overline{Q}_{1,\,\omega}\setminus \overline C_{1,\,\omega}.
\end{aligned}\right.
\]
We see that  $\tilde f\in L^{n+1}(Q_{1,\,\omega})$, and $\tilde u\in \mathcal{W}^{2,1}_{n+1}(Q_{1,\,\omega},\, \omega) \cap C(\overline{Q}_{1,\,\omega})$. As $\mathcal{L} \in \mathbb{L}_{\nu, k}^\infty(K_0)$,  
there exists a unique solution $v\in \mathcal{W}^{2,1}_{n+1}(Q_{1,\omega}, \omega) \cap C(\overline{Q}_{1,\omega})$ solving the equation
\[\left\{
\begin{aligned}
\mathcal{L}v&=\tilde f \quad &&\text{in} \quad Q_{1,\,\omega},\\[4pt]
v&=\tilde u \quad &&\text{on}\quad \partial' Q_{1,\,\omega}.
\end{aligned}\right.
\]
Note again that we actually have $v$ is in the standard parabolic Sobolev space $W^{2,1}_{n+1}(Q_{1,\omega})$ because of the boundedness of $w$. Due to this, and by the uniqueness of the solution, we have $v=u$ in $C_{1,\,\omega}$. Thus, it follows form Theorem \ref{Lin thm-2} that
\begin{align*}
\lVert D^2u\lVert_{L^{p}(C_{1,\,\omega},\, \omega)}=\lVert D^2v\lVert_{L^{p}(C_{1,\,\omega},\, \omega)}
&\leq N\Big(\sup_{\partial'Q_{1,\,\omega}}|\tilde u|+ \lVert \tilde f\lVert_{L^{n+1}(Q_{1,\, \omega},\,\omega^{-n})}\Big)\\
&=N\Big( \sup_{\partial' C_{1,\,\omega}} |u| +\lVert f\lVert_{L^{n+1}(C_{1,\,\omega},\,\omega^{-n})}\Big),
\end{align*}
where $N=N(n, \nu, K_0, p)>0$. Therefore, \eqref{Lin esti} is proved under the extra assumption that $\mathcal{L} \in \mathbb{L}_{\nu, k}^\infty(K_0)$.

\smallskip \noindent
{\bf Step 2}. We remove the smoothness assumptions on $(a_{ij})$ and $\omega$. We first show how to remove the smoothness assumption on $(a_{ij})$. For the given measurable matrix $(a_{ij})$ satisfying \eqref{elliptic}, by taking the convolution of $a_{ij}$ with suitable mollifiers, we can find a sequence $\{a_{ij}^m\}_m \in C^\infty$  such that
\[
a_{ij}^{m}(x,t) \rightarrow a_{ij}(x,t) \quad \text{for a.e.} \quad (x,t) \in Q_{1,\omega} \quad \text{as} \quad m \rightarrow \infty,
\]
for every $i, j \in \{1, 2,\ldots, n\}$. Moreover, the matrix $(a_{ij}^m)$ satisfies \eqref{elliptic} for all $m \in \mathbb{N}$. Let us denote
\[
\mathcal{L}^m \phi (x, t) = \phi_t - \omega(x) a_{ij}^m(x,t) D_{ij} \phi.
\]
As $u \in \mathcal{W}^{2,1}_{n+1}(C_{1, \omega}, \omega)$ and $(a_{ij}^m)$ satisfies \eqref{elliptic} for all $m \in \mathbb{N}$, by the Lebesgue dominated convergence theorem, we see that
\begin{equation} \label{converge-L-k-final}
\mathcal{L}^{m} u \rightarrow \mathcal{L} u \quad \text{in} \quad L^{n+1}(C_{1,\omega}, \omega^{-n}) \quad \text{as} \quad  m\rightarrow \infty.
\end{equation}
On the other hand, by applying the estimate we just proved for $\mathcal{L}^m$ with smooth coefficients, we obtain
\[
\lVert D^2u\lVert_{L^{p}(C_{1,\,\omega},\, \omega)} \leq N\Big(\sup_{\partial' C_{1,\,\omega}} |u|+\lVert \mathcal{L}^m u\lVert_{L^{n+1}(C_{1,\,\omega},\,\omega^{-n})}\Big), \quad \forall \ m \in \mathbb{N}.
\]
Here, $N = N(n, \nu, K_0, p)>0$. From this last estimate, letting $m\rightarrow \infty$ and using \eqref{converge-L-k-final}, we obtain
\begin{equation} \label{est-12-25-smooth}
\lVert D^2u\lVert_{L^{p}(C_{1,\,\omega},\, \omega)} \leq N\Big(\sup_{\partial' C_{1,\,\omega}} |u|+\lVert \mathcal{L} u\lVert_{L^{n+1}(C_{1,\,\omega},\,\omega^{-n})}\Big).
\end{equation}
This proves \eqref{Lin esti} when $(a_{ij})$ only satisfies \eqref{elliptic}. 

\smallskip
Next, we remove the smoothness assumption on $\omega$. For the given $\omega \in A_{1+\frac{1}{n}}$ with $\frac{1}{k} \leq \omega \leq k$ for some $k \in \mathbb{N}$, let $\omega_{m} = \omega * \phi_{\frac{1}{m}}$ as in Proposition \ref{weighted-BMO-regularization}. We note that $\frac{1}{k} \leq \omega_m \leq k$, and it follows from Proposition \ref{A-p-regularization} and Proposition \ref{weighted-BMO-regularization} that
\begin{align*}
& [\omega_m]_{A_{1+\frac{1}{n}}} \leq 2^{n+1} K_0, \quad \text{and} \\ 
& [[\omega_m]]_{\textup{BMO}(B_2, \omega_m)} \leq [[\omega]]_{\textup{BMO}(B_3, \omega)} \leq \delta \leq \delta_0(n, \nu, 2^{n+1}K_0).
\end{align*}
Then, we can follow the same approach as above when removing the smoothness assumption on $(a_{ij})$. This is possible because $u \in \mathcal{W}^{2,1}_{n+1}(C_{1, \omega}, \omega)$ and $ \frac{1}{k} \leq \omega, \omega_m \leq k$  which allow us to pass the limit of the regularized sequence $\{\omega_m\}_m$  using the Lebesgue dominated convergence theorem and a scaling argument. Hence, we also obtain \eqref{est-12-25-smooth} with the extra assumption that $\frac{1}{k} \leq \omega \leq k$ for $k \in \mathbb{N}$, but $\omega$ is not required to be smooth.

\smallskip \noindent
{\bf Step 3}. We assume that $u \in C^{2,1}(\overline{C}_{1, \omega})$ and we remove the condition $\frac{1}{k} \leq \omega \leq k$. For the given $\omega \in A_{1+\frac{1}{n}}$ satisfying $[\omega]_{A_{1+\frac{1}{n}}} \leq K_0$. With a suitable dilation in the time variable, we can assume without loss of generality that
\[
(\omega^{-n})_{B_1} =1.
\]
In this case $C_{1,\omega} = C_1$.   For $k, l \in \mathbb{N}$, let us define
\[
\beta_{k,l}(x) = \left\{
\begin{array}{ll}
\omega(x) & \quad \text{if} \quad l^{-1} \leq \omega(x) \leq k,\\ \smallskip
l^{-1} & \quad \text{if} \quad \omega(x) < l^{-1}, \\ \smallskip
k & \quad \text{if} \quad \omega(x) > k, \quad x \in \mathbb{R}^n.
\end{array} \right.
\]
Also, let $b_{k,l} = (\beta_{k,l}^{-n})_{B_1}^{1/n}$, $\omega_{k, l}(x) = b_{k,l} \beta_{k,l}(x)$ and
\[
\mathcal{L}_{k,l} \phi (x, t) = \phi_t - \omega_{k,l}(x) a_{ij}(x,t) D_{ij} \phi.
\]
By Proposition \ref{cut-omega} and Proposition \ref{stability-cut-omega}, we note that $\omega_{k,l} \in A_{1+\frac{1}{n}}$ and
\begin{align*}
& [\omega_{k, l}]_{A_{1+\frac{1}{n}}}=[\beta_{k, l}]_{A_{1+\frac{1}{n}}} \leq K_0 +2, \quad \text{and} \\
& [[\omega_{k,l}]]_{\textup{BMO}(B_2,\, \omega_{k,l})}=[[\beta_{k,l}]]_{\textup{BMO}(B_2,\, \beta_{k,l})} \leq \bar{N}_0 [[\omega]]_{\textup{BMO}(B_2,\, \omega)} \leq \bar{N}_0\delta \leq \delta_0(n, \nu, K_0+2).
\end{align*}
We also note that $(\omega_{k,l}^{-n})_{B_1} =1$ and that $\frac{1}{K} \le \omega_{k,l} \le K$ for any $K \in \mathbb{N}$ satisfying $K \geq \max\{ b_{k,l} k,  b_{k,l}^{-1} l\}$. As $u \in C^{2,1}(\overline{C}_{1})$, we see that $u \in \mathcal{W}^{2,1}_{n+1}(C_{1},\, \omega_{k,l}) \cap C(\overline{C}_{1})$. Then, by applying \eqref{est-12-25-smooth} to $u$ and  $\mathcal{L}_{k,l}$, we find $N = N(n, \nu, K_0, p)>0$ such that
\[
\lVert D^2u\lVert_{L^{p}(C_{1},\, \omega_{k,l})} \leq N\Big(\sup_{\partial' C_{1}} |u|+\lVert \mathcal{L}_{k,l} u\lVert_{L^{n+1}(C_{1},\,\omega_{k,l}^{-n})}\Big), 
\]
for all $k, l \in \mathbb{N}$. This particularly implies that
\begin{align}  \notag 
 \lVert D^2u\lVert_{L^{p}(C_{1},\, \beta_{k,l})} & \leq  Nb_{k,l}^{-\frac{1}{p}} \Big(\sup_{\partial' C_{1}} |u|+ b_{k,l}^{-\frac{n}{n+1}} \lVert \mathcal{L}_{k,l} u\lVert_{L^{n+1}(C_{1},\,\beta_{k,l}^{-n})}\Big)  \\ \label{eqn-L-k-12-25}
& \leq  Nb_{k,l}^{-\frac{1}{p}} \Big(\sup_{\partial' C_{1}} |u|+ b_{k,l}^{-\frac{n}{n+1}} \lVert \overline{\mathcal{L}}_{k,l} u\lVert_{L^{n+1}(C_{1},\,\beta_{k,l}^{-n})}\Big)  \\ \notag
& \quad + N b_{k,l}^{-\frac{1}{p}- \frac{n}{n+1}} |b_{k,l} -1| \|a_{ij} D_{ij} u\|_{L^{n+1}(C_1,\, \beta_{k,l})}
\end{align}
for all $k, l \in \mathbb{N}$, where
\[
\overline{\mathcal{L}}_{k,l} \phi (x, t) = \phi_t - \beta_{k,l}(x) a_{ij}(x,t) D_{ij} \phi.
\]
Recall that $\omega_{k, l}(x) = b_{k,l} \beta_{k,l}(x)$. Using \eqref{cylinder measure}, we infer that
\begin{align*}
\|a_{ij} D_{ij} u\|_{L^{n+1}(C_1,\, \beta_{k,l})}
&\leq N(n, \nu)\|D_{ij}u\|_{L^{\infty}(C_1)}b_{k,l}^{-\frac{1}{n+1}}[\omega_{k,l}(C_1)]^{\frac{1}{n+1}}\\
&\leq N(n, \nu, K_0)\|D_{ij}u\|_{L^{\infty}(C_1)}b_{k,l}^{-\frac{1}{n+1}}.
\end{align*}
Note that $b_{k,l} \rightarrow 1$ as $k, l \rightarrow \infty$. Thus,
\begin{align*}
&b_{k,l}^{-\frac{1}{p} - \frac{n}{n+1}}  |b_{k,l} -1| \|a_{ij} D_{ij} u\|_{L^{n+1}(C_1,\, \beta_{k,l})}\\
&\leq N(n, \nu, K_0) \|D_{ij}u\|_{L^\infty(C_1)}  b_{k,l}^{-\frac{1}{p}-1} |b_{k,l}-1| \rightarrow 0.
\end{align*}
Due to this and as $u \in \mathcal{W}^{2,1}_{n+1}(C_{1}, \omega) \cap C(\overline{C}_{1})$ and by the Lebesgue monotone convergence theorem, we can pass the limit as $k\rightarrow \infty$ in \eqref{eqn-L-k-12-25}, and then similarly send $l \rightarrow \infty$, to obtain \eqref{Lin esti} for $u \in C^{2,1}(\overline{C}_{1,\omega})$. 

\smallskip
\noindent
{\bf Step 4}. We remove the smoothness assumption on $u$ that we used in {\bf Step 3} and prove \eqref{Lin esti} for $u \in \mathcal{W}^{2,1}_{n+1, \textup{loc}}(C_{1,\omega}, \omega) \cap C(\overline{C}_{1, \omega})$.   Without loss of generality, we assume that $C_{1,\omega} = C_1$. Note that it is sufficient to prove \eqref{Lin esti} under the assumption that
\begin{equation} \label{Lu-finite-norm}
\|\mathcal{L} u\|_{L^{n+1}(C_1, \omega^{-n})} < \infty
\end{equation}
as \eqref{Lin esti} is trivial otherwise.  In addition, we observe that to prove \eqref{Lin esti}, we can first prove an estimate  similar to \eqref{Lin esti} in which $C_{1,\omega}$ is replaced by $B_{r} \times (-r^2(\omega^{-n})_{B_r}^{1/n}, r-1)$ with $r \in (1/2, 1)$ and sufficiently close to $1$, and then use \eqref{Lu-finite-norm} and the assumption that $u \in C(\overline{C}_{1})$ to pass the limit $r\rightarrow 1^-$ to obtain \eqref{Lin esti}. Hence, without loss of generality, we can assume that $u$ is defined in a domain that is slightly larger than $C_1$, and $ u \in \mathcal{W}^{2,1}_{n+1}(C_{1}, \omega) \cap C(\overline{C}_{1})$.  

\smallskip
Now, let $\varphi \in C_c^\infty(B_1)$ be a radially non-increasing function satisfying
\[
0 \leq \varphi \leq 1 \quad \text{and} \quad \int_{\mathbb{R}^n} \varphi(x) dx =1.
\]
Similarly, let $\bar{\varphi} \in C_{c}(-1, 1)$ be a standard mollifier satisfying
\[
0 \leq \bar{\varphi} \leq 1 \quad \text{and} \quad \int_{\mathbb{R}} \bar{\varphi}(s) ds =1.
\]
Then, we take $\phi_\epsilon(x,t) = \epsilon^{-(n+2)}\varphi (\epsilon^{-1} x) \bar{\varphi}(\epsilon^{-2}t)$ for $(x,t) \in \mathbb{R}^n \times \mathbb{R}$ and $\epsilon>0$.

\smallskip
Next, let us define 
\[
u_m(x,t) = u * \phi_{\frac{1}{m}}(x, t), \quad (x,t) \in \overline{C}_1, \quad m \in \mathbb{N}.
\]
We can see that $u_m \in C^{2,1}(\overline{C}_1)$ and
\begin{equation} \label{m-m-2026-6}
u_m(x,t) \rightarrow u (x,t) \quad \text{uniformly in} \quad  \overline{C}_1.
\end{equation}
Moreover, as $m \rightarrow \infty$
\[
\partial_t u_m(x,t) \rightarrow \partial_t u(x,t) \quad \text{and} \quad D_{ij} u_m(x,t) \rightarrow D_{ij} u(x,t) \quad \text{for a.e} \ (x,t) \in C_1.
\]
As $\varphi$ is radially non-increasing, we note that
\[
|\partial_t u_m(x,t)| \leq \mathcal{M} u_t(\cdot, t)(x) \quad \text{and} \quad |D_{ij} u_m(x,t)| \leq \mathcal{M} D_{ij}u(\cdot, t)(x)
\]
for a.e.~$(x,t) \in C_1$ and for $i,j \in \{1, 2,\ldots, n\}$, where $\mathcal{M}$ is the uncentered Hardy-Littlewood maximal function. Because $\omega^{-n} \in A_{n+1}$ by \eqref{omega-n-wei-12-22}, it follows from \cite{Muckenhoupt} that $\mathcal{M}: L^{n+1}(\mathbb{R}^n, \omega^{-n}) \rightarrow L^{n+1}(\mathbb{R}^n, \omega^{-n})$ is bounded; see also \cite[Definition 2.1.3, p.~87; Theorem 7.1.9, p.~507]{Grafakos}, for example. Therefore, by the Lebesgue dominated convergence theorem, we see that
\[
\partial_t u_m \rightarrow u_t  \quad \text{and} \quad \omega(x) a_{ij}(x,t) D_{ij} u_m \rightarrow \omega(x) a_{ij}(x,t)D_{ij}u \quad \text{in} \quad L^{n+1}(C_1, \omega^{-n}),
\]
as $m \rightarrow \infty$. Due to this, we see that $u_m \in \mathcal{W}^{2,1}_{n+1}(C_1, \omega) \cap C^{2,1}(\overline{C}_1)$ and
\begin{equation} \label{last-est-convo-26}
\lim_{m \rightarrow \infty} \|\mathcal{L} u_m\|_{L^{n+1}(C_1, \,\omega^{-n})} =  \|\mathcal{L} u\|_{L^{n+1}(C_1,\, \omega^{-n})}.
\end{equation}
Now, applying the estimate proved in {\bf Step 3}, we obtain
\begin{align*}
\lVert D^2u_m\lVert_{L^{p}(C_{1},\, \omega)}  & \leq N\Big(\sup_{\partial' C_{1}} |u_m|+ \lVert \mathcal{L} u_m \lVert_{L^{n+1}(C_{1},\,\omega^{-n})} \Big), \quad \forall \ m \in \mathbb{N},
\end{align*}
where $N = N(n, \nu, K_0, p)>0$. From \eqref{m-m-2026-6}, \eqref{last-est-convo-26}, and by taking $m\rightarrow \infty$ in the last estimate, we obtain \eqref{Lin esti}. The proof is then completed.
\end{proof}

\medskip
\appendix
\section{Proof of Remark \ref{remark-example}} \label{proof-weigh-example} 
\begin{proof} As explained in the statement, part (a) is proved in \cite{CMP, Cho-Fang-Phan}, and part (c) follows part (b) by a localization argument and a covering lemma. Hence, it remains to prove \eqref{BMO-varphi-example}.  We claim that there exists a constant $N = N(n)>0$ such that
\begin{equation} \label{claim-BMO-we-expli}
\frac{1}{\varphi(B_r(x_0))} \int_{B_r(x_0)} |\varphi(x) - (\varphi)_{B_r(x_0)}|\, dx \leq \frac{N}{|\ln(4r_0)|},
\end{equation}
for all $B_r(x_0) \subset B_{r_0}$. We verify \eqref{claim-BMO-we-expli} by considering the following cases.

\smallskip \noindent
{\bf Case 1: $x_0 =0$.} Recall that $\varphi(x)=-\ln |x|$ for $x\in B_{1/e}$. By polar coordinates, we have
\begin{align*}
\varphi(B_r)= - n \sigma_n \int_0^r s^{n-1}\ln(s)\, ds
=\sigma_n \big( r^n|\ln r|+r^n/n \big)\ge  \sigma_n r^n |\ln r|,
\end{align*}
where $\sigma_n=|B_1|$.
Also, by a direct computation, we infer that
\begin{align*}
& \int_{B_r} |\varphi(x) - (\varphi)_{B_r}|\, dx  \leq  \int_{B_r} \fint_{B_r}\big |\ln|x| - \ln|y| \big|\, dydx\\
& = \frac{n^2\sigma_n}{r^n} \int_0^r \int_0^r s^{n-1} \tau^{n-1} |\ln (\tau/s)|\, d\tau  ds=\frac{2n^2\sigma_n}{r^n} \int_0^r \int_0^s s^{n-1} \tau^{n-1} |\ln (\tau/s)|\, d\tau  ds\\
& =\frac{2n^2\sigma_n}{r^n} \int_0^r s^{2n-1} \Big[\int_0^{1}  t^{n-1} |\ln t|\, dt \Big] ds =  \frac{2\sigma_n}{r^n}\int_0^r s^{2n-1}\, ds= \frac{\sigma_n r^n}{n},
\end{align*}
where we used the identity $\int_0^1 t^{n-1}|\ln t|\, dt=\frac{1}{n^2}$.
Combining the last two estimates yields
\[
\frac{1}{\varphi(B_r)} \int_{B_r} |\varphi(x) - (\varphi)_{B_r}|\, dx \leq \frac{1}{n|\ln r|}.
\]
Hence, \eqref{claim-BMO-we-expli} holds since $|\ln r| \geq |\ln (4r_0)|$ with $r \in (0, r_0)$.

\smallskip \noindent
{\bf Case 2: $|x_0| \leq 3r$.} In this case, we note that
\[
B_{r}(x_0) \subset B_{4r} \subset B_{4r_0}.
\]
Also, by the triangle inequality, we note that
\[
\fint_{B_{r}(x_0)} |\varphi(x) - (\varphi)_{B_r(x_0)}|\, dx \leq 2 \fint_{B_r(x_0)}| \varphi(x) - (\varphi)_{B_{4r}}|\, dx.
\]
Then, we have
\begin{align*}
& \frac{1}{\varphi(B_r(x_0))} \int_{B_{r}(x_0)} |\varphi(x) - (\varphi)_{B_{r}(x_0)}|\, dx \leq \frac{2}{\varphi(B_{r}(x_0))} \int_{B_{r}(x_0)} |\varphi(x) - (\varphi)_{B_{4r}}|\, dx \\
& \leq \frac{N(n)}{\varphi(B_{4r})} \int_{B_{4r}} |\varphi(x) - (\varphi)_{B_{4r}}|\, dx  \leq \frac{N(n)}{|\ln (4r_0)|},
\end{align*}
where in the last step we used {\bf Case 1} and the doubling property that
\[
\varphi(B_r(x_0)) \geq \gamma \varphi(B_{4r}) \quad \text{with} \quad \gamma  = \gamma(n)>0.
\]

\smallskip \noindent
{\bf Case 3: $|x_0| > 3r$.} In this case, we have
\[
|x|\ge |x_0| - |x -x_0|\ge 2r, \quad \forall \ x \in B_r(x_0).
\]
From this, and note also that $B_r(x_0)\subset B_{r_0}$, then
\begin{equation} \label{ball-type-2}
 2r \leq  |x| \leq r_0 , \quad \forall \ x \in B_r(x_0).
\end{equation}
Therefore,
\[
-\ln(r_0) \leq -\ln|x| \leq - \ln (2r), \quad \forall \    x \in B_r(x_0)
\]
Then,
\[
\varphi(B_r(x_0))=\int_{B_r(x_0)} \varphi(x)\, dx\geq 
\sigma_n r^n |\ln (r_0)|.
\]
From this, we see that
\begin{align*}
& \frac{1}{\varphi(B_r(x_0))} \int_{B_r(x_0)} |\varphi(x) - (\varphi)_{B_r(x_0)}|\, dx \\
& \leq \frac{1}{\sigma_n r^n |\ln(r_0)|} \int_{B_{r}(x_0)} \Big |-\ln|x|  + \fint_{B_r(x_0)} \ln |y|\,  dy \Big|\, dx \\
& \leq \frac{1}{|\ln(4r_0)|} \fint_{B_{r}(x_0)}  \fint_{B_{r}(x_0)} \big|  \ln |x| - \ln|y| \big|\, dy dx.
\end{align*}
By \eqref{ball-type-2} and the mean value theorem, we note that
\[
\big|\ln|x| - \ln|y|\big|  \leq 1, \quad \forall \ x ,\, y \in B_r(x_0).
\] 
Then
\[
 \frac{1}{\varphi(B_r(x_0))} \int_{B_r(x_0)} |\varphi - (\varphi)_{B_r(x_0)}| dx \leq \frac{1}{|\ln(4r_0)|}.
 \]
The proof is then completed.
\end{proof}

\section{Proof of Lemma \ref{covering-1}}\label{Appendix-A}
\begin{proof} We split the proof into three steps.

\smallskip
\noindent
\textbf{Step 1}. We begin by proving the first assertion in \eqref{conclusion covering}, which claims that
\[
\omega(\Gamma\setminus \tilde E) =0.
\] 
Due to $\omega\in A_{1+\frac{1}{n}}$, we have $\omega(x)>0$ for almost every $x\in \R^n$. Thus, for each Lebesgue point $X=(x,t)$ of $\Gamma$, we have the sequence of cylinders $\{C_{2^{-i},\, \omega}(X_i)\}_{i\in \mathbb{N}}$ with $X_i=(x,\, t+\frac{1}{2}4^{-i}(\omega^{-n})_{B_{2^{-i}}(x)}^{1/n})\in \mathbb{R}^n\times \mathbb{R}$ such that
\[
\lim_{i\rightarrow \infty}\frac{\omega\big(C_{2^{-i},\, \omega}(X_i)\cap \Gamma\big)}{\omega\big(C_{2^{-i},\, \omega}(X_i)\big)}=1,
\]
which implies that there exists $i_0\in \N$ large such that 
\[
X\in C_{2^{-i_0},\, \omega}(X_{i_0})\in \A, \quad \text{then}\quad X\in \tilde C_{2^{-i_0},\, \omega}(X_{i_0})\subset \tilde E.
\]
Therefore, $\Gamma\subset \tilde E$ except for a set of zero measure, then the first assertion in \eqref{conclusion covering} is proved.

\smallskip
\noindent
\textbf{Step 2}. We prove the second assertion in \eqref{conclusion covering}. Note that this assertion is trivial when $\omega(\Gamma)=0$. Hence, from now on, we assume that $\omega(\Gamma)>0$.  We define
\begin{equation}\label{redef of A}
\A_0= \left\{C = C_{r,\, \omega}(Y):\ \omega(C \cap \Gamma) = q \omega(C) \right\}.
\end{equation}
It is clear that $\mathcal{A}_0 \subset \A$. Moreover, for a fixed $C_{r,\, \omega}(Y)\in \A$ such that
\[
\omega(C_{r,\, \omega}(Y) \cap \Gamma) \geq q \omega(C_{r,\, \omega}(Y)),
\]
let us define $C_\theta = C_{\theta r,\, \omega}(Y)$, and
\[
g(\theta) = \frac{1}{\omega(C_{\theta})} \int_{C_{\theta}  \cap \Gamma} \omega(x)\, dx dt.
\] 
It follows that $g: [1, \infty) \rightarrow (0, \infty)$ is continuous. Besides,   
\[
 g(1) \geq q \quad \text{and}  \quad \lim_{\theta \rightarrow \infty}g(\theta) =0,
\]
where the latter is because $\Gamma$ is bounded and therefore it has a finite $\omega$-measure. 
Hence, we can find $\theta_0 \geq 1$ so that 
$g(\theta_0) = q$ by the continuity of $g$. 
This implies that $C_{\theta_0} \in \mathcal{A}_0$. Also, as $C_{r,\, \omega}(Y) \subset C_{\theta_0}$, we conclude that  
\[
\tilde E = \displaystyle{\bigcup_{C \in \mathcal{A}}} \tilde C = \displaystyle{\bigcup_{C \in \mathcal{A}_0}} \tilde C.
\]

Now, we construct a sequence of pairwise disjoint cylinders $\{C^i\}_{i=0}^{\infty}\subset \mathcal{A}_0$ as in the proof of the classical Vitali covering lemma. 
First, note that $\Gamma$ is bounded, we see that 
the set $\{r>0: C_{r,\, \omega}(Y)\in \A_0\}$ is also bounded. 
Let us define 
\[
R_0=\sup\{r>0: C_{r,\, \omega}(Y)\in \A_0\}<\infty.
\]
Then, by a compactness argument, we can find a cylinder $C^0$ such that 
\[ 
C^0 = C_{R_0,\,\omega}(Y^0)  \in   \A_0. 
\]
Next, for each $i \in \N$, using a similar argument, we can define $\A_i,\, R_i,\, C^i$ inductively by
\[
\A_i=\Big\{C=C_{r,\, \omega}(Y)\in \A_0: C\cap C^k=\emptyset, \ \forall \ k =0,1,\ldots ,i-1\Big\},
\]
\[
R_i =\sup\{r>0 : \ C_{r,\, \omega}(Y)\in \A_i\},
\] 
and 
\[
C^i =C_{R_i,\,\omega}(Y^i)  \in   \A_i\quad\text{with}\quad Y^{i}=(y^i, s^i).
\]
If the set $\A_i$ is not empty, we continue this process. If the set $\A_i$ is empty, then stop the process and set $R_{k}=0$ for all $k\geq i$. From the construction, we see that
\[
\tilde \A= \A_0\supset\A_1\supset\A_2\supset \ldots, \quad \text{and}\quad R_0\geq R_1\geq R_2\geq R_3\geq\ldots.
\]
Moreover, we claim that 
\begin{equation}  \label{R-k-limit}
\lim_{i \rightarrow \infty} R_i =0.
\end{equation}
To see this, suppose that \eqref{R-k-limit} does not hold. Then, as $\{R_i\}_{i}$ decreases, there is $r_0>0$ such that
\[ 
\lim_{i\rightarrow \infty} R_i =r_0, \quad \text{and} \quad R_i \geq r_0, \quad \forall \, i \in \mathbb{N}.
\]
From \eqref{cylinder measure}, we notice that
\[
\omega(C^i)=\omega(C_{R_i,\,\omega}(Y^i))\geq N(n)R_i^{n+2}\geq N(n)r_0^{n+2}>0, \quad \forall\, i\in \mathbb{N}.
\]
Thus, by this and the disjoint property of $\{C^i\}_{i\in \mathbb{N}}$,
\[
\begin{split}
\omega(\Gamma) &\geq \omega\big(\Gamma\cap(\cup^{\infty}_{i=0}C^i)\big)=\sum^{\infty}_{i=0}\omega\big(\Gamma\cap C^i\big) = q\sum^{\infty}_{i=0}\omega(C^i) \\
& \geq N(n)q \sum^{\infty}_{i=0}r_0^{n+2} =\infty.
\end{split}
\]
However, this contradicts $\omega(\Gamma)<\infty$ as $\Gamma$ is bounded. The claim \eqref{R-k-limit} is proved.

\smallskip
For each $C_{r,\, \omega}(Y) \in \A$, by \eqref{R-k-limit} and the construction of  $\{R_i\}_{i\in \N}$, we can find a unique $i_0 \in \mathbb{N}$ such that
\begin{equation} \label{C-k-zero}
R_0\geq R_1\geq
\ldots \geq R_{i_0} \geq r> R_{i_0+1}\geq \ldots.
\end{equation}
From this and the construction of $\{C^i\}_{i\in \mathbb{N}}$, we infer that there exists some $k_0\in\{0,1,2,\ldots, i_0\}$ such that
\begin{equation} \label{i-zero-C}
C_{r,\, \omega}(Y)\cap C^{k_0} \not=\emptyset.
\end{equation}
Next, for each $i\in \N$, let
\[
{Q}^{i}=B_{3R_{i}}(y^{i})\times \big(s^i-6R_i^2(\omega^{-n})^{1/n}_{B_{3R_{i}}(y^{i})},\, s^i+3R_i^2(\omega^{-n})^{1/n}_{B_{3R_{i}}(y^{i})}\big).
\]
As such, by a direct computation,
\begin{equation}\label{measure comparison}
\begin{aligned}
\omega(Q^i)=\sigma_n(3R_i)^{n+2}(\omega)_{B_{3R_i}(y^i)}(\omega^{-n})_{B_{3R_i}(y^i)}^{\frac{1}{n}}
&\leq \sigma_n(3R_i)^{n+2}K_0\\
&\leq 3^{n+2}K_0\omega(C^{i}),
\end{aligned}
\end{equation}
where $\sigma_n=|B_1|$, and we used \eqref{cylinder measure} for $C^i$ in the last inequality. Moreover, because of \eqref{C-k-zero} and \eqref{i-zero-C}, we see that  $C_{r,\, \omega}(Y) \subset Q^{k_0}$. From this, the first assertion in \eqref{conclusion covering}, and by ignoring the set of zero $\omega$-measure, we have
\begin{equation} \label{Gamma-comp-Ci}
\Gamma\subset \bigcup_{C\in \A}C\subset \bigcup^{\infty}_{i=0}Q^i,\quad \text{so that}\quad \sum^{\infty}_{i=0}\omega(Q^i) \geq \omega(\Gamma).
\end{equation}

\smallskip
On the other hand, since $\omega\in A_{1+\frac{1}{n}}$ with $[\omega]_{A_{1+\frac{1}{n}}}\leq K_0$, by the reverse H\"{o}lder property of $\omega$ (see, \cite[Proposition 7.2.8,~P. 521]{Grafakos}, for example), there exist positive constants $N=N(n,K_0)$ and $\lambda_0=\lambda_0(n,K_0)$ such that
\[
\omega\big(B_r(x)\setminus B_{\eta r}(x)\big)\leq N(1-\eta^n)^{\lambda_0}\omega\big(B_r(x)\big), \quad \forall \ B_r(x)\subset \R^n,
\]
where $\eta\in (0,1)$. Thus,
\begin{equation*}\label{small diff}
\omega\big(C^i\setminus \tilde C^i\big)\leq N(1-\eta^n)^{\lambda_0}\omega\big(C^i\big), \quad \forall \ i\in \N.
\end{equation*}
This, together with \eqref{redef of A} and the disjointness of $\{C^i\}_{i\in \N}$, implies that
\begin{equation}\label{E-Gamma}
\begin{aligned}
\omega(\tilde E\setminus \Gamma)
&\geq\sum^{\infty}_{i=0}\left[\omega(C^i\setminus \Gamma)-\omega(C^i\setminus \tilde C^i)\right]\\
&= \left[1-q-N(1-\eta^n)^{\lambda_0}\right]\sum^{\infty}_{i=0}\omega(C^i).
\end{aligned}
\end{equation}

\smallskip
Therefore, due to the first assertion of \eqref{conclusion covering}, it follows that
\begin{align*}
\omega(\tilde E) &=\omega(\Gamma)+ \omega(\tilde E\setminus \Gamma)\\ 
&\geq \omega(\Gamma)+\left[1-q-N(1-\eta^n)^{\lambda_0}\right]\sum^{\infty}_{i=0}\omega(C^i) &&\text{by \eqref{E-Gamma}}\\
&\geq \omega(\Gamma)+3^{-n-2}K_0^{-1}\left[1-q-N(1-\eta^n)^{\lambda_0}\right]\sum^{\infty}_{i=0}\omega(Q^i) && \text{by \eqref{measure comparison}}\\
&=\left\{1+3^{-n-2}K_0^{-1}\left[1-q-N(1-\eta^n)^{\lambda_0}\right]\right\}\omega(\Gamma) && \text{by \eqref{Gamma-comp-Ci}},
\end{align*}
where $N=N(n, K_0)>0$, and $\lambda_0=\lambda_0(n, K_0)>0$. Then, the second assertion in \eqref{conclusion covering} is proved.

\smallskip
\noindent
\textbf{Step 3}. Lastly, we prove the third assertion in \eqref{conclusion covering}. 
For each $x \in \mathbb{R}^n$, we write
\[
\tilde E(x)=\big\{ t\in \mathbb{R}: (x,t) \in \tilde E \big\} \quad \text{and}\quad \hat{E}(x)=\big\{t \in \mathbb{R}: (x,t) \in \hat{E} \big\}.
\]
By Fubini's theorem, we obtain
\[
\omega(\tilde E) = \int_{\mathbb{R}^n} |\tilde E(x)| \omega(x)\, dx \quad \text{and} \quad \omega(\hat{E}) = \int_{\mathbb{R}^n} |\hat{E}(x)| \omega(x)\, dx,
\]
where $|\tilde E(x)|$ and $|\hat{E}(x)|$ denote the Lebesgue measure on $\mathbb{R}$ of $\tilde E(x)$ and $\hat{E}(x)$, respectively.
As $\omega(x) \geq 0$, it suffices to show that
\begin{equation}\label{time measure}
 |\hat{E}(x)| \geq \xi_1 |\tilde E(x)| \quad \text{for all}\ x \in \mathbb{R}^n, \text{ where } \xi_1 = (l-1)(l+1)^{-1}.
\end{equation}
For any fixed $x$ with $\tilde E(x)$ nonempty, we have $\hat E(x)$ is nonempty and open. Thus, we may assume $\hat E(x)$ is a finite or countable union of disjoint open intervals, i.e.,
\[
\hat E(x)=\cup_{i\in \mathcal{I}}I_i(x),
\]
where $\mathcal {I}$ is an index set. Then, for any $t\in \tilde E(x)$, we have $(x,t)\in C_{\tilde r,\, \omega}(\tilde Y)\in \A$ with some $\tilde r>0$ and $\tilde Y=(\tilde y, \tilde s)\in \R^{n}\times \R$. Due to \eqref{set E}, we see that 
\[
\big(\tilde s+{\tilde r}^2(\omega^{-n})^{1/n}_{B_{\tilde r}(\tilde y)},\, \tilde s+l{\tilde r}^2(\omega^{-n})^{1/n}_{B_{\tilde r}(\tilde y)}\big)\subset I_{i_0}(x)\quad \text{for some}\quad i_0\in \mathcal {I}.
\]
Set $r_{i_0}^2(x)=|I_{i_0}(x)|/(l-1)$, we can determine an $a_{i_0}(x)\in \R$ so that
\[
I_{i_0}(x)=(a_{i_0}+r_{i_0}^2, a_{i_0}+lr_{i_0}^2).
\]
Also, note that $r_{i_0}^2\geq {\tilde r}^2(\omega^{-n})^{1/n}_{B_{\tilde r}(\tilde y)}$, it follows from the last two formulas that
\begin{equation}\label{pull back}
a_{i_0}-r_{i_0}^2\leq \tilde s-{\tilde r}^2(\omega^{-n})^{1/n}_{B_{\tilde r}(\tilde y)}\leq \tilde s\leq a_{i_0}+lr_{i_0}^2.
\end{equation}
For convenience, we set $J_{i}(x)=(a_{i}-r_{i}^2,a_{i}+lr_{i}^2)$ for all $i\in \mathcal {I}$.  By \eqref{pull back}, we have
$\tilde E(x)\subset \bigcup_{i\in \mathcal {I}}J_i(x)$, and 
\[
|J_i(x)|=(l+1)r_i^2=\frac{l+1}{l-1}|I_i(x)|,\quad \forall\,  i\in \mathcal {I}.
\]
 Thus, for all $x\in \R^n$ with $E(x)$ nonempty, we have 
\begin{equation*}
|\hat E(x)|=\sum_{i\in \mathcal {I}}|I_i(x)|= \sum_{i\in \mathcal {I}}\frac{l-1}{l+1}|J_i(x)|\geq \frac{l-1}{l+1}|\tilde E(x)|.
\end{equation*}
Therefore, we proved \eqref{time measure}. 
\end{proof}

\end{document}